\pdfoutput=1
\RequirePackage{ifpdf}
\ifpdf 
\documentclass[pdftex]{sigma}
\else
\documentclass{sigma}
\fi

\numberwithin{equation}{section}

\newtheorem{Theorem}{Theorem}[section]
\newtheorem{Corollary}[Theorem]{Corollary}
\newtheorem{Lemma}[Theorem]{Lemma}
\newtheorem{Proposition}[Theorem]{Proposition}
\newtheorem{Conjecture}[Theorem]{Conjecture}
 { \theoremstyle{definition}
\newtheorem{Definition}[Theorem]{Definition}
\newtheorem{Example}[Theorem]{Example}
\newtheorem{Remark}[Theorem]{Remark} }

\usepackage{tikz}
\usetikzlibrary{positioning}
\usetikzlibrary{cd}
\usetikzlibrary{decorations.pathreplacing}
\usetikzlibrary{shapes.geometric}

\DeclareMathOperator{\id}{id}
\DeclareMathOperator{\diag}{diag}
\DeclareMathOperator{\image}{Im}

\DeclareMathOperator{\tr}{tr}

\newcommand*{\C}{\mathbb{C}}
\newcommand*{\Q}{\mathbb{Q}}
\newcommand*{\R}{\mathbb{R}}

\newcommand*{\Z}{\mathbb{Z}}

\newcommand*{\maxzero}[1]{[#1]_{+}}
\newcommand*{\level}{\ell}
\newcommand*{\dcn}{ h^{\vee}}
\newcommand*{\rf}{\mathcal{F}}
\newcommand*{\univsf}{\mathcal{F}_{\mathrm{sf}}}

\newcommand*{\innerproduct}[2]{( #1 \,|\, #2 )}
\newcommand{\empvertex}{\tikz \draw (0,0) circle [radius=2pt];}
\newcommand{\fillvertex}{\tikz \fill (0,0) circle [radius=2pt];}
\newcommand{\tetvertex}{
 \begin{tikzpicture}
 \draw (0,0.4) rectangle (0.2,0.6);
 \end{tikzpicture}
}
\newcommand{\edgevertex}{
 \begin{tikzpicture}
 \draw [fill] (0,0.5) circle [radius=0.1] ;
 \end{tikzpicture}
}
\newcommand*{\network}{\mathcal{N}}
\newcommand{\defnp}{
 \begin{tikzpicture}
 \draw (0,0.8) node{$(i,s)$};
 \draw [fill] (0,0.5) circle [radius=0.1] ;
 \draw [dashed] (0.2, 0.5) -- (1.2,0.5);
 \draw (1.4,0.8) node{$t$};
 \draw (1.3,0.4) rectangle (1.5,0.6);
 \end{tikzpicture}
}
\newcommand{\defn}{
 \begin{tikzpicture}
 \draw (0,0.8) node{$(i,s)$};
 \draw [fill] (0,0.5) circle [radius=0.1] ;
 \draw (0.7,0.7) node{$a$};
 \draw [<-] (0.2, 0.5) -- (1.2,0.5);
 \draw (1.4,0.8) node{$t$};
 \draw (1.3,0.4) rectangle (1.5,0.6);
 \end{tikzpicture}
}
\newcommand{\defnpp}{
 \begin{tikzpicture}
 \draw (0,0.8) node{$(i,s)$};
 \draw [fill] (0,0.5) circle [radius=0.1] ;
 \draw (0.7,0.7) node{$a$};
 \draw [->] (0.2, 0.5) -- (1.2,0.5);
 \draw (1.4,0.8) node{$t$};
 \draw (1.3,0.4) rectangle (1.5,0.6);
 \end{tikzpicture}
}
\newcommand{\defNp}{
 \begin{tikzpicture}
 \draw (0,0.8) node{$e$};
 \draw [fill] (0,0.5) circle [radius=0.1] ;
 \draw [dashed] (0.2, 0.5) -- (1.2,0.5);
 \draw (1.4,0.8) node{$t$};
 \draw (1.3,0.4) rectangle (1.5,0.6);
 \end{tikzpicture}
}
\newcommand{\defN}{
 \begin{tikzpicture}
 \draw (0,0.8) node{$e$};
 \draw [fill] (0,0.5) circle [radius=0.1] ;
 \draw [<-] (0.2, 0.5) -- (1.2,0.5);
 \draw (1.4,0.8) node{$t$};
 \draw (1.3,0.4) rectangle (1.5,0.6);
 \end{tikzpicture}
}
\newcommand{\defNpp}{
 \begin{tikzpicture}
 \draw (0,0.8) node{$e$};
 \draw [fill] (0,0.5) circle [radius=0.1] ;
 \draw [->] (0.2, 0.5) -- (1.2,0.5);
 \draw (1.4,0.8) node{$t$};
 \draw (1.3,0.4) rectangle (1.5,0.6);
 \end{tikzpicture}
}
\newcommand*{\al}[1][m]{\alpha_{#1}^{(a)} (x)}
\newcommand*{\be}[1][m]{\beta_{#1}^{(a)} (x)}
\newcommand*{\PP}[1][m]{P_{#1}^{(a)} (x)}
\newcommand*{\sinn}[1][m]{( x^m - x^{-m} )}

\begin{document}

\allowdisplaybreaks

\newcommand{\arXivNumber}{1812.05863}

\renewcommand{\thefootnote}{}

\renewcommand{\PaperNumber}{028}

\FirstPageHeading

\ShortArticleName{Exponents Associated with $Y$-Systems and their Relationship with $q$-Series}

\ArticleName{Exponents Associated with $\boldsymbol{Y}$-Systems\\ and their Relationship with $\boldsymbol{q}$-Series\footnote{This paper is a~contribution to the Special Issue on Cluster Algebras. The full collection is available at \href{https://www.emis.de/journals/SIGMA/cluster-algebras.html}{https://www.emis.de/journals/SIGMA/cluster-algebras.html}}}

\Author{Yuma MIZUNO}

\AuthorNameForHeading{Y.~Mizuno}

\Address{Department of Mathematical and Computing Science, Tokyo Institute of Technology,\\ 2-12-1 Ookayama, Meguro-ku, Tokyo 152-8550, Japan}
\Email{\href{mailto:mizuno.y.aj@m.titech.ac.jp}{mizuno.y.aj@m.titech.ac.jp}}

\ArticleDates{Received September 27, 2019, in final form April 02, 2020; Published online April 18, 2020}

\Abstract{Let $X_r$ be a finite type Dynkin diagram, and $\ell$ be a positive integer greater than or equal to two. The $Y$-system of type $X_r$ with level $\ell$ is a system of algebraic relations, whose solutions have been proved to have periodicity. For any pair $(X_r, \ell)$, we define an integer sequence called exponents using formulation of the $Y$-system by cluster algebras. We give a conjectural formula expressing the exponents by the root system of type $X_r$, and prove this conjecture for $(A_1,\ell)$ and $(A_r, 2)$ cases. We point out that a specialization of this conjecture gives a relationship between the exponents and the asymptotic dimension of an integrable highest weight module of an affine Lie algebra. We also give a point of view from $q$-series identities for this relationship.}

\Keywords{cluster algebras; $Y$-systems; root systems; $q$-series}

\Classification{13F60; 17B22; 81R10}

\renewcommand{\thefootnote}{\arabic{footnote}}
\setcounter{footnote}{0}

\tableofcontents
\section{Introduction}
In the study of thermodynamic Bethe ansatz, Zamolodchikov introduced a system of algebraic relations called $Y$-system for any finite type simply laced Dynkin diagram and conjectured that solutions of the $Y$-system have periodicity. This periodicity conjecture was proved by Fomin and Zelevinsky~\cite{FZ_Ysystem}. Their proof was implicitly based on the theory of cluster algebras that they introduced in~\cite{FZ1}.
In fact, in~\cite{FZ4} they refined the result on the periodicity using the concept in cluster algebras called $Y$-seed mutations.

More generally, for any finite type Dynkin diagram $X_r$ ($=A_r$, $B_r$, $C_r$, $D_r$, $E_{6,7,8}$, $F_4$ or $G_2$) and integer $\level \geq 2$ called level,
the $Y$-system associated with the pair $(X_r, \level)$ is defined in~\cite{KunibaNakanishi92}.
When $X$ is $ADE$ and the level is $2$, this is the $Y$-system introduced by Zamolodchikov.
The periodicities for these $Y$-systems were also proved in~\cite{IIKKNa,IIKKNb, Keller}.
Their proof used the quiver representation theory, in particular the algebraic objects called the cluster categories.

Although these periodicities are themselves very interesting,
it is also interesting that they relate to conformal field theories through the Rogers dilogarithm function.
The sum of special values of the Rogers dilogarithm function associated with the $Y$-system gives the central charge of a conformal field theory~\cite{IIKKNa,IIKKNb, Nak}.
This is called the dilogarithm identities in conformal field theories.
Note that the central charge also appears as the exponential growth of a character in the representation theory of affine Lie algebras~\cite{KacPet}.
For the proof of these identities, the periodicity of the $Y$-system mentioned earlier plays an essential role.

The purpose of this paper is to introduce a sequence of integers associated with the $Y$-system that we call \emph{exponents} and to reveal interesting features of it.
In particular, we give a conjectural formula on the exponents that gives a new connection between the theory of cluster algebras and the representational theory of affine Lie algebras, and give proofs for some cases.

In order to define the exponents, we first review roughly formulation of the $Y$-systems by cluster algebras.
First, we define a quiver $Q(X_r, \level)$ for a pair $(X_r, \level)$.
Then, we can obtain the $Y$-system as algebraic relations between rational functions obtained by repeating mutations, which are certain operations defined on quivers.
In particular, we can paraphrase the periodicity of the $Y$-system associated with a pair $(X_r, \level)$ to the periodicity of the cluster transformation $\mu$, which is a rational function obtained from a mutation sequence on $Q(X_r,\level)$.
Let $\mathbf{I}$ be the set of vertices in $Q(X_r, \level)$.
Then the cluster transformation $\mu$ is actually an $\mathbf{I}$-tuple of single rational functions in the variables $y:= (y_i \,|\, i \in \mathbf{I})$,
that is, $\mu \in \Q (y_i \,|\, i \in \mathbf{I})^{\mathbf{I}}$.
Keller~\cite{Keller} (for $X=ADE$) and Inoue et al.~\cite{IIKKNa,IIKKNb} (for $X=BCFG$) proved the periodicity of the $Y$-system by showing that $\mu$ has the following periodicity:
\begin{gather*}
 \underbrace{\mu \circ \mu \circ \dots \circ \mu}_{t(\level + \dcn)} (y) = y ,
\end{gather*}
where $t$ is defined by
\begin{gather*}
 t =
 \begin{cases}
 1 & \text{if $X_r = A_r,D_r,E_6 ,E_7$ or $E_8$},\\
 2 & \text{if $X_r = B_r,C_r$ or $F_4$}, \\
 3 & \text{if $X_r =G_2$},
 \end{cases}
\end{gather*}
and $\dcn$ is the dual Coxeter number of $X_r$.

Using this result, we define the exponents as follows. First, we see that there is a unique $\eta \in (\R_{>0})^{\mathbf{I}} $ that satisfies $\mu (\eta) =\eta$.
Then, let $J(y)$ be the Jacobian matrix of the rational function~$\mu$:
\begin{gather*}
J(y) = \left( \frac{\partial \mu_{i}}{\partial y_j} (y) \right)_{i,j \in \mathbf{I} }.
\end{gather*}
The periodicity of $\mu$ implies that the $t\big(\level + \dcn\big)$-th power of $J(\eta)$ is the identity matrix.
Thus all the eigenvalues of $J(\eta)$ are $t\big(\level + \dcn\big)$-th root of unities.
These eigenvalues can be written as
\begin{gather*}
{\rm e}^{ \frac{2 \pi{\rm i} m_1}{t (\level + \dcn)} } , {\rm e}^{ \frac{2 \pi{\rm i} m_2}{t (\level + \dcn)} } , \dots ,{\rm e}^{ \frac{2 \pi{\rm i} m_{\lvert \mathbf{I} \rvert}}{t (\level + \dcn)} } ,
\end{gather*}
using a sequence of integers $0 \leq m_1 \leq m_2 \leq \dots \leq m_{\lvert \mathbf{I} \rvert} < t\big(\level + \dcn\big)$. We say that this sequence $m_1 , \dots , m_{\lvert \mathbf{I} \rvert}$ is the exponents of $Q(X_r , \level)$.

We give a conjectural formula on the exponents in terms of root systems.
Let $\Delta$ be the root system of type~$X_r$ with an inner product $\innerproduct{\cdot}{\cdot}$ normalized as $\innerproduct{\alpha}{\alpha} =2$ for long roots $\alpha$.
Let $\alpha_1 , \dots , \alpha_r$ be simple roots of $\Delta$.
For any $a=1, \dots, r$, we define an integer $t_a$ by $t_a = 2/ \innerproduct{\alpha_a}{\alpha_a}$.
Let $\Delta^{{\rm long}}$ and $\Delta^{{\rm long}}$ be the set of the long roots and the short roots, respectively.
Let $\rho$ be the half of the sum of the positive roots.
Using these materials, we define two polynomials $N_{X_r,\level}(x)$ and $D_{X_r, \level}(x)$ by
\begin{gather*}
N_{X_r, \level}(x) = \prod_{a=1}^{r} \frac{ x^{t(\level + \dcn)} -1 }{ x^{t/t_a} -1 }, \qquad
D_{X_r, \level}(x)=D_{X_r, \level}^{{\rm long}} (x) D_{X_r, \level}^{{\rm short}} (x),
\end{gather*}
where the polynomials $D_{X_r, \level}^{{\rm long}}(x)$ and $D_{X_r, \level}^{{\rm short}}(x)$ are defined by
\begin{gather*}
D_{X_r, \level}^{{\rm long}}(x)=\prod_{\alpha \in \Delta^{{\rm long}} } { \big( x^t - {\rm e}^{ \frac{2 \pi{\rm i} \innerproduct{\rho}{\alpha}}{\level + \dcn} } \big) } ,
\qquad
D_{X_r, \level}^{{\rm short}} (x) = \prod_{\alpha \in \Delta^{{\rm short}}} { \big( x - {\rm e}^{\frac{2 \pi{\rm i} \innerproduct{\rho}{\alpha}}{\level + \dcn} } \big) }.
\end{gather*}

The following formula is the main conjecture of this paper.
It says that the exponents are obtained from the ratio of $N_{X_r, \level}(x)$ and $D_{X_r, \level}(x)$.
\begin{Conjecture}\label{intro conj: char poly}
 The following identity holds for the characteristic polynomial of $J (\eta)$:
 \begin{align*}
 \det (x I - J (\eta)) = \frac{N_{X_r, \level}(x)}{D_{X_r, \level}(x)}.
 \end{align*}
\end{Conjecture}

We prove the conjecture in the following cases. In these cases, we give explicit expressions for eigenvectors of $J(\eta)$ using known explicit expressions for $\eta$.
\begin{Theorem}\label{intro theorem: A1 level}
 Conjecture~{\rm \ref{intro conj: char poly}} is true in the following cases:
 \begin{enumerate}\itemsep=0pt
 \item[$1)$] $(A_1, \level)$ for all $\level \geq 2$,
 \item[$2)$] $(A_r , 2)$ for all $r\geq 1$.
 \end{enumerate}
\end{Theorem}

This conjecture gives a relationship between cluster algebras and root systems.
Even more interestingly, it relates to the representation theory of affine Lie algebras.
Let us explain this.
Let $\mathfrak{g}$ be the finite dimensional simple Lie algebra of type $X_r$,
and $\hat{\mathfrak{g}}$ be the affine Lie algebra associated with $\mathfrak{g}$.
It is a central extension of the loop algebra of $\mathfrak{g}$ together with a derivation operator $d$.
Let $L(\Lambda)$ be the integrable highest weight $\hat{\mathfrak{g}}$-module with the highest weight $\Lambda$.
The (specialized) character $\chi_\Lambda (q)$ is a $q$-series such that its coefficients are the multiplicities of the eigenvalues of the derivation operator $d$ in $L(\Lambda)$.
If we set $q={\rm e}^{2 \pi{\rm i} \tau}$, the character $\chi_\Lambda(q)$ converges to a holomorphic function
on the upper half-plane $\mathbb{H} = \{ \tau \in \C \,|\, \image \tau >0 \}$.
Let $\tau \downarrow 0$ denote the limit in the positive imaginary axis.
Kac and Peterson~\cite{KacPet} found that $\chi_\Lambda (q)$ has the following asymptotics:
 \begin{gather*}
 \lim_{\tau\downarrow 0} \chi_{\Lambda} \big({\rm e}^{2 \pi{\rm i} \tau}\big) {\rm e}^{-\frac{\pi{\rm i} c(\level)}{12 \tau}} = a(\Lambda),
\end{gather*}
for some rational number $c$ and real number $a(\Lambda)$.
The real number $a(\Lambda)$ is called the asymptotic dimension of $L(\Lambda)$.
They also found the explicit formula of asymptotic dimensions.
Using their formula, we find that the right-hand side of our conjecture at $x=1$ and the asymptotic dimension
of $L(\level \Lambda_0)$ where $\Lambda_0$ is the $0$-th fundamental weight is related as follows:
\begin{gather*}
 \frac{N_{X_r, \level} (1)}{D_{X_r, \level}(1)} = \lvert P / Q \rvert^{-1} a(\level \Lambda_0)^{-2},
\end{gather*}
where $P$ and $Q$ are the weight lattice and the root lattice for $\mathfrak{g}$, respectively.
Therefore, if the conjecture is true, we can describe the asymptotic dimension using the exponents.

Another interesting feature about the exponents that we deal with in this paper is that the relationship between the exponents and partition $q$-series.
Partition $q$-series, which is defined by Kato and Terashima~\cite{KatoTerashima}, is a certain $q$-series associated with a mutation sequence.
Typically, it has the following form:
\begin{gather*}
 \mathcal{Z} (q) = \sum_{u \in (\Z_{\geq 0})^{T} }
 \frac{q^{ \frac{1}{2} u^{\mathsf{T}} K u }}{ (q)_{u_1} \cdots (q)_{u_T} },
\end{gather*}
where $T$ is a positive integer, $K$ is a $T \times T$ positive definite symmetric matrix with rational coefficients, and $(q)_n$ is defined by
\begin{gather*}
 (q)_{n} = \prod_{k=1}^{n} \big(1 - q^k\big).
\end{gather*}
This type of $q$-series has been studied in various contexts, such as characters of conformal field theories, Rogers-Ramanujan type identities, algebraic $K$-theory, and 3-dimensional topology~\cite{BartlettWarnaar,calegari2017bloch,FGK,GaroufalidisLe,georgiev1995combinatorial,HKOTT,KNS,Lee,NRT,Stoyanovskii,StoyanovskiiFeigin,Warnaar2007,WarnaarZudilin,Zagier}.

As explained earlier, we can obtain a mutation sequence for any pair $(X_r, \level)$ that describes the corresponding $Y$-system.
We can define the partition $q$-series $\mathcal{Z} (q)$ associated with this mutation sequence.
We show that the exponents and the asymptotics of this partition $q$-series are related as follows:
\begin{gather*}
 \lim_{\varepsilon \to 0}\mathcal{Z} \big({\rm e}^{-\varepsilon}\big) {\rm e}^{-\frac{a}{\varepsilon}} =
\sqrt{\frac{\det A_+ }{\det(I - J (\eta))}},
\end{gather*}
where $a$ is a positive real number and $A_+$ is a matrix determined by the mutation sequence that we call (the plus one of) the Neumann--Zagier matrix.
This result explains the conjectural relationship between the exponents and the asymptotic dimension mentioned earlier from a~viewpoint of $q$-series.
As observed in~\cite{KatoTerashima}, the partition $q$-series conjecturally coincides with a~finite sum of the string functions associated with the integrable highest $\hat{\mathfrak{g}}$-module $L(\level \Lambda_0)$ via the conjecture in~\cite{KNS}.
Since Kac and Peterson~\cite{KacPet} showed that the asymptotic dimension appears in the asymptotics of the string functions,
we can obtain the relationship between the exponents and the asymptotic dimension from the relationship between the partition $q$-series and the string functions.
We see that the resulting relationship is exactly our conjecture at $x = 1$.
This gives us a consistency between our conjecture on exponents and the known conjecture on $q$-series, and also gives an interesting connection between the theory of cluster algebras and the representation theory of affine Lie algebras.

\begin{Remark} The tuple $\eta$ of positive real numbers plays an important role in our definition of the exponents. We remark that there is a conjectural explicit formula of $\eta$ for general $(X_r,\level)$ described by the $q$-dimension of Kirillov-Reshetikhin modules~\cite[Conjecture~2]{Kuniba93} (see also~\cite[Conjecture~14.2 and Remark~14.3]{KNS}).
 A version of this conjecture is also found in~\cite{Kirillov87}.
 It is not difficult to verify this conjecture for type $A$.
 It was also proved for type $D$ by Lee~\cite{Lee13},
 for type~$E_6$ by Gleitz~\cite{Gleitz},
 and for all classical types by Lee~\cite{Lee17}.
\end{Remark}

This paper is organized as follows.
In Section~\ref{sec:quiver mutation and y-seed mutation}, we review quiver mutations and $Y$-seed mutations.
In Section~\ref{sec:exponents}, we define the exponents associated with the $Y$-system for any pair~$(X_r, \level)$, and give a conjecture that describe the exponents in terms of the root system of type~$X_r$.
In Section~\ref{sec:proofs for type A}, we give proofs of this conjecture for $(A_q, \level)$ and $(A_r, 2)$.
In Section~\ref{sec:partition q series}, we study partition $q$-series.
We calculate the asymptotics of partition $q$-series in Section~\ref{sec:asymptotics of Z},
and give an explicit formula for the partition $q$-series associated with any pair $(X_r, \level)$ in Section~\ref{sec:partition q of gamma Xr level}.
In Section~\ref{sec:relationships with affine Lie alg}, we discuss relationships between the topics discussed so far and the representation theory of affine Lie algebras.
We see that the asymptotics of the conjectural identity between the partition $q$-series and the string functions yields our conjecture at $x=1$.
We point out that this gives the relationship between the exponents and the asymptotic dimensions of integrable highest weight representations of affine Lie algebras.

\section[Quiver mutations and $Y$-seed mutations]{Quiver mutations and $\boldsymbol{Y}$-seed mutations}\label{sec:quiver mutation and y-seed mutation}

Here, we review quiver mutations and $Y$-seed mutations following~\cite{FZ4}.
\subsection{Quiver mutations}
Set $n$ to be a fixed positive integer.
A \emph{quiver} is a directed graph with vertices $\mathbf{I}:=\{1 ,\dots , n \}$ that may have multiple edges. In this paper, we assume that quivers do not have 1-loops and 2-cycles:
\begin{align*}
\begin{tikzpicture}
\node(a) at (0,0) [circle,draw,fill,scale=0.6]{};
\draw [arrows={-Stealth[scale=1.2]}] (a) edge [in=140,out=40,looseness=15] (a);
\node(b) at (2,0.2) [circle,draw,fill,scale=0.6]{};
\node(c) at (3,0.2) [circle,draw,fill,scale=0.6]{};
\draw[arrows={-Stealth[scale=1.2]}] (b) edge [bend left=30](c);
\draw[arrows={-Stealth[scale=1.2]}] (c) edge [bend left = 30](b);
\node (lu) at (-0.7,1.4-0.4) []{};
\node (rb) at (0.7,0-0.4) []{};
\node (lb) at (-0.7,-0.4) []{};
\node (ru) at (0.7,1.4-0.4) []{};
\draw[] (lu)--(rb);
\draw[] (lb)--(ru);
\node (lu2) at (-0.7+2.5,1.4-0.4) []{};
\node (rb2) at (0.7+2.5,0-0.4) []{};
\node (lb2) at (-0.7+2.5,-0.4) []{};
\node (ru2) at (0.7+2.5,1.4-0.4) []{};
\draw[] (lu2)--(rb2);
\draw[] (lb2)--(ru2);
\end{tikzpicture}.
\end{align*}
For example,
\[
\begin{tikzpicture}
[scale=1.1]
\node(k) at (0,0) [circle,scale=0.6,draw,fill]{};
\node(label) at (0.25,-0.25) []{$2$};
\node(a) at (-1,0) [circle,scale=0.6,draw,fill]{};
\node(label1) at (-1+0.25,-0.25) []{$1$};
\node(b) at (0.7,0.7) [circle,scale=0.6,draw,fill]{};
\node(label3) at (0.7+0.25,0.7-0.25) []{$3$};
\node(c) at (0,-1) [circle,scale=0.6,draw,fill]{};
\node(label4) at (0.25,-1-0.25) []{$4$};
\draw[arrows={-Stealth[scale=1.2] Stealth[scale=1.2]}] (k)--(c);
\foreach \from/\to in {k/b,a/k,b/a}
\draw[arrows={-Stealth[scale=1.2]}] (\from)--(\to);
\end{tikzpicture}
\]
is a quiver.

\begin{Definition}\label{def: quiver mutation}
 Let $Q$ be a quiver, and let $k$ be a vertex of $Q$.
 The \emph{quiver mutation} $\mu_k$ is a~transformation that transforms $Q$ into the quiver $\mu_k(Q)$ defined by the following three steps:
 \begin{enumerate}\itemsep=0pt
 \item For each length two path $i \to k \to j$, add a new arrow $i \to j$.
 \item Reverse all arrows incident to the vertex $k$.
 \item Remove all 2-cycles.
 \end{enumerate}
\end{Definition}

The transition
\[
\begin{tikzpicture}
[scale=1.0]
\node(k) at (0,0) [circle,scale=0.6,draw,fill]{};
\node(label) at (0.25,-0.25) []{$k$};
\node(a) at (-1,0) [circle,scale=0.6,draw,fill]{};
\node(b) at (0.7,0.7) [circle,scale=0.6,draw,fill]{};
\node(c) at (0,-1) [circle,scale=0.6,draw,fill]{};
\draw[arrows={-Stealth[scale=1.2] Stealth[scale=1.2]}] (k)--(c);
\foreach \from/\to in {k/b,a/k,b/a}
\draw[arrows={-Stealth[scale=1.2]}] (\from)--(\to);
\node(mutation) at (1.8,-0.2) [scale=1.2] {$\longrightarrow$};
\node(mutation) at (1.8,0.2) [] {$\mu_k$};
\node(k') at (4,0) [circle,scale=0.6,draw,fill]{};
\node(label') at (4.25,-0.25) []{$k$};
\node(a') at (3,0) [circle,scale=0.6,draw,fill]{};
\node(b') at (4.7,0.7) [circle,scale=0.6,draw,fill]{};
\node(c') at (4,-1) [circle,scale=0.6,draw,fill]{};
\draw[arrows={-Stealth[scale=1.2] Stealth[scale=1.2]}] (c')--(k');
\draw[arrows={-Stealth[scale=1.2] Stealth[scale=1.2]}] (a')--(c');
\draw[arrows={-Stealth[scale=1.2]}] (b')--(k');
\draw[arrows={-Stealth[scale=1.2]}] (k')--(a');
\end{tikzpicture}
\]
is an example of a quiver mutation, where we have omitted all labels other than the vertex $k$.

It is sometimes convenient to identify a quiver with a skew-symmetric matrix.
Let $Q$ be a~quiver.
For vertices $i$,~$j$, we denote the number of arrows from $i$ to $j$ by $Q_{ij}$.
Let $B$ be the $n \times n$ matrix whose $(i,j)$-entry is defined by
\begin{gather*}
 B_{ij} = Q_{ij} - Q_{ji}.
\end{gather*}
Then the matrix $B$ is skew-symmetric.
Conversely, given any skew-symmetric matrix $B$,
we can construct a quiver $Q$ by
\begin{gather*}
Q_{ij} = \maxzero{B_{ij}},
\end{gather*}
where $\maxzero{x} = \max(0,x)$, and this gives a bijection between the set of $n \times n$ skew-symmetric integer matrices and the set of quivers with vertices labeled with $1, \dots, n$.

In the language of skew-symmetric matrices, the quiver mutation $Q \mapsto \mu_k(Q)$ can be described as follows:
\begin{gather*}
\widetilde{B}_{ij} =
 \begin{cases}
-B_{ij} & \text{if $i=k$ or $j=k$},\\
B_{ij} + B_{ik}B_{kj} & \text{if $B_{ik}>0$ and $B_{kj}>0$}, \\
B_{ij} - B_{ik}B_{kj}& \text{if $B_{ik}<0$ and $B_{kj}<0$}, \\
B_{ij} & \text{otherwise},
\end{cases}
\end{gather*}
where $B$ and $\widetilde{B}$ are the skew-symmetric matrices corresponding to $Q$ and $\mu_k(Q)$, respectively.

Let $\nu$ be a permutation of $\{1 ,\dots , n \}$.
We define the action of $\nu$ on a quiver $Q$ by $\sigma(Q)_{ij} = Q_{\nu^{-1}(i)\nu^{-1}(j)}$.

Let $m=(m_1 ,\dots , m_T)$ be a sequence of vertices of a quiver $Q$, and $\nu$ be a permutation on the vertices of $Q$.
Consider the transitions of quivers
\begin{gather*}
Q(0) \xrightarrow{\mu_{m_1}} Q(1) \xrightarrow{\mu_{m_2}} \cdots
\xrightarrow{\mu_{m_T}} Q(T)
\xrightarrow{\nu} \nu(Q(T)).
\end{gather*}
where $Q(0):=Q$. We say that a triple $\gamma =(Q,m,\nu)$ is a \emph{mutation loop} if $Q(0)=\nu(Q(T))$

We define the following notion introduced in~\cite{Nakb}.
\begin{Definition}\label{def:regular}
 We say that a mutation loop $\gamma =(Q,m,\nu)$ is \emph{regular} if it satisfies the following conditions:
 \begin{enumerate}\itemsep=0pt
 \item The set $\{ \nu^n (m_t) \,|\, n\in \Z ,\, 1\leq t \leq T\}$ coincides with~$\mathbf{I}$.
 \item The vertices $m_1, \dots , m_T$ belong to distinct $\nu$-orbits in~$\mathbf{I}$.
 \end{enumerate}
\end{Definition}

\subsection[$Y$-seed mutations]{$\boldsymbol{Y}$-seed mutations}
Let $\rf$ be the field of rational functions in the variables $y_1 , \dots , y_n$ over $\Q$, and let
\begin{gather*}
\univsf = \left\{ \frac{f(y_1, \dots ,y_n)}{g(y_1 , \dots, y_n)} \in \mathcal{F} \,\bigg|\, \begin{array}{l} \text{$f$ and $g$ are non-zero polynomials in $\Q [y_1,\dots, y_n ] $}\\
\text{with non-negative coefficients}\end{array} \right\}
\end{gather*}
be the set of subtraction-free rational expressions in $y_1, \dots, y_n$ over $\Q$.
This is closed under the usual multiplication and addition, and is called \emph{universal semifield} in the variables $y_1 ,\dots , y_n$.

A \emph{$Y$-seed} is a pair $(Q,Y)$ where $Q$ is a quiver with vertices $\{ 1 , \dots , n \}$, and $Y = (Y_1 \dots, Y_n)$ is an $n$-tuple of elements of $\univsf$.
Given this, $Y$-seed mutations are defined as follows.

\begin{Definition} Let $(Q,Y)$ be a $Y$-seed, and let $k \in \{ 1, \dots , n \}$.
 The \emph{$Y$-seed mutation} $\mu_k$ is a~transformation that transforms $(Q,Y)$ into
 the $Y$-seed $\mu_k (Q,Y) = \big(\widetilde{Q} , \widetilde{Y} \big)$, where $\widetilde{Q} = \mu_k (Q)$ defined in Definition \ref{def: quiver mutation}, and $\widetilde{Y}$ is defined as follows:
 \begin{gather*}
 \widetilde{Y}_i =
 \begin{cases}
 Y_{k}^{-1} & \text{if $i=k$}, \\
 Y_{i} \big(Y_k^{-1}+1\big)^{-Q_{ki}} & \text{if $i \neq k$, $Q_{ki} \geq 0$}, \\
 Y_{i} (Y_k+1)^{Q_{ik}} & \text{if $i \neq k$, $Q_{ik} \geq 0$}.
 \end{cases}
 \end{gather*}
\end{Definition}

Let $\nu$ be a permutation of $\{1 ,\dots , n \}$.
We define an action of $\nu$ on a $Y$-seed by $\nu(Q,Y) = (\nu(Q),\nu(Y) )$ where $\nu(Y)_i = Y_{\nu^{-1}(i)}$.

Let $\gamma = (Q , m ,\nu)$ be a mutation loop,
and let $(Q(0) , Y(0))$ be the $Y$-seed defined by $Q(0) = Q$ and $Y(0)=(y_1 ,\dots, y_n)$.
Then the mutation loop $\gamma$ gives the following transitions of $Y$-seeds:
\begin{gather}\label{eq:Y transition}
(Q(0),Y(0)) \xrightarrow{\mu_{m_1}}
\cdots
\xrightarrow{\mu_{m_T}} (Q(T),Y(T))
\xrightarrow{\nu} (\nu(Q(T)), \nu (Y(T))).
\end{gather}
Although $Q(0)=\nu(Q(T))$ holds from the definition of a mutation loop, we have $Y(0) \neq \nu(Y(T))$ in general. We denote $\nu(Y(T))$ by $\mu_\gamma(y)$, and call it the \emph{cluster transformation} of $\gamma$. It is an element of $(\univsf)^n$.

\section[Exponents determined from a pair $(X_r, \ell)$]{Exponents determined from a pair $\boldsymbol{(X_r, \ell)}$}\label{sec:exponents}
\subsection{Root systems}\label{sec:root system}

In this section, we introduce notations of root systems.
Let $\Delta$ be a~root system of type $X_r$ on a~$\R$-vector space with an inner product normalized as $\innerproduct{\alpha}{\alpha}=2$ for long roots $\alpha$, where~$X_r$ is a~finite type Dynkin diagram in the Fig.~\ref{fig:finite dynkin}.
\begin{figure}[t]
 \centering
 \begin{tabular}{ll}
 \raisebox{4mm}{$A_r$} &
 \begin{tikzpicture}
 [scale=1.2,auto=left,black_vertex/.style={circle,draw,fill,scale=0.75}]
 \node (1) at (0,0) [black_vertex][label=below:$1$]{};
 \node (2) at (1,0) [black_vertex][label=below:$2$]{};
 \node (dots) at (2,0) []{$\cdots$};
 \node (r-1) at (3,0) [black_vertex] [label=below:$r-1$]{};
 \node (r) at (4,0) [black_vertex][label=below:$r$]{};
 \draw [] (1)--(2);
 \draw [] (2)--(dots);
 \draw [] (dots)--(r-1);
 \draw [] (r-1)--(r);
 \end{tikzpicture} \\
 \raisebox{4mm}{$B_r$} &
 \begin{tikzpicture}
 [scale=1.2,auto=left,black_vertex/.style={circle,draw,fill,scale=0.75}]
 \node (1) at (0,0) [black_vertex][label=below:$1$]{};
 \node (2) at (1,0) [black_vertex][label=below:$2$]{};
 \node (dots) at (2,0) []{$\cdots$};
 \node (r-1) at (3,0) [black_vertex][label=below:$r-1$]{};
 \node (r) at (4,0) [black_vertex][label=below:$r$]{};
 \draw [] (1)--(2);
 \draw [] (2)--(dots);
 \draw [] (dots)--(r-1);
 \draw [] (3,0.08)--(4,0.08);
 \draw [] (3,-0.08)--(4,-0.08);
 \draw
 (3.5,0) --++ (120:.3)
 (3.5,0) --++ (-120:.3);
 \end{tikzpicture} \\
 \raisebox{4mm}{$C_r$} &
 \begin{tikzpicture}
 [scale=1.2,auto=left,black_vertex/.style={circle,draw,fill,scale=0.75}]
 \node (1) at (0,0) [black_vertex][label=below:$1$]{};
 \node (2) at (1,0) [black_vertex][label=below:$2$]{};
 \node (dots) at (2,0) []{$\cdots$};
 \node (r-1) at (3,0) [black_vertex][label=below:$r-1$]{};
 \node (r) at (4,0) [black_vertex][label=below:$r$]{};
 \draw [] (1)--(2);
 \draw [] (2)--(dots);
 \draw [] (dots)--(r-1);
 \draw [] (3,0.08)--(4,0.08);
 \draw [] (3,-0.08)--(4,-0.08);
 \draw
 (3.5,0) --++ (60:.3)
 (3.5,0) --++ (-60:.3);
 \end{tikzpicture} \\
 \raisebox{4mm}{$D_r$} &
 \begin{tikzpicture}
 [scale=1.2,auto=left,black_vertex/.style={circle,draw,fill,scale=0.75}]
 \node (1) at (0,0) [black_vertex][label=below:$1$]{};
 \node (2) at (1,0) [black_vertex][label=below:$2$]{};
 \node (dots) at (2,0) []{$\cdots$};
 \node (r-2) at (3,0) [black_vertex] [label=below:$r-2$]{};
 \node (r-1) at (4,0) [black_vertex][label=below:$r-1$]{};
 \node (r) at (3,1) [black_vertex][label=right:$r$]{};
 \draw [] (1)--(2);
 \draw [] (2)--(dots);
 \draw [] (dots)--(r-2);
 \draw [] (r-2)--(r-1);
 \draw [] (r-2)--(r);
 \end{tikzpicture} \\
 \raisebox{4mm}{$E_6$} &
 \begin{tikzpicture}
 [scale=1.2,auto=left,black_vertex/.style={circle,draw,fill,scale=0.75}]
 \node (1) at (0,0) [black_vertex][label=below:$1$]{};
 \node (2) at (1,0) [black_vertex][label=below:$2$]{};
 \node (3) at (2,0) [black_vertex][label=below:$3$]{};
 \node (5) at (3,0) [black_vertex] [label=below:$5$]{};
 \node (6) at (4,0) [black_vertex][label=below:$6$]{};
 \node (4) at (2,1) [black_vertex][label=right:$4$]{};
 \draw [] (1)--(2);
 \draw [] (2)--(3);
 \draw [] (3)--(4);
 \draw [] (3)--(5);
 \draw [] (5)--(6);
 \end{tikzpicture} \\
 \raisebox{4mm}{$E_7$} &
 \begin{tikzpicture}
 [scale=1.2,auto=left,black_vertex/.style={circle,draw,fill,scale=0.75}]
 \node (1) at (0,0) [black_vertex][label=below:$1$]{};
 \node (2) at (1,0) [black_vertex][label=below:$2$]{};
 \node (3) at (2,0) [black_vertex][label=below:$3$]{};
 \node (4) at (3,0) [black_vertex] [label=below:$4$]{};
 \node (5) at (4,0) [black_vertex][label=below:$5$]{};
 \node (6) at (5,0) [black_vertex][label=below:$6$]{};
 \node (7) at (2,1) [black_vertex][label=right:$7$]{};
 \draw [] (1)--(2);
 \draw [] (2)--(3);
 \draw [] (3)--(4);
 \draw [] (4)--(5);
 \draw [] (5)--(6);
 \draw [] (3)--(7);
 \end{tikzpicture} \\
 \raisebox{4mm}{$E_8$} &
 \begin{tikzpicture}
 [scale=1.2,auto=left,black_vertex/.style={circle,draw,fill,scale=0.75}]
 \node (1) at (0,0) [black_vertex][label=below:$1$]{};
 \node (2) at (1,0) [black_vertex][label=below:$2$]{};
 \node (3) at (2,0) [black_vertex][label=below:$3$]{};
 \node (4) at (3,0) [black_vertex][label=below:$4$]{};
 \node (5) at (4,0) [black_vertex][label=below:$5$]{};
 \node (6) at (5,0) [black_vertex][label=below:$6$]{};
 \node (7) at (6,0) [black_vertex][label=below:$7$]{};
 \node (8) at (4,1) [black_vertex][label=right:$8$]{};
 \draw [] (1)--(2);
 \draw [] (2)--(3);
 \draw [] (3)--(4);
 \draw [] (4)--(5);
 \draw [] (5)--(6);
 \draw [] (6)--(7);
 \draw [] (5)--(8);
 \end{tikzpicture} \\
 \raisebox{4mm}{$F_4$} &
 \begin{tikzpicture}
 [scale=1.2,auto=left,black_vertex/.style={circle,draw,fill,scale=0.75}]
 \node (1) at (0,0) [black_vertex][label=below:$1$]{};
 \node (2) at (1,0) [black_vertex][label=below:$2$]{};
 \node (3) at (2,0) [black_vertex][label=below:$3$]{};
 \node (4) at (3,0) [black_vertex][label=below:$4$]{};
 \draw [] (1)--(2);
 \draw [] (1,0.08)--(2,0.08);
 \draw [] (1,-0.08)--(2,-0.08);
 \draw [] (3)--(4);
 \draw
 (1.5,0) --++ (120:.3)
 (1.5,0) --++ (-120:.3);
 \end{tikzpicture} \\
 \raisebox{4mm}{$G_2$} &
 \begin{tikzpicture}
 [scale=1.2,auto=left,black_vertex/.style={circle,draw,fill,scale=0.75}]
 \node (1) at (0,0) [black_vertex][label=below:$1$]{};
 \node (2) at (1,0) [black_vertex][label=below:$2$]{};
 \draw [] (1)--(2);
 \draw [] (0,0.08)--(1,0.08);
 \draw [] (0,-0.08)--(1,-0.08);
 \draw
 (0.5,0) --++ (120:.3)
 (0.5,0) --++ (-120:.3);
 \end{tikzpicture} \\
 \end{tabular}
 \caption{The list of finite type Dynkin diagrams.}
 \label{fig:finite dynkin}
\end{figure}
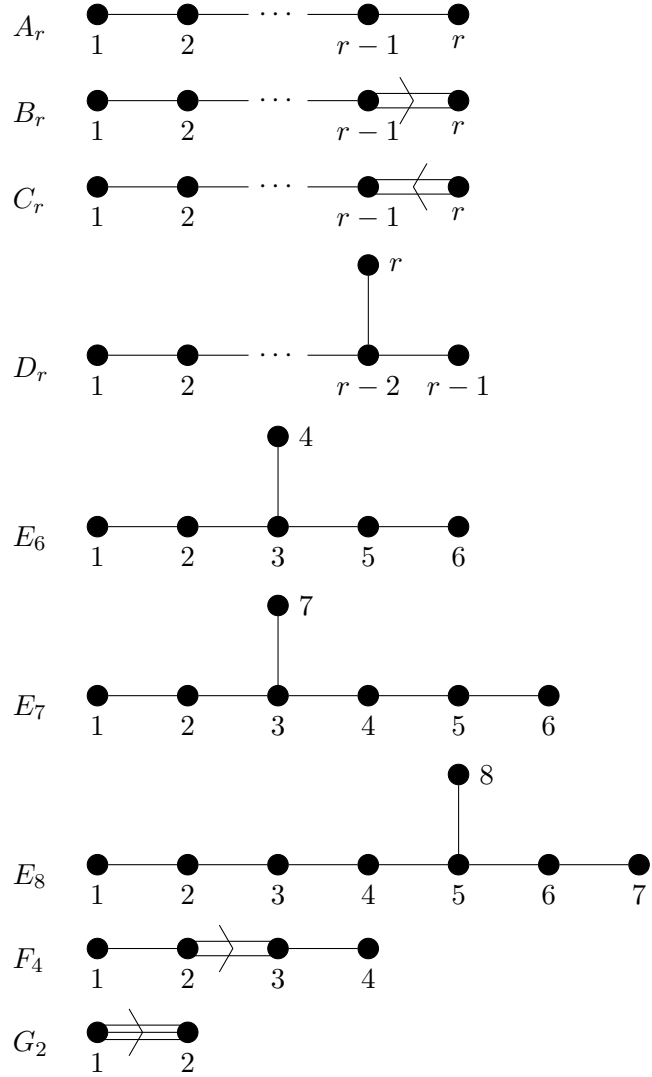
Let $\alpha_1 , \dots, \alpha_r$ be simple roots, where the numberings are consistent with the numberings of nodes in Fig.~\ref{fig:finite dynkin}.
Let $\Delta^{{\rm long}}$ and $\Delta^{{\rm short}}$ be the set of long roots and short roots, respectively. Let $\Delta_+$ be the set of positive roots, and $\rho$ be the half of the sum of the positive roots:
\begin{gather*}
 \rho = \frac{1}{2} \sum_{\alpha \in \Delta_+} \alpha.
\end{gather*}

We define an integer $t$ by
\begin{gather}\label{eq: t list}
t =
\begin{cases}
1 & \text{if $X_r = A_r,D_r,E_6 ,E_7$ or $E_8$},\\
2 & \text{if $X_r = B_r,C_r$ or $F_4$}, \\
3 & \text{if $X_r =G_2$}.
\end{cases}
\end{gather}
For $a=1, \dots, r$, we also define an integer $t_a \in \{ 1,2,3 \}$ by
\begin{gather*}
 t_a = \frac{2}{\innerproduct{\alpha_a}{\alpha_a}}.
\end{gather*}

\subsection[Quiver $Q(X_r,\ell)$]{Quiver $\boldsymbol{Q(X_r,\ell)}$}\label{section: quiver Xr l}
For any finite type Dynkin diagram $X_r$ and positive integer $\level$ such that $\level \geq 2$ (called a \emph{level}), Inoue, Iyama, Keller, Kuniba and Nakanishi~\cite{IIKKNa,IIKKNb} defined a quiver $Q=Q(X_r, \level)$ and a~mutation loop $\gamma = \gamma(X_r ,\level)$ on $Q(X_r, \level)$ that have the following periodicity:
\begin{gather*}
 \underbrace{\mu_\gamma \circ \mu_\gamma \circ \dots \circ \mu_\gamma}_{t(\level + \dcn)} (y) = y,
\end{gather*}
where $\dcn$ is the \emph{dual Coxeter number} of $X_r$.
The list of dual Coxeter numbers is given by
\begin{gather}\label{eq: dual coxeter}
\begin{array}{c|ccccccccc}
X_r & A_r & B_r & C_r & D_r & E_6 & E_7 & E_8 & F_4 & G_2 \\
\hline
h^{\vee} & r+1 & 2r-1 & r+1 & 2r-2 & 12 & 18 & 30 & 9 & 4
\end{array}.
\end{gather}

We review the definition of $Q(X_r , \level)$ in this section,
and the definition of $\gamma(X_r , \level)$ in the next section.
In the definition of $Q(X_r ,\level)$, we also give auxiliary labels on vertices as
\begin{alignat*}{3}
 &\text{$+$ or $-$} &&\qquad \text{if $X=A, D$ or $E$},& \\
 &\text{$+$ or $-$, and $\fillvertex$ or $\empvertex$} &&\qquad \text{if $X=B, C$ or $F$},& \\
 &\text{$+,-,\mathrm{I},\mathrm{II},\mathrm{III},\mathrm{IV},\mathrm{V}$ or $\mathrm{VI}$, and $\fillvertex$ or $\empvertex$} &&\qquad \text{if $X=G$}.&
\end{alignat*}
These labels will be used when defining a mutation sequence $\gamma(X_r ,\level)$.

\subsubsection*{Type $\boldsymbol{ADE}$}
First, we assume that $X_r$ is simply laced, that is, $X = A, D$ or $E$.
Let $X'_{r'}$ be another simply laced Dynkin diagram.
Let $C = (C_{ij})_{i,j=1,\dots, r}$ and $C' = (C_{i'j'})_{i',j'=1,\dots, r'}$ be Cartan matrices of types $X_r$ and $X'_{r'}$, respectively. Let $I= \{ 1 , \dots , r \}$ and $I' = \{ 1, \dots ,r' \}$ be the nodes in these Dynkin diagrams. We fix bipartite decompositions $I= I_+ \sqcup I_-$ and $I'=I'_+ \sqcup I'_-$.
Let $\mathbf{I}=I \times I'$. For an element $\mathbf{i} = (i,i') \in \mathbf{I}$, we write $\mathbf{i}\colon (++)$ if $\mathbf{i} \in I_+ \times I'_+$, etc. Let $B$ be the skew-symmetric $\mathbf{I} \times \mathbf{I}$ matrix defined by
\begin{gather*}
 B_{\mathbf{i}\mathbf{j}} =
 \begin{cases}
 -C_{ij} \delta_{i'j'}
 &\text{if $\mathbf{i}\colon (-+)$, $\mathbf{j}\colon (++)$ or $\mathbf{i}\colon (+-)$, $\mathbf{j}\colon (--)$,} \\
 C_{ij} \delta_{i'j'} &\text{if $\mathbf{i}\colon (++)$, $\mathbf{j}\colon (-+)$ or $\mathbf{i}\colon (--)$, $\mathbf{j}\colon (+-)$}, \\
 -\delta_{ij} C_{i'j'}
 &\text{if $\mathbf{i}\colon (++)$, $\mathbf{j}\colon (+-)$ or $\mathbf{i}\colon (--)$, $\mathbf{j}\colon (-+)$}, \\
 \delta_{ij} C_{i'j'} &\text{if $\mathbf{i}\colon (+-)$, $\mathbf{j}\colon (++)$ or $\mathbf{i}\colon (-+)$, $\mathbf{j}\colon (--)$}, \\
 0 &\text{otherwise}.
 \end{cases}
\end{gather*}
We define a quiver $Q(X_r, \level)$ as the quiver corresponding to the skew-symmetric matrix $B$ for $(X_r, X'_{r'})=(X_r, A_{\level-1})$. For each vertex $\mathbf{i}$, we label $\mathbf{i}$ as $+$ if $\mathbf{i}\colon (++)$ or $(--)$ and $-$ if $\mathbf{i}\colon (+-)$ or $(-+)$.

For example, the quiver $Q(A_4 ,4)$ is given by
\[
\begin{tikzpicture}
[scale=0.70,auto=left,black_vertex/.style={circle,draw,fill,scale=0.75},white_vertex/.style={circle,draw,scale=0.75}]
\draw [decorate,decoration={brace,amplitude=8pt}]
(1,2) -- (1,6) node [black,midway,xshift=-0.4cm] {$\level-1$};
\draw [decorate,decoration={brace,amplitude=8pt,mirror}]
(2,1.5) -- (8,1.5) node [black,midway,yshift=-0.9cm] {$r$};
\node (11) at (2,2) [black_vertex]{};
\node at (2.3,2.3) []{$-$};
\node (12) at (2,4) [black_vertex]{};
\node at (2.3,4.3) []{$+$};
\node (13) at (2,6) [black_vertex]{};
\node at (2.3,6.3) []{$-$};
\node (21) at (4,2) [black_vertex]{};
\node at (4.3,2.3) []{$+$};
\node (22) at (4,4) [black_vertex]{};
\node at (4.3,4.3) []{$-$};
\node (23) at (4,6) [black_vertex]{};
\node at (4.3,6.3) []{$+$};
\node (31) at (6,2) [black_vertex]{};
\node at (6.3,2.3) []{$-$};
\node (32) at (6,4) [black_vertex]{};
\node at (6.3,4.3) []{$+$};
\node (33) at (6,6) [black_vertex]{};
\node at (6.3,6.3) []{$-$};
\node (41) at (8,2) [black_vertex]{};
\node at (8.3,2.3) []{$+$};
\node (42) at (8,4) [black_vertex]{};
\node at (8.3,4.3) []{$-$};
\node (43) at (8,6) [black_vertex]{};
\node at (8.3,6.3) []{$+$};
\foreach \from/\to in {12/11,
 12/13,
 21/22,
 23/22,
 32/31,
 32/33,
 41/42,
 43/42,
 11/21,
 13/23,
 22/12,
 22/32,
 31/21,
 31/41,
 33/23,
 33/43,
 42/32}
\draw[arrows={-Stealth[scale=1.5]}] (\from)--(\to);
\end{tikzpicture},
\]
and the quiver $Q(D_5, 4)$ is given by
\[
 \begin{tikzpicture}
 [scale=0.70,auto=left,black_vertex/.style={circle,draw,fill,scale=0.75},white_vertex/.style={circle,draw,scale=0.75}]
 \draw [decorate,decoration={brace,amplitude=8pt}]
 (1,2) -- (1,6) node [black,midway,xshift=-0.4cm] {$\level-1$};
 \node (11) at (2,2) [black_vertex]{};
 \node at (2.3,2.3) []{$-$};
 \node (12) at (2,4) [black_vertex]{};
 \node at (2.3,4.3) []{$+$};
 \node (13) at (2,6) [black_vertex]{};
 \node at (2.3,6.3) []{$-$};
 \node (21) at (4,2) [black_vertex]{};
 \node at (4.3,2.3) []{$+$};
 \node (22) at (4,4) [black_vertex]{};
 \node at (4.3,4.3) []{$-$};
 \node (23) at (4,6) [black_vertex]{};
 \node at (4.3,6.3) []{$+$};
 \node (31) at (6,2) [black_vertex]{};
 \node at (6.3,2.3) []{$-$};
 \node (32) at (6,4) [black_vertex]{};
 \node at (6.3,4.3) []{$+$};
 \node (33) at (6,6) [black_vertex]{};
 \node at (6.3,6.3) []{$-$};
 \node (41) at (8,2) [black_vertex]{};
 \node at (8.3,2.3) []{$+$};
 \node (42) at (8,4) [black_vertex]{};
 \node at (8.3,4.3) []{$-$};
 \node (43) at (8,6) [black_vertex]{};
 \node at (8.3,6.3) []{$+$};
 \node (51) at (7.5,1) [black_vertex]{};
 \node at (7.8,1.3) []{$+$};
 \node (52) at (7.5,3) [black_vertex]{};
 \node at (7.8,3.3) []{$-$};
 \node (53) at (7.5,5) [black_vertex]{};
 \node at (7.8,5.3) []{$+$};
 \foreach \from/\to in {12/11,
 12/13,
 21/22,
 23/22,
 32/31,
 32/33,
 41/42,
 43/42,
 11/21,
 13/23,
 22/12,
 22/32,
 31/21,
 31/41,
 33/23,
 33/43,
 42/32}
 \draw[arrows={-Stealth[scale=1.5]}] (\from)--(\to);
 \foreach \from/\to in {31/51,52/32,33/53,51/52,53/52}
 \draw[arrows={-Stealth[scale=1.5]}] (\from)--(\to);
 \end{tikzpicture}.
\]

\subsubsection*{Type $\boldsymbol{BCF}$}
Next, we assume that $X=B ,C$ or $F$.

We define $Q(B_r, \level)$ as
\[
 \begin{tikzpicture}
 [scale=0.70,auto=left,black_vertex/.style={circle,draw,fill,scale=0.75},white_vertex/.style={circle,draw,scale=0.75}]
 \draw [decorate,decoration={brace,amplitude=8pt}]
 (0,2) -- (0,6) node [black,midway,xshift=-0.4cm] {$\level-1$};
 \draw [decorate,decoration={brace,amplitude=8pt,mirror}]
 (0.5,1.5) -- (6,1.5) node [black,midway,yshift=-0.9cm] {$r-1$};
 \draw [decorate,decoration={brace,amplitude=8pt,mirror}]
 (10,1.5) -- (15.5,1.5) node [black,midway,yshift=-0.9cm] {$r-1$};

 \node at (1.2,2) [scale=1.5] {$\cdots$};
 \node at (1.2,4) [scale=1.5] {$\cdots$};
 \node at (1.2,6) [scale=1.5] {$\cdots$};
 \node at (15,2) [scale=1.5] {$\cdots$};
 \node at (15,4) [scale=1.5] {$\cdots$};
 \node at (15,6) [scale=1.5] {$\cdots$};

 \node (11) at (2,2) [white_vertex]{};
 \node at (2.3,2.3) []{$-$};
 \node (12) at (2,4) [white_vertex]{};
 \node at (2.3,4.3) []{$+$};
 \node (13) at (2,6) [white_vertex]{};
 \node at (2.3,6.3) []{$-$};
 \node (21) at (4,2) [white_vertex]{};
 \node at (4.3,2.3) []{$+$};
 \node (22) at (4,4) [white_vertex]{};
 \node at (4.3,4.3) []{$-$};
 \node (23) at (4,6) [white_vertex]{};
 \node at (4.3,6.3) []{$+$};
 \node (31) at (6,2) [white_vertex]{};
 \node at (6.3,2.3) []{$-$};
 \node (32) at (6,4) [white_vertex]{};
 \node at (6.3,4.3) []{$+$};
 \node (33) at (6,6) [white_vertex]{};
 \node at (6.3,6.3) []{$-$};
 \node (41) at (8,1) [black_vertex]{};
 \node at (8.3,1.3) []{$+$};
 \node (42) at (8,2) [black_vertex]{};
 \node at (8.3,2.3) []{$-$};
 \node (43) at (8,3) [black_vertex]{};
 \node at (8.3,3.3) []{$+$};
 \node (44) at (8,4) [black_vertex]{};
 \node at (8.3,4.3) []{$-$};
 \node (45) at (8,5) [black_vertex]{};
 \node at (8.3,5.3) []{$+$};
 \node (46) at (8,6) [black_vertex]{};
 \node at (8.3,6.3) []{$-$};
 \node (47) at (8,7) [black_vertex]{};
 \node at (8.3,7.3) []{$+$};
 \node (51) at (10,2) [white_vertex]{};
 \node at (10.3,2.3) []{$+$};
 \node (52) at (10,4) [white_vertex]{};
 \node at (10.3,4.3) []{$-$};
 \node (53) at (10,6) [white_vertex]{};
 \node at (10.3,6.3) []{$+$};
 \node (61) at (12,2) [white_vertex]{};
 \node at (12.3,2.3) []{$-$};
 \node (62) at (12,4) [white_vertex]{};
 \node at (12.3,4.3) []{$+$};
 \node (63) at (12,6) [white_vertex]{};
 \node at (12.3,6.3) []{$-$};
 \node (71) at (14,2) [white_vertex]{};
 \node at (14.3,2.3) []{$+$};
 \node (72) at (14,4) [white_vertex]{};
 \node at (14.3,4.3) []{$-$};
 \node (73) at (14,6) [white_vertex]{};
 \node at (14.3,6.3) []{$+$};

 \foreach \from/\to in {12/11,
 12/13,
 21/22,
 23/22,
 32/31,
 32/33,
 41/42,
 41/44,
 43/42,
 43/44,
 43/46,
 45/42,
 45/44,
 45/46,
 47/44,
 47/46,
 51/52,
 53/52,
 62/61,
 62/63,
 71/72,
 73/72,
 11/21,
 13/23,
 22/12,
 22/32,
 31/21,
 31/41,
 31/43,
 33/23,
 33/45,
 33/47,
 42/51,
 44/32,
 46/53,
 52/43,
 52/45,
 52/62,
 61/51,
 61/71,
 63/53,
 63/73,
 72/62}
 \draw[arrows={-Stealth[scale=1.5]}] (\from)--(\to);
 \foreach \from/\to in {42/31,44/52,46/33}
 \draw[arrows={-Stealth[scale=1.5]}] (\from)--(\to);

 \end{tikzpicture}
\]
if $\level$ is even, and
\[
 \begin{tikzpicture}
 [scale=0.70,auto=left,black_vertex/.style={circle,draw,fill,scale=0.75},white_vertex/.style={circle,draw,scale=0.75}]
 \draw [decorate,decoration={brace,amplitude=8pt}]
 (0,2) -- (0,4) node [black,midway,xshift=-0.4cm] {$\level-1$};
 \draw [decorate,decoration={brace,amplitude=8pt,mirror}]
 (0.5,1.5) -- (6,1.5) node [black,midway,yshift=-0.9cm] {$r-1$};
 \draw [decorate,decoration={brace,amplitude=8pt,mirror}]
 (10,1.5) -- (15.5,1.5) node [black,midway,yshift=-0.9cm] {$r-1$};

 \node at (1.2,2) [scale=1.5] {$\cdots$};
 \node at (1.2,4) [scale=1.5] {$\cdots$};
 \node at (15,2) [scale=1.5] {$\cdots$};
 \node at (15,4) [scale=1.5] {$\cdots$};

 \node (11) at (2,2) [white_vertex]{};
 \node at (2.3,2.3) []{$-$};
 \node (12) at (2,4) [white_vertex]{};
 \node at (2.3,4.3) []{$+$};
 \node (21) at (4,2) [white_vertex]{};
 \node at (4.3,2.3) []{$+$};
 \node (22) at (4,4) [white_vertex]{};
 \node at (4.3,4.3) []{$-$};
 \node (31) at (6,2) [white_vertex]{};
 \node at (6.3,2.3) []{$-$};
 \node (32) at (6,4) [white_vertex]{};
 \node at (6.3,4.3) []{$+$};
 \node (41) at (8,1) [black_vertex]{};
 \node at (8.3,1.3) []{$+$};
 \node (42) at (8,2) [black_vertex]{};
 \node at (8.3,2.3) []{$-$};
 \node (43) at (8,3) [black_vertex]{};
 \node at (8.3,3.3) []{$+$};
 \node (44) at (8,4) [black_vertex]{};
 \node at (8.3,4.3) []{$-$};
 \node (45) at (8,5) [black_vertex]{};
 \node at (8.3,5.3) []{$+$};
 \node (51) at (10,2) [white_vertex]{};
 \node at (10.3,2.3) []{$+$};
 \node (52) at (10,4) [white_vertex]{};
 \node at (10.3,4.3) []{$-$};
 \node (61) at (12,2) [white_vertex]{};
 \node at (12.3,2.3) []{$-$};
 \node (62) at (12,4) [white_vertex]{};
 \node at (12.3,4.3) []{$+$};
 \node (71) at (14,2) [white_vertex]{};
 \node at (14.3,2.3) []{$+$};
 \node (72) at (14,4) [white_vertex]{};
 \node at (14.3,4.3) []{$-$};

 \foreach \from/\to in {12/11,
 21/22,
 32/31,
 41/42,
 41/44,
 43/42,
 43/44,
 45/42,
 45/44,
 51/52,
 62/61,
 71/72,
 11/21,
 22/12,
 22/32,
 31/21,
 31/41,
 31/43,
 42/51,
 44/32,
 52/43,
 52/45,
 52/62,
 61/51,
 61/71,
 72/62}
 \draw[arrows={-Stealth[scale=1.5]}] (\from)--(\to);
 \foreach \from/\to in {42/31,44/52}
 \draw[arrows={-Stealth[scale=1.5]}] (\from)--(\to);

 \end{tikzpicture}
\]
if $\level$ is odd.

We define $Q(C_r, \level)$ as
\[
 \begin{tikzpicture}
 [scale=0.70,auto=left,black_vertex/.style={circle,draw,fill,scale=0.75},white_vertex/.style={circle,draw,scale=0.75}]
 \draw [decorate,decoration={brace,amplitude=8pt}]
 (0,1) -- (0,7) node [black,midway,xshift=-0.4cm] {$2\level-1$};
 \draw [decorate,decoration={brace,amplitude=8pt,mirror}]
 (0.5,0.5) -- (8,0.5) node [black,midway,yshift=-0.9cm] {$r-1$};

 \node at (1.2,1) [scale=1.5] {$\cdots$};
 \node at (1.2,2) [scale=1.5] {$\cdots$};
 \node at (1.2,3) [scale=1.5] {$\cdots$};
 \node at (1.2,4) [scale=1.5] {$\cdots$};
 \node at (1.2,5) [scale=1.5] {$\cdots$};
 \node at (1.2,6) [scale=1.5] {$\cdots$};
 \node at (1.2,7) [scale=1.5] {$\cdots$};

 \node (11) at (2,1) [black_vertex]{};
 \node at (2.3,1.3) []{$-$};
 \node (12) at (2,2) [black_vertex]{};
 \node at (2.3,2.3) []{$+$};
 \node (13) at (2,3) [black_vertex]{};
 \node at (2.3,3.3) []{$-$};
 \node (14) at (2,4) [black_vertex]{};
 \node at (2.3,4.3) []{$+$};
 \node (15) at (2,5) [black_vertex]{};
 \node at (2.3,5.3) []{$-$};
 \node (16) at (2,6) [black_vertex]{};
 \node at (2.3,6.3) []{$+$};
 \node (17) at (2,7) [black_vertex]{};
 \node at (2.3,7.3) []{$-$};
 \node (21) at (4,1) [black_vertex]{};
 \node at (4.3,1.3) []{$+$};
 \node (22) at (4,2) [black_vertex]{};
 \node at (4.3,2.3) []{$-$};
 \node (23) at (4,3) [black_vertex]{};
 \node at (4.3,3.3) []{$+$};
 \node (24) at (4,4) [black_vertex]{};
 \node at (4.3,4.3) []{$-$};
 \node (25) at (4,5) [black_vertex]{};
 \node at (4.3,5.3) []{$+$};
 \node (26) at (4,6) [black_vertex]{};
 \node at (4.3,6.3) []{$-$};
 \node (27) at (4,7) [black_vertex]{};
 \node at (4.3,7.3) []{$+$};
 \node (31) at (6,1) [black_vertex]{};
 \node at (6.3,1.3) []{$-$};
 \node (32) at (6,2) [black_vertex]{};
 \node at (6.3,2.3) []{$+$};
 \node (33) at (6,3) [black_vertex]{};
 \node at (6.3,3.3) []{$-$};
 \node (34) at (6,4) [black_vertex]{};
 \node at (6.3,4.3) []{$+$};
 \node (35) at (6,5) [black_vertex]{};
 \node at (6.3,5.3) []{$-$};
 \node (36) at (6,6) [black_vertex]{};
 \node at (6.3,6.3) []{$+$};
 \node (37) at (6,7) [black_vertex]{};
 \node at (6.3,7.3) []{$-$};
 \node (41) at (8,1) [black_vertex]{};
 \node at (8.3,1.3) []{$+$};
 \node (42) at (8,2) [black_vertex]{};
 \node at (8.3,2.3) []{$-$};
 \node (43) at (8,3) [black_vertex]{};
 \node at (8.3,3.3) []{$+$};
 \node (44) at (8,4) [black_vertex]{};
 \node at (8.3,4.3) []{$-$};
 \node (45) at (8,5) [black_vertex]{};
 \node at (8.3,5.3) []{$+$};
 \node (46) at (8,6) [black_vertex]{};
 \node at (8.3,6.3) []{$-$};
 \node (47) at (8,7) [black_vertex]{};
 \node at (8.3,7.3) []{$+$};

 \draw [decorate,decoration={brace,amplitude=8pt}]
 (11.2,2) -- (11.2,6) node [black,midway,xshift=-0.4cm] {$\level-1$};

 \node (41p) at (14,1) [black_vertex]{};
 \node at (14.3,1.3) []{$+$};
 \node (42p) at (14,2) [black_vertex]{};
 \node at (14.3,2.3) []{$-$};
 \node (43p) at (14,3) [black_vertex]{};
 \node at (14.3,3.3) []{$+$};
 \node (44p) at (14,4) [black_vertex]{};
 \node at (14.3,4.3) []{$-$};
 \node (45p) at (14,5) [black_vertex]{};
 \node at (14.3,5.3) []{$+$};
 \node (46p) at (14,6) [black_vertex]{};
 \node at (14.3,6.3) []{$-$};
 \node (47p) at (14,7) [black_vertex]{};
 \node at (14.3,7.3) []{$+$};

 \node (52) at (12,2) [white_vertex]{};
 \node at (12.3,2.3) []{$-$};
 \node (54) at (12,4) [white_vertex]{};
 \node at (12.3,4.3) []{$+$};
 \node (56) at (12,6) [white_vertex]{};
 \node at (12.3,6.3) []{$-$};
 \node (52p) at (16,2) [white_vertex]{};
 \node at (16.3,2.3) []{$+$};
 \node (54p) at (16,4) [white_vertex]{};
 \node at (16.3,4.3) []{$-$};
 \node (56p) at (16,6) [white_vertex]{};
 \node at (16.3,6.3) []{$+$};

 \foreach \from/\to in {12/11,
 12/13,
 14/13,
 14/15,
 16/15,
 16/17,
 21/22,
 23/22,
 23/24,
 25/24,
 25/26,
 27/26,
 32/31,
 32/33,
 34/33,
 34/35,
 36/35,
 36/37,
 41/42,
 43/42,
 43/44,
 45/44,
 45/46,
 47/46,
 11/21,
 13/23,
 15/25,
 17/27,
 22/12,
 22/32,
 24/14,
 24/34,
 26/16,
 26/36,
 31/21,
 31/41,
 33/23,
 33/43,
 35/25,
 35/45,
 37/27,
 37/47,
 42/32,
 44/34,
 46/36}
 \draw[arrows={-Stealth[scale=1.5]}] (\from)--(\to);
 \foreach \from/\to in {41p/42p,43p/42p,43p/44p,45p/44p,45p/46p,47p/46p}
 \draw[arrows={-Stealth[scale=1.5]}] (\from)--(\to);
 \foreach \from/\to in {54/52,54/56,54p/52p,54p/56p}
 \draw[arrows={-Stealth[scale=1.5]}] (\from)--(\to);
 \foreach \from/\to in {42p/52,42p/52p,44p/54,44p/54p,46p/56,46p/56p}
 \draw[arrows={-Stealth[scale=1.5]}] (\from)--(\to);
 \foreach \from/\to in {52/41p,52/43p,54p/43p,54p/45p,56/45p,56/47p}
 \draw[arrows={-Stealth[scale=1.5]}] (\from)--(\to);
 \end{tikzpicture}
\]
if $\level$ is even, and
\[
 \begin{tikzpicture}
 [scale=0.70,auto=left,black_vertex/.style={circle,draw,fill,scale=0.75},white_vertex/.style={circle,draw,scale=0.75}]
 \draw [decorate,decoration={brace,amplitude=8pt}]
 (0,1) -- (0,5) node [black,midway,xshift=-0.4cm] {$2\level-1$};
 \draw [decorate,decoration={brace,amplitude=8pt,mirror}]
 (0.5,0.5) -- (8,0.5) node [black,midway,yshift=-0.9cm] {$r-1$};

 \node at (1.2,1) [scale=1.5] {$\cdots$};
 \node at (1.2,2) [scale=1.5] {$\cdots$};
 \node at (1.2,3) [scale=1.5] {$\cdots$};
 \node at (1.2,4) [scale=1.5] {$\cdots$};
 \node at (1.2,5) [scale=1.5] {$\cdots$};

 \node (11) at (2,1) [black_vertex]{};
 \node at (2.3,1.3) []{$-$};
 \node (12) at (2,2) [black_vertex]{};
 \node at (2.3,2.3) []{$+$};
 \node (13) at (2,3) [black_vertex]{};
 \node at (2.3,3.3) []{$-$};
 \node (14) at (2,4) [black_vertex]{};
 \node at (2.3,4.3) []{$+$};
 \node (15) at (2,5) [black_vertex]{};
 \node at (2.3,5.3) []{$-$};
 \node (21) at (4,1) [black_vertex]{};
 \node at (4.3,1.3) []{$+$};
 \node (22) at (4,2) [black_vertex]{};
 \node at (4.3,2.3) []{$-$};
 \node (23) at (4,3) [black_vertex]{};
 \node at (4.3,3.3) []{$+$};
 \node (24) at (4,4) [black_vertex]{};
 \node at (4.3,4.3) []{$-$};
 \node (25) at (4,5) [black_vertex]{};
 \node at (4.3,5.3) []{$+$};
 \node (31) at (6,1) [black_vertex]{};
 \node at (6.3,1.3) []{$-$};
 \node (32) at (6,2) [black_vertex]{};
 \node at (6.3,2.3) []{$+$};
 \node (33) at (6,3) [black_vertex]{};
 \node at (6.3,3.3) []{$-$};
 \node (34) at (6,4) [black_vertex]{};
 \node at (6.3,4.3) []{$+$};
 \node (35) at (6,5) [black_vertex]{};
 \node at (6.3,5.3) []{$-$};
 \node (41) at (8,1) [black_vertex]{};
 \node at (8.3,1.3) []{$+$};
 \node (42) at (8,2) [black_vertex]{};
 \node at (8.3,2.3) []{$-$};
 \node (43) at (8,3) [black_vertex]{};
 \node at (8.3,3.3) []{$+$};
 \node (44) at (8,4) [black_vertex]{};
 \node at (8.3,4.3) []{$-$};
 \node (45) at (8,5) [black_vertex]{};
 \node at (8.3,5.3) []{$+$};

 \draw [decorate,decoration={brace,amplitude=8pt}]
 (11.2,2) -- (11.2,4) node [black,midway,xshift=-0.4cm] {$\level-1$};

 \node (41p) at (14,1) [black_vertex]{};
 \node at (14.3,1.3) []{$+$};
 \node (42p) at (14,2) [black_vertex]{};
 \node at (14.3,2.3) []{$-$};
 \node (43p) at (14,3) [black_vertex]{};
 \node at (14.3,3.3) []{$+$};
 \node (44p) at (14,4) [black_vertex]{};
 \node at (14.3,4.3) []{$-$};
 \node (45p) at (14,5) [black_vertex]{};
 \node at (14.3,5.3) []{$+$};

 \node (52) at (12,2) [white_vertex]{};
 \node at (12.3,2.3) []{$-$};
 \node (54) at (12,4) [white_vertex]{};
 \node at (12.3,4.3) []{$+$};
 \node (52p) at (16,2) [white_vertex]{};
 \node at (16.3,2.3) []{$+$};
 \node (54p) at (16,4) [white_vertex]{};
 \node at (16.3,4.3) []{$-$};

 \foreach \from/\to in {12/11,
 12/13,
 14/13,
 14/15,
 21/22,
 23/22,
 23/24,
 25/24,
 32/31,
 32/33,
 34/33,
 34/35,
 41/42,
 43/42,
 43/44,
 45/44,
 11/21,
 13/23,
 15/25,
 22/12,
 22/32,
 24/14,
 24/34,
 31/21,
 31/41,
 33/23,
 33/43,
 35/25,
 35/45,
 42/32,
 44/34}
 \draw[arrows={-Stealth[scale=1.5]}] (\from)--(\to);
 \foreach \from/\to in {41p/42p,43p/42p,43p/44p,45p/44p}
 \draw[arrows={-Stealth[scale=1.5]}] (\from)--(\to);
 \foreach \from/\to in {54/52,54p/52p}
 \draw[arrows={-Stealth[scale=1.5]}] (\from)--(\to);
 \foreach \from/\to in {42p/52,42p/52p,44p/54,44p/54p}
 \draw[arrows={-Stealth[scale=1.5]}] (\from)--(\to);
 \foreach \from/\to in {52/41p,52/43p,54p/43p,54p/45p}
 \draw[arrows={-Stealth[scale=1.5]}] (\from)--(\to);
 \end{tikzpicture}
\]
if $\level$ is odd,
where we identify the rightmost columns in the left quivers with the center columns in the right quivers.

We define $Q(F_4, \level)$ as
\[
 \begin{tikzpicture}
 [scale=0.70,auto=left,black_vertex/.style={circle,draw,fill,scale=0.75},white_vertex/.style={circle,draw,scale=0.75}]
 \draw [decorate,decoration={brace,amplitude=8pt}]
 (1,1) -- (1,7) node [black,midway,xshift=-0.4cm] {$2\level-1$};

 \node (11) at (2,1) [black_vertex]{};
 \node at (2.3,1.3) []{$-$};
 \node (12) at (2,2) [black_vertex]{};
 \node at (2.3,2.3) []{$+$};
 \node (13) at (2,3) [black_vertex]{};
 \node at (2.3,3.3) []{$-$};
 \node (14) at (2,4) [black_vertex]{};
 \node at (2.3,4.3) []{$+$};
 \node (15) at (2,5) [black_vertex]{};
 \node at (2.3,5.3) []{$-$};
 \node (16) at (2,6) [black_vertex]{};
 \node at (2.3,6.3) []{$+$};
 \node (17) at (2,7) [black_vertex]{};
 \node at (2.3,7.3) []{$-$};
 \node (21) at (4,1) [black_vertex]{};
 \node at (4.3,1.3) []{$+$};
 \node (22) at (4,2) [black_vertex]{};
 \node at (4.3,2.3) []{$-$};
 \node (23) at (4,3) [black_vertex]{};
 \node at (4.3,3.3) []{$+$};
 \node (24) at (4,4) [black_vertex]{};
 \node at (4.3,4.3) []{$-$};
 \node (25) at (4,5) [black_vertex]{};
 \node at (4.3,5.3) []{$+$};
 \node (26) at (4,6) [black_vertex]{};
 \node at (4.3,6.3) []{$-$};
 \node (27) at (4,7) [black_vertex]{};
 \node at (4.3,7.3) []{$+$};

 \draw [decorate,decoration={brace,amplitude=8pt}]
 (8.2,2) -- (8.2,6) node [black,midway,xshift=-0.4cm] {$\level-1$};

 \node (41p) at (13,1) [black_vertex]{};
 \node at (13.3,1.3) []{$+$};
 \node (42p) at (13,2) [black_vertex]{};
 \node at (13.3,2.3) []{$-$};
 \node (43p) at (13,3) [black_vertex]{};
 \node at (13.3,3.3) []{$+$};
 \node (44p) at (13,4) [black_vertex]{};
 \node at (13.3,4.3) []{$-$};
 \node (45p) at (13,5) [black_vertex]{};
 \node at (13.3,5.3) []{$+$};
 \node (46p) at (13,6) [black_vertex]{};
 \node at (13.3,6.3) []{$-$};
 \node (47p) at (13,7) [black_vertex]{};
 \node at (13.3,7.3) []{$+$};

 \node (52) at (11,2) [white_vertex]{};
 \node at (11.3,2.3) []{$-$};
 \node (54) at (11,4) [white_vertex]{};
 \node at (11.3,4.3) []{$+$};
 \node (56) at (11,6) [white_vertex]{};
 \node at (11.3,6.3) []{$-$};
 \node (52p) at (15,2) [white_vertex]{};
 \node at (15.3,2.3) []{$+$};
 \node (54p) at (15,4) [white_vertex]{};
 \node at (15.3,4.3) []{$-$};
 \node (56p) at (15,6) [white_vertex]{};
 \node at (15.3,6.3) []{$+$};

 \node (62) at (9,2) [white_vertex]{};
 \node at (9.3,2.3) []{$+$};
 \node (64) at (9,4) [white_vertex]{};
 \node at (9.3,4.3) []{$-$};
 \node (66) at (9,6) [white_vertex]{};
 \node at (9.3,6.3) []{$+$};
 \node (62p) at (17,2) [white_vertex]{};
 \node at (17.3,2.3) []{$-$};
 \node (64p) at (17,4) [white_vertex]{};
 \node at (17.3,4.3) []{$+$};
 \node (66p) at (17,6) [white_vertex]{};
 \node at (17.3,6.3) []{$-$};

 \foreach \from/\to in {12/11,
 12/13,
 14/13,
 14/15,
 16/15,
 16/17,
 21/22,
 23/22,
 23/24,
 25/24,
 25/26,
 27/26,
 11/21,
 13/23,
 15/25,
 17/27,
 22/12,
 24/14,
 26/16}
 \draw[arrows={-Stealth[scale=1.5]}] (\from)--(\to);
 \foreach \from/\to in {41p/42p,43p/42p,43p/44p,45p/44p,45p/46p,47p/46p}
 \draw[arrows={-Stealth[scale=1.5]}] (\from)--(\to);
 \foreach \from/\to in {54/52,54/56,54p/52p,54p/56p}
 \draw[arrows={-Stealth[scale=1.5]}] (\from)--(\to);
 \foreach \from/\to in {62/64,66/64,62p/64p,66p/64p}
 \draw[arrows={-Stealth[scale=1.5]}] (\from)--(\to);
 \foreach \from/\to in {52/62,64/54,56/66,62p/52p,54p/64p,66p/56p}
 \draw[arrows={-Stealth[scale=1.5]}] (\from)--(\to);
 \foreach \from/\to in {42p/52,42p/52p,44p/54,44p/54p,46p/56,46p/56p}
 \draw[arrows={-Stealth[scale=1.5]}] (\from)--(\to);
 \foreach \from/\to in {52/41p,52/43p,54p/43p,54p/45p,56/45p,56/47p}
 \draw[arrows={-Stealth[scale=1.5]}] (\from)--(\to);
 \end{tikzpicture}
\]
if $\level$ is even, and
\[
 \begin{tikzpicture}
 [scale=0.70,auto=left,black_vertex/.style={circle,draw,fill,scale=0.75},white_vertex/.style={circle,draw,scale=0.75}]
 \draw [decorate,decoration={brace,amplitude=8pt}]
 (1,1) -- (1,5) node [black,midway,xshift=-0.4cm] {$2\level-1$};

 \node (11) at (2,1) [black_vertex]{};
 \node at (2.3,1.3) []{$-$};
 \node (12) at (2,2) [black_vertex]{};
 \node at (2.3,2.3) []{$+$};
 \node (13) at (2,3) [black_vertex]{};
 \node at (2.3,3.3) []{$-$};
 \node (14) at (2,4) [black_vertex]{};
 \node at (2.3,4.3) []{$+$};
 \node (15) at (2,5) [black_vertex]{};
 \node at (2.3,5.3) []{$-$};

 \node (21) at (4,1) [black_vertex]{};
 \node at (4.3,1.3) []{$+$};
 \node (22) at (4,2) [black_vertex]{};
 \node at (4.3,2.3) []{$-$};
 \node (23) at (4,3) [black_vertex]{};
 \node at (4.3,3.3) []{$+$};
 \node (24) at (4,4) [black_vertex]{};
 \node at (4.3,4.3) []{$-$};
 \node (25) at (4,5) [black_vertex]{};
 \node at (4.3,5.3) []{$+$};

 \draw [decorate,decoration={brace,amplitude=8pt}]
 (8.2,2) -- (8.2,4) node [black,midway,xshift=-0.4cm] {$\level-1$};

 \node (41p) at (13,1) [black_vertex]{};
 \node at (13.3,1.3) []{$+$};
 \node (42p) at (13,2) [black_vertex]{};
 \node at (13.3,2.3) []{$-$};
 \node (43p) at (13,3) [black_vertex]{};
 \node at (13.3,3.3) []{$+$};
 \node (44p) at (13,4) [black_vertex]{};
 \node at (13.3,4.3) []{$-$};
 \node (45p) at (13,5) [black_vertex]{};
 \node at (13.3,5.3) []{$+$};

 \node (52) at (11,2) [white_vertex]{};
 \node at (11.3,2.3) []{$-$};
 \node (54) at (11,4) [white_vertex]{};
 \node at (11.3,4.3) []{$+$};
 \node (52p) at (15,2) [white_vertex]{};
 \node at (15.3,2.3) []{$+$};
 \node (54p) at (15,4) [white_vertex]{};
 \node at (15.3,4.3) []{$-$};

 \node (62) at (9,2) [white_vertex]{};
 \node at (9.3,2.3) []{$+$};
 \node (64) at (9,4) [white_vertex]{};
 \node at (9.3,4.3) []{$-$};
 \node (62p) at (17,2) [white_vertex]{};
 \node at (17.3,2.3) []{$-$};
 \node (64p) at (17,4) [white_vertex]{};
 \node at (17.3,4.3) []{$+$};

 \foreach \from/\to in {12/11,
 12/13,
 14/13,
 14/15,
 21/22,
 23/22,
 23/24,
 25/24,
 11/21,
 13/23,
 15/25,
 22/12,
 24/14}
 \draw[arrows={-Stealth[scale=1.5]}] (\from)--(\to);
 \foreach \from/\to in {41p/42p,43p/42p,43p/44p,45p/44p}
 \draw[arrows={-Stealth[scale=1.5]}] (\from)--(\to);
 \foreach \from/\to in {54/52,54p/52p}
 \draw[arrows={-Stealth[scale=1.5]}] (\from)--(\to);
 \foreach \from/\to in {62/64,62p/64p}
 \draw[arrows={-Stealth[scale=1.5]}] (\from)--(\to);
 \foreach \from/\to in {52/62,64/54,62p/52p,54p/64p}
 \draw[arrows={-Stealth[scale=1.5]}] (\from)--(\to);
 \foreach \from/\to in {42p/52,42p/52p,44p/54,44p/54p}
 \draw[arrows={-Stealth[scale=1.5]}] (\from)--(\to);
 \foreach \from/\to in {52/41p,52/43p,54p/43p,54p/45p}
 \draw[arrows={-Stealth[scale=1.5]}] (\from)--(\to);
 \end{tikzpicture}
\]
if $\level$ is odd,
where we identify the rightmost columns in the left quivers with the center columns in the right quivers.

\subsubsection*{Type $\boldsymbol{G}$}
Finally, we assume that $X_r = G_2$. We define $Q(G_2, \level)$ as
\[
 \begin{tikzpicture}
 [scale=0.70,auto=left,black_vertex/.style={circle,draw,fill,scale=0.75},white_vertex/.style={circle,draw,scale=0.75}]
 \draw [decorate,decoration={brace,amplitude=8pt}]
 (1,3) -- (1,9) node [black,midway,xshift=-0.4cm] {$\level-1$};


 \node (11) at (2,3) [white_vertex]{};
 \node at (1.6,3.4) []{$\mathrm{IV}$};
 \node (12) at (2,6) [white_vertex]{};
 \node at (1.6,6.4) []{$\mathrm{I}$};
 \node (13) at (2,9) [white_vertex]{};
 \node at (1.6,9.4) []{$\mathrm{IV}$};
 \node (21) at (4,1) [black_vertex]{};
 \node at (4.3,1.3) []{$+$};
 \node (22) at (4,2) [black_vertex]{};
 \node at (4.3,2.3) []{$-$};
 \node (23) at (4,3) [black_vertex]{};
 \node at (4.3,3.3) []{$+$};
 \node (24) at (4,4) [black_vertex]{};
 \node at (4.3,4.3) []{$-$};
 \node (25) at (4,5) [black_vertex]{};
 \node at (4.3,5.3) []{$+$};
 \node (26) at (4,6) [black_vertex]{};
 \node at (4.3,6.3) []{$-$};
 \node (27) at (4,7) [black_vertex]{};
 \node at (4.3,7.3) []{$+$};
 \node (28) at (4,8) [black_vertex]{};
 \node at (4.3,8.3) []{$-$};
 \node (29) at (4,9) [black_vertex]{};
 \node at (4.3,9.3) []{$+$};
 \node (210) at (4,10) [black_vertex]{};
 \node at (4.3,10.3) []{$-$};
 \node (211) at (4,11) [black_vertex]{};
 \node at (4.3,11.3) []{$+$};

 \foreach \from/\to in {12/11,
 12/13,
 21/22,
 21/24,
 23/22,
 23/24,
 23/26,
 25/22,
 25/24,
 25/26,
 25/28,
 27/24,
 27/26,
 27/28,
 27/210,
 29/26,
 29/28,
 29/210,
 211/28,
 211/210,
 11/23,
 13/29,
 26/12}
 \draw[arrows={-Stealth[scale=1.5]}] (\from)--(\to);
 \foreach \from/\to in {11/21,11/25,13/27,13/211,22/11,24/11,28/13,210/13}
 \draw[arrows={-Stealth[scale=1.5]}] (\from)--(\to);

 \node (11p) at (7,3) [white_vertex]{};
 \node at (6.6,3.4) []{$\mathrm{II}$};
 \node (12p) at (7,6) [white_vertex]{};
 \node at (6.6,6.4) []{$\mathrm{V}$};
 \node (13p) at (7,9) [white_vertex]{};
 \node at (6.6,9.4) []{$\mathrm{II}$};
 \node (21p) at (9,1) [black_vertex]{};
 \node at (9.3,1.3) []{$+$};
 \node (22p) at (9,2) [black_vertex]{};
 \node at (9.3,2.3) []{$-$};
 \node (23p) at (9,3) [black_vertex]{};
 \node at (9.3,3.3) []{$+$};
 \node (24p) at (9,4) [black_vertex]{};
 \node at (9.3,4.3) []{$-$};
 \node (25p) at (9,5) [black_vertex]{};
 \node at (9.3,5.3) []{$+$};
 \node (26p) at (9,6) [black_vertex]{};
 \node at (9.3,6.3) []{$-$};
 \node (27p) at (9,7) [black_vertex]{};
 \node at (9.3,7.3) []{$+$};
 \node (28p) at (9,8) [black_vertex]{};
 \node at (9.3,8.3) []{$-$};
 \node (29p) at (9,9) [black_vertex]{};
 \node at (9.3,9.3) []{$+$};
 \node (210p) at (9,10) [black_vertex]{};
 \node at (9.3,10.3) []{$-$};
 \node (211p) at (9,11) [black_vertex]{};
 \node at (9.3,11.3) []{$+$};

 \foreach \from/\to in {11p/12p,
 13p/12p,
 21p/22p,
 21p/24p,
 23p/22p,
 23p/24p,
 23p/26p,
 25p/22p,
 25p/24p,
 25p/26p,
 25p/28p,
 27p/24p,
 27p/26p,
 27p/28p,
 27p/210p,
 29p/26p,
 29p/28p,
 29p/210p,
 211p/28p,
 211p/210p,
 11p/23p,
 13p/29p,
 26p/12p}
 \draw[arrows={-Stealth[scale=1.5]}] (\from)--(\to);
 \foreach \from/\to in {12p/25p,12p/27p,22p/11p,24p/11p,28p/13p,210p/13p}
 \draw[arrows={-Stealth[scale=1.5]}] (\from)--(\to);

 \node (11pp) at (12,3) [white_vertex]{};
 \node at (11.6,3.4) []{$\mathrm{VI}$};
 \node (12pp) at (12,6) [white_vertex]{};
 \node at (11.6,6.4) []{$\mathrm{III}$};
 \node (13pp) at (12,9) [white_vertex]{};
 \node at (11.6,9.4) []{$\mathrm{VI}$};
 \node (21pp) at (14,1) [black_vertex]{};
 \node at (14.3,1.3) []{$+$};
 \node (22pp) at (14,2) [black_vertex]{};
 \node at (14.3,2.3) []{$-$};
 \node (23pp) at (14,3) [black_vertex]{};
 \node at (14.3,3.3) []{$+$};
 \node (24pp) at (14,4) [black_vertex]{};
 \node at (14.3,4.3) []{$-$};
 \node (25pp) at (14,5) [black_vertex]{};
 \node at (14.3,5.3) []{$+$};
 \node (26pp) at (14,6) [black_vertex]{};
 \node at (14.3,6.3) []{$-$};
 \node (27pp) at (14,7) [black_vertex]{};
 \node at (14.3,7.3) []{$+$};
 \node (28pp) at (14,8) [black_vertex]{};
 \node at (14.3,8.3) []{$-$};
 \node (29pp) at (14,9) [black_vertex]{};
 \node at (14.3,9.3) []{$+$};
 \node (210pp) at (14,10) [black_vertex]{};
 \node at (14.3,10.3) []{$-$};
 \node (211pp) at (14,11) [black_vertex]{};
 \node at (14.3,11.3) []{$+$};
 \draw [decorate,decoration={brace,amplitude=8pt,mirror}]
 (15,1) -- (15,11) node [black,midway,xshift=1.4cm] {$3\level-1$};

 \foreach \from/\to in {12pp/11pp,
 12pp/13pp,
 21pp/22pp,
 21pp/24pp,
 23pp/22pp,
 23pp/24pp,
 23pp/26pp,
 25pp/22pp,
 25pp/24pp,
 25pp/26pp,
 25pp/28pp,
 27pp/24pp,
 27pp/26pp,
 27pp/28pp,
 27pp/210pp,
 29pp/26pp,
 29pp/28pp,
 29pp/210pp,
 211pp/28pp,
 211pp/210pp,
 11pp/23pp,
 13pp/29pp,
 26pp/12pp}
 \draw[arrows={-Stealth[scale=1.5]}] (\from)--(\to);
 \foreach \from/\to in {24pp/12pp,28pp/12pp,12pp/25pp,12pp/27pp}
 \draw[arrows={-Stealth[scale=1.5]}] (\from)--(\to);
 \end{tikzpicture}
\]
if $\level$ is even, and
\[
 \begin{tikzpicture}
 [scale=0.70,auto=left,black_vertex/.style={circle,draw,fill,scale=0.75},white_vertex/.style={circle,draw,scale=0.75}]
 \draw [decorate,decoration={brace,amplitude=8pt}]
 (1,3) -- (1,6) node [black,midway,xshift=-0.4cm] {$\level-1$};


 \node (11) at (2,3) [white_vertex]{};
 \node at (1.6,3.4) []{$\mathrm{I}$};
 \node (12) at (2,6) [white_vertex]{};
 \node at (1.6,6.4) []{$\mathrm{IV}$};
 \node (21) at (4,1) [black_vertex]{};
 \node at (4.3,1.3) []{$+$};
 \node (22) at (4,2) [black_vertex]{};
 \node at (4.3,2.3) []{$-$};
 \node (23) at (4,3) [black_vertex]{};
 \node at (4.3,3.3) []{$+$};
 \node (24) at (4,4) [black_vertex]{};
 \node at (4.3,4.3) []{$-$};
 \node (25) at (4,5) [black_vertex]{};
 \node at (4.3,5.3) []{$+$};
 \node (26) at (4,6) [black_vertex]{};
 \node at (4.3,6.3) []{$-$};
 \node (27) at (4,7) [black_vertex]{};
 \node at (4.3,7.3) []{$+$};
 \node (28) at (4,8) [black_vertex]{};
 \node at (4.3,8.3) []{$-$};

 \foreach \from/\to in {12/11,
 21/22,
 21/24,
 23/22,
 23/24,
 23/26,
 25/22,
 25/24,
 25/26,
 25/28,
 27/24,
 27/26,
 27/28,
 11/23}
 \draw[arrows={-Stealth[scale=1.5]}] (\from)--(\to);
 \foreach \from/\to in {11/21,11/25,22/11,24/11,26/12}
 \draw[arrows={-Stealth[scale=1.5]}] (\from)--(\to);

 \node (11p) at (7,3) [white_vertex]{};
 \node at (6.6,3.4) []{$\mathrm{V}$};
 \node (12p) at (7,6) [white_vertex]{};
 \node at (6.6,6.4) []{$\mathrm{II}$};
 \node (21p) at (9,1) [black_vertex]{};
 \node at (9.3,1.3) []{$+$};
 \node (22p) at (9,2) [black_vertex]{};
 \node at (9.3,2.3) []{$-$};
 \node (23p) at (9,3) [black_vertex]{};
 \node at (9.3,3.3) []{$+$};
 \node (24p) at (9,4) [black_vertex]{};
 \node at (9.3,4.3) []{$-$};
 \node (25p) at (9,5) [black_vertex]{};
 \node at (9.3,5.3) []{$+$};
 \node (26p) at (9,6) [black_vertex]{};
 \node at (9.3,6.3) []{$-$};
 \node (27p) at (9,7) [black_vertex]{};
 \node at (9.3,7.3) []{$+$};
 \node (28p) at (9,8) [black_vertex]{};
 \node at (9.3,8.3) []{$-$};

 \foreach \from/\to in {11p/12p,
 21p/22p,
 21p/24p,
 23p/22p,
 23p/24p,
 23p/26p,
 25p/22p,
 25p/24p,
 25p/26p,
 25p/28p,
 27p/24p,
 27p/26p,
 27p/28p,
 11p/23p,
 26p/12p}
 \draw[arrows={-Stealth[scale=1.5]}] (\from)--(\to);
 \foreach \from/\to in {12p/25p,12p/27p,22p/11p,24p/11p}
 \draw[arrows={-Stealth[scale=1.5]}] (\from)--(\to);

 \node (11pp) at (12,3) [white_vertex]{};
 \node at (11.6,3.4) []{$\mathrm{III}$};
 \node (12pp) at (12,6) [white_vertex]{};
 \node at (11.6,6.4) []{$\mathrm{VI}$};
 \node (21pp) at (14,1) [black_vertex]{};
 \node at (14.3,1.3) []{$+$};
 \node (22pp) at (14,2) [black_vertex]{};
 \node at (14.3,2.3) []{$-$};
 \node (23pp) at (14,3) [black_vertex]{};
 \node at (14.3,3.3) []{$+$};
 \node (24pp) at (14,4) [black_vertex]{};
 \node at (14.3,4.3) []{$-$};
 \node (25pp) at (14,5) [black_vertex]{};
 \node at (14.3,5.3) []{$+$};
 \node (26pp) at (14,6) [black_vertex]{};
 \node at (14.3,6.3) []{$-$};
 \node (27pp) at (14,7) [black_vertex]{};
 \node at (14.3,7.3) []{$+$};
 \node (28pp) at (14,8) [black_vertex]{};
 \node at (14.3,8.3) []{$-$};
 \draw [decorate,decoration={brace,amplitude=8pt,mirror}]
 (15,1) -- (15,8) node [black,midway,xshift=1.4cm] {$3\level-1$};

 \foreach \from/\to in {21pp/22pp,
 21pp/24pp,
 23pp/22pp,
 23pp/24pp,
 23pp/26pp,
 25pp/22pp,
 25pp/24pp,
 25pp/26pp,
 25pp/28pp,
 27pp/24pp,
 27pp/26pp,
 27pp/28pp,
 11pp/23pp,
 26pp/12pp}
 \draw[arrows={-Stealth[scale=1.5]}] (\from)--(\to);
 \foreach \from/\to in {24pp/12pp,28pp/12pp,12pp/25pp,12pp/27pp,12pp/11pp}
 \draw[arrows={-Stealth[scale=1.5]}] (\from)--(\to);
 \end{tikzpicture}
\]
if $\level$ is odd,
where we identify the three columns consisting of filled vertices.

\subsection[Mutation loop on $Q(X_r,\ell)$]{Mutation loop on $\boldsymbol{Q(X_r,\ell)}$}\label{section: mutation loop Xr l}
In this section, we define a mutation loop $\gamma(X_r, \level)$ on the quiver $Q(X_r , \level)$.

\subsubsection*{Type $\boldsymbol{ADE}$}
First, we assume that $X_r$ is simply laced, that is, $X = A, D$ or $E$.
Let $\mathbf{I}$ be the set of the vertices in $Q(X_r , \level)$, and $\mathbf{I}_+$ and $\mathbf{I}_-$ be the subsets of $\mathbf{I}$ consisting of the vertices labeled with $+$ and $-$, respectively.
We define the following two compositions of mutations:
\begin{gather*}
 \mu_{+} = \prod_{m \in \mathbf{I}_+} \mu_m, \qquad
 \mu_{-} = \prod_{m \in \mathbf{I}_-} \mu_m.
\end{gather*}

We can easily verify the following lemma.
\begin{Lemma}
 Let $Q=Q(X_r , \level)$.
 Consider the following transitions of quivers:
 \begin{gather*}
 Q \xrightarrow{\mu_+} Q' \xrightarrow{\mu_-} Q''.
 \end{gather*}
 Then $Q'$ and $Q''$ do not depend on orders of the mutations in $\mu_+$ and $\mu_-$.
 Furthermore, we have $Q=Q''$.
\end{Lemma}

Thus
\begin{gather*}
 \gamma(X_r ,\level) := ( (\mathbf{I}_+ , \mathbf{I}_-) , \id)
\end{gather*}
is a mutation loop, where $\mathbf{I}_+$ and $\mathbf{I}_-$ in this equation mean any sequence in which all the elements in $\mathbf{I}_+$ and $\mathbf{I}_-$, respectively, appear exactly once.

\subsubsection*{Type $\boldsymbol{BCF}$}
Next, we assume that $X = B, C$ or $F$.
Let $\mathbf{I}$ be the set of the vertices in $Q(X_r , \level)$, and $\mathbf{I}_{+}^{\fillvertex}$, $\mathbf{I}_{-}^{\fillvertex}$, $\mathbf{I}_{+}^{\empvertex}$ and $\mathbf{I}_{-}^{\empvertex}$ be the subsets of $\mathbf{I}$ consisting of the vertices labeled with $(+, \fillvertex) , (-,\fillvertex) , (+,\empvertex)$ and $(-,\empvertex)$, respectively.
We define the following two compositions of mutations:
\begin{gather*}
\mu_{+}^{\fillvertex}\mu_{+}^{\empvertex}= \prod_{m \in \mathbf{I}_{+}^{\fillvertex} \sqcup \mathbf{I}_{+}^{\empvertex}} \mu_m, \qquad
\mu_{-}^{\fillvertex} = \prod_{m \in \mathbf{I}_{-}^{\fillvertex}} \mu_m.
\end{gather*}

Let $\nu$ be the permutation of vertices in $Q(X_r ,\level)$ that is identity on the filled vertices, and left-right reflection on the unfilled vertices.

\begin{Lemma}[\cite{IIKKNa,IIKKNb}]
 Let $Q=Q(X_r , \level)$.
 Consider the following transitions of quivers:
 \begin{gather*}
 Q \xrightarrow{\mu_{+}^{\fillvertex} \mu_{+}^{\empvertex}} Q' \xrightarrow{\mu_{-}^{\fillvertex}} Q''.
 \end{gather*}
 Then $Q'$ and $Q''$ do not depend on orders of the mutations in $\mu_{+}^{\fillvertex}\mu_{+}^{\empvertex}$ and $\mu_{-}^{\fillvertex}$.
 Furthermore, we have $Q= \nu(Q'')$.
\end{Lemma}

Thus
\begin{gather*}
\gamma(X_r ,\level) := \big( \big(\mathbf{I}_{+}^{\fillvertex} \sqcup \mathbf{I}_{+}^{\empvertex} , \mathbf{I}_{-}^{\fillvertex}\big) , \nu\big)
\end{gather*}
is a mutation loop, where $\mathbf{I}_{+}^{\fillvertex} \sqcup \mathbf{I}_{+}^{\empvertex}$ and $\mathbf{I}_{-}^{\fillvertex}$ in this equation mean any sequence in which all the elements in $\mathbf{I}_{+}^{\fillvertex} \sqcup \mathbf{I}_{+}^{\empvertex}$ and $\mathbf{I}_{-}^{\fillvertex}$, respectively, appear exactly once.

\subsubsection*{Type $\boldsymbol{G}$}\label{section: type G mutation loop}
Next, we assume that $X=G$.
Let $\mathbf{I}$ be the set of the vertices in $Q(G_2 , \level)$.
Let $\mathbf{I}_{+}^{\fillvertex}$, $\mathbf{I}_{-}^{\fillvertex}$ be the subsets of $\mathbf{I}$ consisting of the vertices labeled with $(+, \fillvertex)$ and $(-,\fillvertex)$ respectively,
and let $\mathbf{I}_{\mathrm{I}}^{\empvertex}, \dots , \mathbf{I}_{\mathrm{VI}}^{\empvertex}$ be the subsets of $\mathbf{I}$ consisting of the vertices labeled with $(\mathrm{I},\empvertex) ,\dots , (\mathrm{VI},\empvertex)$, respectively.
We define the following two compositions of mutations:
\begin{gather*}
\mu_{+}^{\fillvertex} \mu_{\mathrm{I}}^{\empvertex}= \prod_{m \in \mathbf{I}_{+}^{\fillvertex} \sqcup \mathbf{I}_{\mathrm{I}}^{\empvertex}} \mu_m, \qquad
\mu_{-}^{\fillvertex} \mu_{\mathrm{II}}^{\empvertex} = \prod_{m \in \mathbf{I}_{-}^{\fillvertex} \sqcup \mathbf{I}_{\mathrm{II}}^{\empvertex}} \mu_m.
\end{gather*}

We choose indices of vertices in $Q(G_2, \level)$ to be
\begin{gather*}
 \{ (a, i) \,|\, a=1,2, 3, i = 1, \dots , \level-1 \} \cup
 \{ (4, i) \,|\, i = 1, \dots , 3\level -1\} ,
\end{gather*}
so that $(a,i)$ is the $i$-th row (from the bottom) vertex in the $a$-th quiver (from the left).
Let $\nu$ be the permutation of the vertices in $Q(G_2,\level)$ that is the identity on the filled vertices, that is, $\nu(4,i)=(4,i)$ for all $i$, and acts on the unfilled vertices as
\begin{gather*}
 \nu(1,i) = (2,i), \\
 \nu(2,i) = (3,i), \\
 \nu(3,i) = (1,i),
\end{gather*}
for all $i$.

\begin{Lemma}[\cite{IIKKNb}]
 Let $Q=Q(G_2 , \level)$.
 Consider the following transitions of quivers:
 \begin{gather*}
 Q \xrightarrow{\mu_{+}^{\fillvertex} \mu_{\mathrm{I}}^{\empvertex}} Q' \xrightarrow{\mu_{-}^{\fillvertex} \mu_{\mathrm{II}}^{\empvertex}} Q''.
 \end{gather*}
 Then $Q'$ and $Q''$ do not depend on orders of the mutations in $\mu_{+}^{\fillvertex} \mu_{\mathrm{I}}^{\empvertex}$ and $\mu_{-}^{\fillvertex} \mu_{\mathrm{II}}^{\empvertex}$.
 Furthermore, we have $Q= \nu(Q'')$.
\end{Lemma}

Thus
\begin{gather*}
\gamma(G_2 ,\level) := \big( \big(\mathbf{I}_{+}^{\fillvertex} \sqcup \mathbf{I}_{\mathrm{I}}^{\empvertex} , \mathbf{I}_{-}^{\fillvertex} \sqcup \mathbf{I}_{\mathrm{II}}^{\empvertex}\big) , \nu\big)
\end{gather*}
is a mutation loop, where $\mathbf{I}_{+}^{\fillvertex} \sqcup \mathbf{I}_{\mathrm{I}}^{\empvertex}$ and $ \mathbf{I}_{-}^{\fillvertex} \sqcup \mathbf{I}_{\mathrm{II}}^{\empvertex}$ in this equation mean any sequence in which all the elements in $\mathbf{I}_{+}^{\fillvertex} \sqcup \mathbf{I}_{\mathrm{I}}^{\empvertex}$ and $ \mathbf{I}_{-}^{\fillvertex} \sqcup \mathbf{I}_{\mathrm{II}}^{\empvertex}$, respectively, appear exactly once.

\subsection{Periodicity}
As above, we have defined quiver $Q(X_r, \level)$ and the mutation loop $\gamma(X_r, \level)$ for any pair $(X_r ,\level)$.
The following lemma follows from these definitions.
\begin{Lemma}\label{lemma:gamma is regular}
 The mutation loop $\gamma(X_r, \level)$ is regular.
\end{Lemma}

A remarkable fact, as stated in the beginning of Section \ref{section: quiver Xr l}, is that these mutation loops have periodicity.
\begin{Theorem}[\cite{IIKKNa,IIKKNb, Keller}]\label{theorem: periodicity}
 Let $\gamma = \gamma (X_r ,\level)$. Then the following holds:
 \begin{gather*}
 \underbrace{\mu_\gamma \circ \mu_\gamma \circ \dots \circ \mu_\gamma}_{t(\level + \dcn)} (y) = y ,
 \end{gather*}
where $\dcn$ is the dual Coxeter number of $X_r$ $($see \eqref{eq: dual coxeter} for the list of dual Coxeter numbers$)$, and~$t$ is given by \eqref{eq: t list}.
\end{Theorem}

\subsection[Exponents of $Q(X_r, \ell)$]{Exponents of $\boldsymbol{Q(X_r,\ell)}$}\label{sec:def of exponents}

First, we see that the rational function $\mu_\gamma$ for the mutation loop $\gamma = \gamma(X_r , \level)$ satisfies the following property.
\begin{Proposition}\label{prop:unique positive real eta}
 The fixed point equation $\mu_\gamma (y) =y$ has a unique positive real solution.
\end{Proposition}
This lemma will be proved in Section~\ref{sec:asymptotics of Z} (see Corollary~\ref{corollary: K = Ap inv Am}).
This positive real solution is denoted by $\eta \in (\R_{>0})^{\mathbf{I} }$.

Let $J(y)$ be the Jacobian matrix of the rational function $\mu_\gamma$:
\begin{gather*}
 J_\gamma(y) := \left( \frac{\partial \mu_{i}}{\partial y_j} (y) \right)_{i,j \in \mathbf{I} },
\end{gather*}
where $\mu_i$ is the $i$-th components of $\mu_\gamma$.
Then Theorem~\ref{theorem: periodicity} implies that the $t\big(\level + \dcn\big)$-th power of $J_\gamma(\eta)$ is the identity matrix:
\begin{gather*}
 J_\gamma(\eta)^{t (\level + \dcn)} = I.
\end{gather*}
Thus all the eigenvalues of $J_\gamma(\eta)$ are $t\big(\level + \dcn\big)$-th root of unities.
These eigenvalues can be written as
\begin{gather*}
 {\rm e}^{ \frac{2 \pi{\rm i} m_1}{t (\level + \dcn)} } , {\rm e}^{ \frac{2 \pi{\rm i} m_2}{t (\level + \dcn)} } , \dots ,{\rm e}^{ \frac{2 \pi{\rm i} m_{\lvert \mathbf{I} \rvert}}{t (\level + \dcn)} } ,
\end{gather*}
using a sequence of integers $0 \leq m_1 \leq m_2 \leq \dots \leq m_{\lvert \mathbf{I} \rvert} < t\big(\level + \dcn\big)$.
We say that this sequence $m_1 , \dots , m_{\lvert \mathbf{I} \rvert}$ is the \emph{exponents} of $Q(X_r , \level)$.

\begin{Example}
 \begin{figure}[t]
 \centering
 \begin{tikzpicture}
 [scale=0.75,auto=left,black_vertex/.style={circle,draw,fill,scale=0.75},white_vertex/.style={circle,draw,scale=0.75}]
 \node at (3.5,4.6) []{$\mu_{+}$};
 \node at (3.5,4) [scale=1.75]{$\longrightarrow$};

 \node (11) at (2,2) [black_vertex][label=left:{$1$}]{};
 \node at (2.3,2.3) []{$-$};
 \node (12) at (2,4) [black_vertex][label=left:{$2$}]{};
 \node at (2.3,4.3) []{$+$};
 \node (13) at (2,6) [black_vertex][label=left:{$3$}]{};
 \node at (2.3,6.3) []{$-$};

 \foreach \from/\to in {12/11,
 12/13}
 \draw[arrows={-Stealth[scale=1.5]}] (\from)--(\to);
 \end{tikzpicture}
 \begin{tikzpicture}
 [scale=0.75,auto=left,black_vertex/.style={circle,draw,fill,scale=0.75},white_vertex/.style={circle,draw,scale=0.75}]
 \node at (3.5,4.6) []{$\mu_{-}$};
 \node at (3.5,4) [scale=1.75]{$\longrightarrow$};

 \node (11) at (2,2) [black_vertex][label=left:{$1$}]{};
 \node at (2.3,2.3) []{$-$};
 \node (12) at (2,4) [black_vertex][label=left:{$2$}]{};
 \node at (2.3,4.3) []{$+$};
 \node (13) at (2,6) [black_vertex][label=left:{$3$}]{};
 \node at (2.3,6.3) []{$-$};

 \foreach \from/\to in {11/12,
 13/12}
 \draw[arrows={-Stealth[scale=1.5]}] (\from)--(\to);
 \end{tikzpicture}
 \begin{tikzpicture}
 [scale=0.75,auto=left,black_vertex/.style={circle,draw,fill,scale=0.75},white_vertex/.style={circle,draw,scale=0.75}]

 \node (11) at (2,2) [black_vertex][label=left:{$1$}]{};
 \node at (2.3,2.3) []{$-$};
 \node (12) at (2,4) [black_vertex][label=left:{$2$}]{};
 \node at (2.3,4.3) []{$+$};
 \node (13) at (2,6) [black_vertex][label=left:{$3$}]{};
 \node at (2.3,6.3) []{$-$};

 \foreach \from/\to in {12/11,
 12/13}
 \draw[arrows={-Stealth[scale=1.5]}] (\from)--(\to);
 \end{tikzpicture}
 \caption{The mutation loop $\gamma(A_1, 4)$.}
 \label{fig:A1 mutation loop}
 \end{figure}
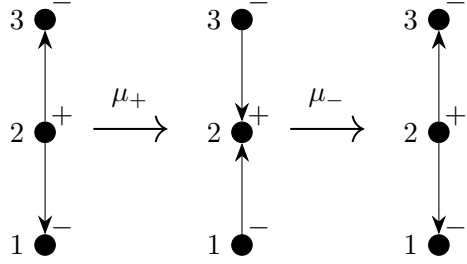
 Let $(X_r,\level) = (A_1,4)$.
 In this case, the index set of the quiver is given by $\mathbf{I} = I \times I' = \{ (1,1),(1,2),(1,3) \}$.
 We choose bipartite decompositions of $I = \{1\}$ and $I'=\{1,2,3\}$
 as $I = I_+ \sqcup I_- = \{\; \} \sqcup \{1\}$ and
 $I' = I'_+ \sqcup I'_- = \{1,3 \} \sqcup \{2 \}$, respectively.
 Then we obtain the decomposition $\mathbf{I} = \mathbf{I}_+ \sqcup \mathbf{I}_- = \{ (1,2) \} \sqcup \{ (1,1),(1,3) \}$,
 and the mutation loop $\gamma(A_1,4)$ is given in Fig.~\ref{fig:A1 mutation loop}, where we simply write $(1,i')$ as $i'$ in the figure.

 The cluster transformation is now given by
 \begin{gather*}
 \mu_1= y_{1}^{-1} y_{2}^{-1} (y_{2} + 1),\qquad \mu_2=
 y_{1} y_{2} y_{3} (y_{1} y_{2} + y_{2} + 1)^{-1} (y_{2} y_{3} + y_{2} + 1)^{-1} ,\\
 \mu_3=
 y_{2}^{-1} y_{3}^{-1} (y_{2} + 1) ,
 \end{gather*}
 where $\mu_i$ is the $i$-th component of $\mu_\gamma(y_1,y_2,y_3)$.
 We can directly verify that $\eta = (2,1/3,2)$ is the positive real solution of $\mu_i = y_i$ ($i=1,2,3$).

 Although the matrix $J_\gamma(y)$ is a little complicated to write down directly, we can see that the matrix $L(y)$ defined by $L(y) = \diag(\mu_1,\mu_2,\mu_3)^{-1}J_\gamma(y)\diag(y_1,y_2,y_3)$ can be written as a~product of two simple matrices as follows:
 \begin{gather*}
 L(y) =
 \begin{pmatrix}
 -1 & 0 & 0 \\
 \dfrac{y_{2} + 1}{y_{1} y_{2} + y_{2} + 1} & 1 & \dfrac{y_{2} + 1}{y_{2} y_{3} + y_{2} + 1} \\
 0 & 0 & -1
 \end{pmatrix}
 \begin{pmatrix}
 1 & \dfrac{1}{y_{2} + 1} & 0 \\
 0 & -1 & 0 \\
 0 & \dfrac{1}{y_{2} + 1} & 1
 \end{pmatrix}.
 \end{gather*}
 By substituting $\eta$, we obtain
 \begin{gather*}
 L(\eta) =
 \begin{pmatrix}
 -1 & 0 & 0 \\
 \frac{2}{3} & 1 & \frac{2}{3} \\
 0 & 0 & -1
 \end{pmatrix}
 \begin{pmatrix}
 1 & \frac{3}{4} & 0 \\
 0 & -1 & 0 \\
 0 & \frac{3}{4} & 1
 \end{pmatrix}
 =
 \begin{pmatrix}
 -1 & -\frac{3}{4} & 0 \\
 \frac{2}{3} & 0 & \frac{2}{3} \\
 0 & -\frac{3}{4} & -1
 \end{pmatrix},
 \end{gather*}
 and
 \begin{align*}
 \det(xI_3 - J_\gamma(\eta)) &= \det(xI_3 - L(\eta)) =x^3 + 2x^2 + 2x +1 \\
 &=
 \big(x-{\rm e}^{\frac{2 \pi{\rm i} \cdot 2}{6} }\big)\big(x-{\rm e}^{\frac{2 \pi{\rm i} \cdot 3}{6} }\big)\big(x-{\rm e}^{\frac{2 \pi{\rm i} \cdot 4}{6} }\big).
 \end{align*}
 Thus the exponents of $Q(A_1,4)$ are $2$, $3$, $4$.
\end{Example}

\subsection{Conjecture on exponents}
We give a conjectural formula on exponents in terms of root systems.
We define two polyno\-mials~$N_{X_r,\level}(x)$ and $D_{X_r, \level}(x)$ by
\begin{gather*}
 N_{X_r, \level}(x) = \prod_{a=1}^{r} \frac{ x^{t(\level + \dcn)} -1 }{ x^{t/t_a} -1 }, \qquad
 D_{X_r, \level}(x)=D_{X_r, \level}^{{\rm long}} (x) D_{X_r, \level}^{{\rm short}} (x),
\end{gather*}
where the polynomials $D_{X_r, \level}^{{\rm long}}(x)$ and $D_{X_r, \level}^{{\rm short}}(x)$ are defined by
\begin{gather*}
 D_{X_r, \level}^{{\rm long}}(x)=\prod_{\alpha \in \Delta^{{\rm long}} } { \big( x^t - {\rm e}^{ \frac{2 \pi{\rm i} \innerproduct{\rho}{\alpha}}{\level + \dcn} } \big) } ,
 \qquad
 D_{X_r, \level}^{{\rm short}} (x) = \prod_{\alpha \in \Delta^{{\rm short}}} { \big( x - {\rm e}^{\frac{2 \pi{\rm i} \innerproduct{\rho}{\alpha}}{\level + \dcn} } \big) }.
\end{gather*}

When $X=A ,D$ or $E$, the polynomials~$N_{X_r, \level}(x)$ and~$D_{X_r, \level}(x)$ can be written more simply:
\begin{gather*}
 N_{X_r, \level}(x) = \left( \frac{x^{\level + \dcn} -1 }{x -1} \right)^r, \qquad
 D_{X_r, \level}(x) = \prod_{\alpha \in \Delta} { \big( x - {\rm e}^{ \frac{2 \pi{\rm i} \innerproduct{\rho}{\alpha}}{\level + \dcn} } \big) } .
\end{gather*}

\begin{Conjecture}\label{conj: char poly}
 Let $X_r$ be a finite type Dynkin diagram, and $\level$ be a positive integer such that $\level \geq 2$.
 Let $\gamma=\gamma(X_r , \level)$ be the mutation loop on the quiver $Q(X_r , \level )$ defined in Section~{\rm \ref{section: mutation loop Xr l}}.
 Then the following identity holds for the characteristic polynomial of $J_{\gamma}(\eta)$:
 \begin{align}\label{eq:conj char poly}
 \det (x I - J_{\gamma} (\eta)) = \frac{N_{X_r, \level}(x)}{D_{X_r, \level}(x)}.
 \end{align}
\end{Conjecture}

\begin{Theorem}\label{theorem: A1 level}
 Conjecture~{\rm \ref{conj: char poly}} is true in the following cases:
 \begin{enumerate}\itemsep=0pt
 \item[$1)$] $(A_1, \level)$ for all $\level \geq 2$,
 \item[$2)$] $(A_r , 2)$ for all $r\geq 1$.
 \end{enumerate}
\end{Theorem}
We will prove Theorem \ref{theorem: A1 level} in Section \ref{section: A1 level}.

\subsection{Examples}\label{sec:example of exponents}
In this section, we give examples of calculations on the right-hand side in the conjectural formula~\eqref{eq:conj char poly}.

\begin{figure}[t] \centering
 \begin{tikzpicture}
 [scale=1.0,auto=left]
 \foreach \x in {1, 2, ..., 6}
 \foreach \y in {1, 2, ..., 3}
 \draw [black] (\x,\y) circle (5pt);
 \foreach \x / \y in {1/1, 1/2, 1/3, 2/1,2/2,3/1}
 \draw [black,fill] (\x,\y) circle (3pt);
 \foreach \x / \y in {6/1,6/2,6/3,5/3,5/2,4/3}
 \draw (\x,\y) node[minimum size =2pt,inner sep=2pt,diamond,draw,fill]{};
 \foreach \x in {1, 2, ..., 6}
 \node(\x) at (\x, 0.3) {$\x$};
 \end{tikzpicture}
 \caption{The underlying white circles represent the exponents of $N_{A_3, 3} (x)$, and the black marks represent the exponents of $D_{A_3, 3} (x)$.
 The number of vertices in a same vertical line is a multiplicity.
 Furthermore, the black circles and diamonds represent the exponents that come from the positive roots and the negative roots, respectively.
 The unmarked white circles represent the exponents of $N_{A_3, 3} (x)/D_{A_3, 3} (x)$.}
 \label{fig:exponents A3 3}
\end{figure}
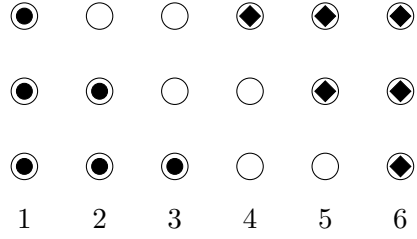

Let $F(x)$ be a polynomial whose roots are all $N$-th roots of unity.
The roots of $F(x)$ can be written as
\begin{gather*}
 {\rm e}^{2 \pi{\rm i} m_1 / N} ,{\rm e}^{2 \pi{\rm i} m_2/N} ,\dots, {\rm e}^{2 \pi{\rm i} m_n/N},
\end{gather*}
where $0 \leq m_1 \leq m_2 \leq \dots \leq m_n < N$ is a sequence of integers.
We call this sequence of integers the exponents of $F(x)$.
The exponents of $F(x)$ is denoted by $\mathcal{E}(F(x))$.

\begin{Example}
 Let $(X_r, \level) = (A_3, 3)$.
 The dual Coxeter number is given by $\dcn = 4 $, so $\level + \dcn = 7$.
 Therefore, the exponents of $N_{A_3, 3} (x)$ is given by
 \begin{gather*}
 \mathcal{E}( N_{A_3, 3} (x)) = ( 1,1,1,2,2,2,3,3,3,4,4,4,5,5,5,6,6,6 ).
 \end{gather*}
 On the other hand, by calculation on the root system of type $A_3$,
 we obtain
 \begin{gather*}
 \mathcal{E}( D_{A_3, 3} (x)) = (1,1,1,2,2,3,4,5,5,6,6,6).
 \end{gather*}
 As the result, the exponents of $N_{A_3, 3} (x)/D_{A_3,3} (x)$ are given by
 \begin{gather*}
 \mathcal{E} \left( \frac{N_{A_3, 3} (x)}{D_{A_3, 3} (x)} \right)
 = (2,3,3,4,4,5).
 \end{gather*}
 Fig.~\ref{fig:exponents A3 3} illustrates these exponents.
\end{Example}

\begin{Example}Let $(X_r, \level) = (B_3, 2)$. The dual Coxeter number is given by $\dcn = 5 $, so $t\big(\level + \dcn\big) = 14$.
 Therefore, the exponents of $N_{B_3, 2} (x)$ is given by
 \begin{gather*}
 \mathcal{E}( N_{B_3, 2} (x)) = ( 1,1,1,2,2,2,3,3,3,4,4,4,5,5,5,6,6,6,7,8,8,8,9,9,9,\\
 \hphantom{\mathcal{E}( N_{B_3, 2} (x)) = (}{}10,10,10,11,11,11,12,12,12,13,13,13 ).
 \end{gather*}
 On the other hand, by calculation on the root system of type $B_3$,
 we obtain
 \begin{gather*}
 \mathcal{E}\big( D_{B_3, 2}^{{\rm long}} (x)\big) = (1,1,2,2,3,3,4,4,5,5,6,6,8,8,9,9,10,10,11,11,12,12,13,13),\\
 \mathcal{E}\big( D_{B_3, 2}^{{\rm short}} (x)\big) = (1,3,5,9,11,13).
 \end{gather*}
 As the result, the exponents of $N_{B_3, 2} (x)/D_{B_3,2} (x)$ are given by
 \begin{align*}
 \mathcal{E} \left( \frac{N_{B_3, 2} (x)}{D_{B_3, 2} (x)} \right)
 = (2,4,6,7,8,10,12).
 \end{align*}
 Fig.~\ref{fig:exponents B3 2} illustrates these exponents.

 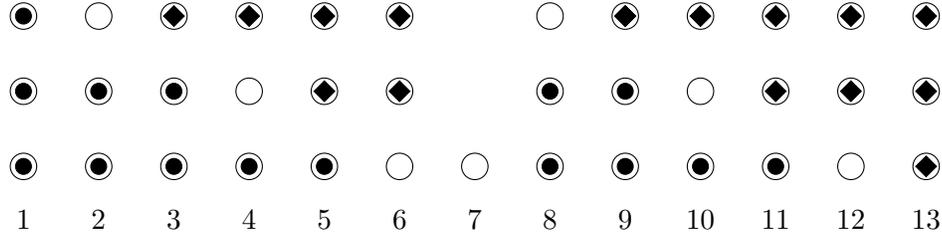
\begin{figure}[t]
 \centering
 \begin{tikzpicture}
 [scale=1.0,auto=left]
 \foreach \x in {1, 2, ..., 6}
 \foreach \y in {1, 2, ..., 3}
 \draw [black] (\x,\y) circle (5pt);
 \foreach \x in {8, 9, ..., 13}
 \foreach \y in {1, 2, ..., 3}
 \draw [black] (\x,\y) circle (5pt);
 \draw [black] (7,1) circle (5pt);
 \foreach \x / \y in {3/3,4/3,5/2,5/3,6/2,6/3, 9/3,10/3,11/2,12/2,12/3,13/1,13/2}
 \draw (\x,\y) node[minimum size =2pt,inner sep=2pt,diamond,draw,fill]{};
 \foreach \x / \y in {1/1, 1/2, 2/1, 2/2,3/1,4/1,8/1, 8/2, 9/1, 9/2,10/1,11/1}
 \draw [black,fill] (\x,\y) circle (3pt);
 \foreach \x / \y in {1/3, 3/2,5/1}
 \draw [black,fill] (\x,\y) circle (3pt);
 \foreach \x / \y in {13/3,11/3,9/3}
 \draw (\x,\y) node[minimum size =2pt,inner sep=2pt,diamond,draw,fill]{};
 \foreach \x in {1, 2, ..., 13}
 \node(\x) at (\x, 0.3) {$\x$};
 \end{tikzpicture}
 \caption{The exponents of $N_{B_3, 2}(x)$ and $D_{B_3, 2}(x)$. The meanings of the symbols are the same as in Fig.~\ref{fig:exponents A3 3}.}
 \label{fig:exponents B3 2}
 \end{figure}
\end{Example}

\section{Proofs of Theorem \ref{theorem: A1 level}}\label{sec:proofs for type A}
\subsection[$(A_1, \ell)$ case]{$\boldsymbol{(A_1, \ell)}$ case}\label{section: A1 level}
Let $\level$ be an integer with $\level \geq 2$. In this section we consider the mutation loop $\gamma = \gamma(A_1 ,\level)$.

Let $\mathbf{I} = \{ 1, \dots , \level -1 \}$ denotes the vertices in $Q(A_1, \level)$.
Let us fix a bipartite decomposition $\mathbf{I}=\mathbf{I}_+ \sqcup \mathbf{I}_-$.
For example, Fig.~\ref{fig:A1 mutation loop} shows the mutation loop $\gamma(A_1, 4)$.

\begin{Lemma} The right-hand side of \eqref{eq:conj char poly} is given by
 \begin{gather*}
 \frac{N_{A_1, \level}(x)}{D_{A_1, \level}(x)} =
 \prod_{a=2}^{\level} \big(x - {\rm e}^{\frac{2 \pi{\rm i} a}{\level + 2}}\big).
 \end{gather*}
\end{Lemma}
\begin{proof}First note that the dual Coxeter number of $A_1$ is given by $\dcn = 2$.
 Then, the lemma follows from
 \begin{gather*}
 N_{A_1, \level}(x) = \frac{x^{\level+2} -1 }{x-1},
 \end{gather*}
 and
 \begin{gather*}
 D_{A_1, \level}(x) = \big(x-{\rm e}^{\frac{2 \pi{\rm i}}{\level + 2}} \big) \big(x-{\rm e}^{\frac{-2 \pi{\rm i} }{\level + 2}} \big)
 =\big(x-{\rm e}^{\frac{2 \pi{\rm i}}{\level + 2}} \big) \big(x-{\rm e}^{\frac{ 2 \pi{\rm i} (\level+1)}{\level + 2}} \big).\tag*{\qed}
 \end{gather*}\renewcommand{\qed}{}
\end{proof}

Therefore, what we have to show is that the exponents of $Q(A_1, \level)$ are $2, 3, \dots, \level$.
To prove this, we define the matrix $L(y)$ by
\begin{align*}
 L(y) =
 \begin{pmatrix}
 \mu_1(y) & & \\
 & \ddots & \\
 & & \mu_{\level-1}(y)
 \end{pmatrix}^{-1}
 J_\gamma(y)
 \begin{pmatrix}
 y_1 & & \\
 & \ddots & \\
 & & y_{\level-1}
 \end{pmatrix},
\end{align*}
where $\mu_i$ is the $i$-th component of $\mu_\gamma$.
The matrices $J_\gamma(\eta)$ and $L(\eta)$ have the same characteristic polynomial because $\mu(\eta) = \eta$.
We will find the eigenvalues of $L(\eta)$ instead of $J_\gamma(\eta)$.

Let $\zeta={\rm e}^{\pi{\rm i} /(\level + 2)}$.
For $m=1, 2,\dots , \level-1$, we define non-zero real numbers
\begin{gather}\label{eq:def of zm}
z_m =
\frac{\sin \frac{\pi m }{\level + 2}\sin \frac{\pi (m+2) }{\level + 2}}{\sin^2 \frac{\pi (m+1)}{\level + 2}} =\frac{\big(\zeta^m-\zeta^{-m}\big)\big(\zeta^{m+2}-\zeta^{-m-2}\big)}{\big(\zeta^{m+1} - \zeta^{-m-1}\big)^2}.
\end{gather}
We define the two matrices $L_+ = \big( L_+^{mk} \big)_{m,k=1 , \dots , \level-1} $ and $L_- = \big( L_-^{mk} \big)_{m,k=1 , \dots , \level-1}$ by
\begin{gather*}
 L_\pm^{mk} =
 \begin{cases}
 \delta_{mk} & \text{if $k \in \mathbf{I}_{\mp}$,}\\
 -\delta_{mk} + z_k (\delta_{m,k-1} + \delta_{m,k+1}) & \text{if $k \in \mathbf{I}_{\pm}$,}
 \end{cases}
\end{gather*}
where $\delta$ is the Kronecker delta.
\begin{Lemma}\label{lemma:L eta = Lm Lp}
 The matrix $L(\eta)$ is expressed as
 \begin{gather*}
 L(\eta) = L_- L_+.
 \end{gather*}
\end{Lemma}
\begin{proof}
 Let
 \begin{gather*}
 (Q,y) \xrightarrow{\mu_+} (Q' ,y') \xrightarrow{\mu_-} (Q'',y'')
 \end{gather*}
 be the $Y$-seed transitions associated with $\gamma$.
 Then the chain rule for Jacobian matrices implies that
 \begin{gather*}
 L(y) = L_- (y') L_+(y),
 \end{gather*}
 where $L_\pm (Y)$ for $Y \in (\univsf)^{\level-1}$ is the $(\level-1) \times (\level -1)$ matrices defined by
 \begin{gather*}
 L_\pm^{mk} (Y) =
 \begin{cases}
 \delta_{mk} & \text{if $k \in \mathbf{I}_{\mp}$,} \\
 -\delta_{mk} + \dfrac{1}{1 + Y_k} (\delta_{m,k-1} + \delta_{m,k+1}) & \text{if $k \in \mathbf{I}_{\pm}$.}
 \end{cases}
 \end{gather*}

 Let $\tilde{\eta}$ be the $(\level-1)$-tuple of real numbers
 such that its $m$-th entry is $\eta$ if $m \in \mathbf{I}_+$ and is $y'|_{y=\eta}$ if $m \in \mathbf{I}_-$.
 The following explicit expression of $\tilde{\eta}$ is known (see \cite[Example 5.3]{KNS2011} for example):
 \begin{gather*}
 \tilde{\eta}_m =
 \frac{\sin^2 \frac{\pi}{\level + 2}}{\sin \frac{\pi m}{\level + 2}\sin \frac{\pi (m+2)}{\level + 2}} .
 \end{gather*}
 We can easily check that $z_m = 1/ (1+\tilde{\eta}_m)$, and this complete the proof.
\end{proof}

For any non-zero complex number $\lambda$, we consider the following difference equation of the numbers $( \phi_m)_{0 \leq m \leq \level}$:
\begin{gather}\label{eq:phi rec}
\phi_{m-1} + \phi_{m+1} = \phi_m z_m^{-1} \big(\lambda + \lambda^{-1}\big)
\end{gather}
with the boundary conditions given by
\begin{gather}\label{eq:phi boundary}
\phi_0=\phi_{\level} =0.
\end{gather}

\begin{Lemma}Let $(\phi_m)_{0 \leq m \leq \level}$ be a non-zero solution of the difference equation~\eqref{eq:phi rec} satisfying the boundary conditions~\eqref{eq:phi boundary}.
 Then the vector
 \begin{gather*}
 \psi=
 \begin{pmatrix}
 \psi_1 \\
 \vdots \\
 \psi_{\level-1}
 \end{pmatrix}
 \end{gather*}
 defined by
 \begin{gather*}
 \psi_m =
 \begin{cases}
 \lambda \phi_m & \text{if $m \in \mathbf{I}_+$,}\\
 \phi_m & \text{if $m \in \mathbf{I}_-$,}
 \end{cases}
 \end{gather*}
 is an eigenvector of the transpose matrix of $L(\eta)$ with an eigenvalue $\lambda^2$, that is,
 \begin{gather*}
 L(\eta)^{\mathsf{T}} \psi = \lambda^2 \psi .
 \end{gather*}
\end{Lemma}
\begin{proof}Let $\psi' = L_-^{\mathsf{T}} \psi $ and $\psi''= L_+^{\mathsf{T}} \psi'$ $\big({=} L(\eta)^{\mathsf{T}} \psi\big)$. In the following equations, we assume that $\psi_0 = \psi_\level = \psi'_0 = \psi'_\level = 0$.
 Then we obtain
 \begin{gather*}
 \psi'_m =
 \begin{cases}
 \psi_m &\text{if $m\in \mathbf{I}_+$,} \\
 -\psi_m + z_m (\psi_{m-1} + \psi_{m+1})&\text{if $m\in \mathbf{I}_-$,}
 \end{cases}
 \end{gather*}
 and
 \begin{align*}
 \psi''_m =
 \begin{cases}
 -\psi'_m + z_m (\psi'_{m-1} + \psi'_{m+1})&\text{if $m\in \mathbf{I}_+$,}\\
 \psi'_m &\text{if $m\in \mathbf{I}_-$.}
 \end{cases}
 \end{align*}
 For any $m \in \mathbf{I}_-$, we compute
 \begin{align*}
 \psi''_m - \lambda^2 \psi_{m}
 &=\psi'_m- \lambda^2 \psi_m =-\psi_m + z_m (\psi_{m-1} + \psi_{m+1}) -\lambda^2 \psi_m \\
 &=-\phi_m + z_m \lambda (\phi_{m-1}+\phi_{m+1}) -\lambda^2 \phi_m \\
 &=z_m \lambda \big( {-}z_m^{-1}\big(\lambda + \lambda^{-1}\big)\phi_m + \phi_{m-1}+\phi_{m+1}\big)=0.
 \end{align*}
 In particular, we obtain $\psi'_{m} = \lambda^2 \psi_m$ for $m \in \mathbf{I}_-$. Thus, for any $m \in \mathbf{I}_+$, we find that
 \begin{align*}
 \psi''_m - \lambda^2 \psi_{m}
 &=-\psi'_m + z_m (\psi'_{m-1} + \psi'_{m+1}) -\lambda^2 \psi_m
 =-\psi_m + z_m \lambda^2( \psi_{m-1} + \psi_{m+1}) -\lambda^2 \psi_m \\
 &=-\lambda \phi_m + z_m \lambda^2( \phi_{m-1} + \phi_{m+1}) -\lambda^3 \phi_m \\
 &=z_m \lambda^2 \big( {-}z_m^{-1}\big(\lambda + \lambda^{-1}\big) \phi_m + \phi_{m-1} + \phi_{m+1}\big)
 =0,
 \end{align*}
 and this complete the proof.
\end{proof}

Now we focus on solving the difference equation \eqref{eq:phi rec} satisfying the boundary conditions~\eqref{eq:phi boundary}. We will show that $\big(\phi_m^{(a)}\big)_{m=0, \dots, \level}$ for
\begin{gather}\label{eq:solution phi}
\phi_m^{(a)} =\det \begin{pmatrix}
2\cos \dfrac{\pi a}{\level+2} & 2 \cos \dfrac{\pi a(m+1)}{\level+2}\\
{\sin \dfrac{\pi (a-1)}{\level+2} }/{\sin \dfrac{\pi}{\level+2}} & {\sin \dfrac{\pi (a-1)(m+1) }{\level+2} }/{\sin \dfrac{\pi (m+1)}{\level+2}}
\end{pmatrix}
\end{gather}
is a non-zero solution of \eqref{eq:phi rec} and \eqref{eq:phi boundary} for $\lambda = \zeta^a$ if $a = 2, \dots, \level$ (Theorem~\ref{theorem:solution phi}).

To prove this, we introduce the following Laurent polynomials:
\begin{gather*}
\alpha^{(a)}(x) = x^a + x^{-a}, \qquad
\beta^{(a)}(x) = \frac{x^{a-1} - x^{-a+1}}{x - x^{-1}}.
\end{gather*}
We write $\alpha^{(a)}\big(x^{m+1}\big)$ and $\beta^{(a)}\big(x^{m+1}\big)$ as $\al$ and $\be$.
We also define a Laurent polyno\-mial~$P_m^{(a)} (x) $ by
\begin{gather*}
P_m^{(a)} (x) =\det
 \left(
 \begin{matrix}
 \alpha_0^{(a)} (x) & \alpha_m^{(a)} (x) \\
 \beta_0^{(a)} (x) & \beta_m^{(a)} (x)
 \end{matrix}
 \right).
\end{gather*}
Note that $\phi_m^{(a)}$ in \eqref{eq:solution phi} can be written as $P_m^{(a)} (\zeta)$.

Let us examine difference equations for $\al$, $\be$ and $\PP$.

\begin{Lemma} The Laurent polynomials $\al$ satisfy the difference equation
\begin{gather}\label{eq:alpha rec}
\alpha_{m-1}^{(a)}(x) + \alpha_{m+1}^{(a)}(x) = \alpha_0^{(a)}(x) \al[m].
\end{gather}
\end{Lemma}
\begin{proof} We compute
 \begin{align*}
\alpha_{m-1}^{(a)} + \alpha_{m+1}^{(a)}(x)& = \big(x^{am} + x^{-am} \big) + \big(x^{a(m+2)} + x^{-a(m+2)} \big) \\
 &= \big(x^a + x^{-a}\big) \big(x^{a(m+1)} + x^{-a(m+1)}\big) =\alpha_0^{(a)}(x) \al[m],
 \end{align*}
 and this proves the lemma.
\end{proof}

\begin{Lemma} The Laurent polynomials $\be$ satisfy the difference equation
\begin{gather}
 \big(\be[m-1] + \be[m+1] \big) (x^{m} - x^{-m}) \big(x^{m+2} - x^{-m-2}\big) \nonumber\\
\qquad{}= \al[0] \be[m] \big(x^{m+1} - x^{-m-1}\big)^2 - \al[m] \be[0] \big(x-x^{-1}\big)^2.\label{eq:beta rec}
\end{gather}
\end{Lemma}
\begin{proof} We compute
 \begin{gather*}
 \bigl(\be[m-1] + \be[m+1] \bigr) \big(x^{m} - x^{-m}) (x^{m+2} - x^{-m-2}\big)
 - \al[0] \be[m] \big(x^{m+1} - x^{-m-1}\big)^2 \\
 \qquad{} = \big(x^{(a-1)m} - x^{-(a-1)m}\big)\big(x^{m+2} - x^{-m-2}\big)\\
 \qquad\quad{} + \big(x^{(a-1)(m+2)} - x^{-(a-1)(m+2)}\big) \big(x^{m} - x^{-m}\big) \\
 \qquad\quad{} - \al[0] \big(x^{(a-1)m} - x^{-(a-1)m}\big) \big(x^{m+1} - x^{-m-1}\big) \\
 \qquad{} = \big(x^{am +2} + x^{-am-2} - x^{am -2m-2} - x^{-am +2m+2} \big) \\
\qquad\quad{} +\big( x^{a(m+2)-2} + x^{-a(m+2)+2} - x^{a(m+2) -2m-2} - x^{-a(m+2) +2m +2} \big)\\
\qquad\quad{} -\big(x^{a(m+2)} + x^{am} - x^{-am+2m+2} - x^{-a(m+2)+2m+2}\\
\qquad\quad{} -x^{a(m+2)-2m-2}-x^{am-2m-2} + x^{-am} + x^{-a(m+2)} \big)\\
\qquad{}=x^{am +2} + x^{-am-2} +x^{a(m+2)-2} + x^{-a(m+2)+2} -x^{am} - x^{-am} - x^{a(m+2)} - x^{-a(m+2)} \\
\qquad {}=\big(x-x^{-1}\big) \big(x^{am+1} - x^{-am-1} +x^{a(m+2)-1} - x^{-a(m+2)+1} \big) \\
\qquad {}=- \big(x-x^{-1}\big)\big(x^{a(m+1)} + x^{-a(m+1)}\big)\big(x^{a-1} - x^{-a+1}\big) \\
\qquad {}=- \al[m] \be[0] \big(x-x^{-1}\big)^2,
 \end{gather*}
 and this proves the lemma.
\end{proof}

\begin{Lemma}\label{lemma: P rec}
 The Laurent polynomials $\PP$ satisfy the difference equation
 \begin{gather*}
 \bigl(\PP[m-1] + \PP[m+1] \bigr) \big(x^{m} - x^{-m}\big) \big(x^{m+2} - x^{-m-2}\big)
 = \PP \al[0] \big( x^{m+1} -x^{-m-1}\big)^2 .
 \end{gather*}
\end{Lemma}
\begin{proof} By using \eqref{eq:alpha rec}, \eqref{eq:beta rec} and the relation
 \begin{gather*}
 \big(x^{m+1} - x^{-m+1}\big)^2 - \big(x - x^{-1}\big)^2 = \big(x^{m}-x^{-m}\big) \big(x^{m+2}-x^{-m-2}\big),
 \end{gather*}
 we obtain
\begin{gather*}
 \PP \al[0] \big( x^{m+1} -x^{-m-1}\big)^2 \\
 \qquad{}
 = \big( \al[0] \be[m] - \al[m] \be[0] \big) \al[0]\big( x^{m+1} -x^{-m-1}\big)^2 \\
 \qquad{}= \al[0] \big( \big(\be[m-1] + \be[m+1] \big) \big(x^{m} - x^{-m}\big) \big(x^{m+2} - x^{-m-2}\big) \\
 \qquad \quad{} + \al \be[0] \big(x-x^{-1}\big)^2\big) -\al[0] \al[m] \be[0] \big( x^{m+1} -x^{-m-1}\big)^2 \\
\qquad{} = \al[0] \big(\be[m-1] + \be[m+1] \big) \big(x^{m} - x^{-m}\big) \big(x^{m+2} - x^{-m-2}\big) \\
\qquad \quad{} - \al[0] \al[m] \be[0] \big( \big(x^{m+1} - x^{-m+1}\big)^2 - \big(x - x^{-1}\big)^2 \big) \\
\qquad{} = \big( \al[0] \big(\be[m-1] + \be[m+1] \big) - \al[0] \al[m] \be[0]\big)\\
\qquad \quad{}\times
 \big(x^{m} - x^{-m}\big) \big(x^{m+2} - x^{-m-2}\big) \\
\qquad{} =\big( \al[0] \big(\be[m-1] + \be[m+1] \big) - \big(\al[m-1] + \al[m+1] \big) \be[0] \big) \\
\qquad \quad{}\times \big(x^{m} - x^{-m}\big) \big(x^{m+2} - x^{-m-2}\big) \\
\qquad{} = \big(\PP[m-1] + \PP[m+1] \big)\big(x^{m} - x^{-m}\big) \big(x^{m+2} - x^{-m-2}\big),
\end{gather*}
completing the proof.
\end{proof}

\begin{figure}[t] \centering
 \includegraphics[width=7.5cm]{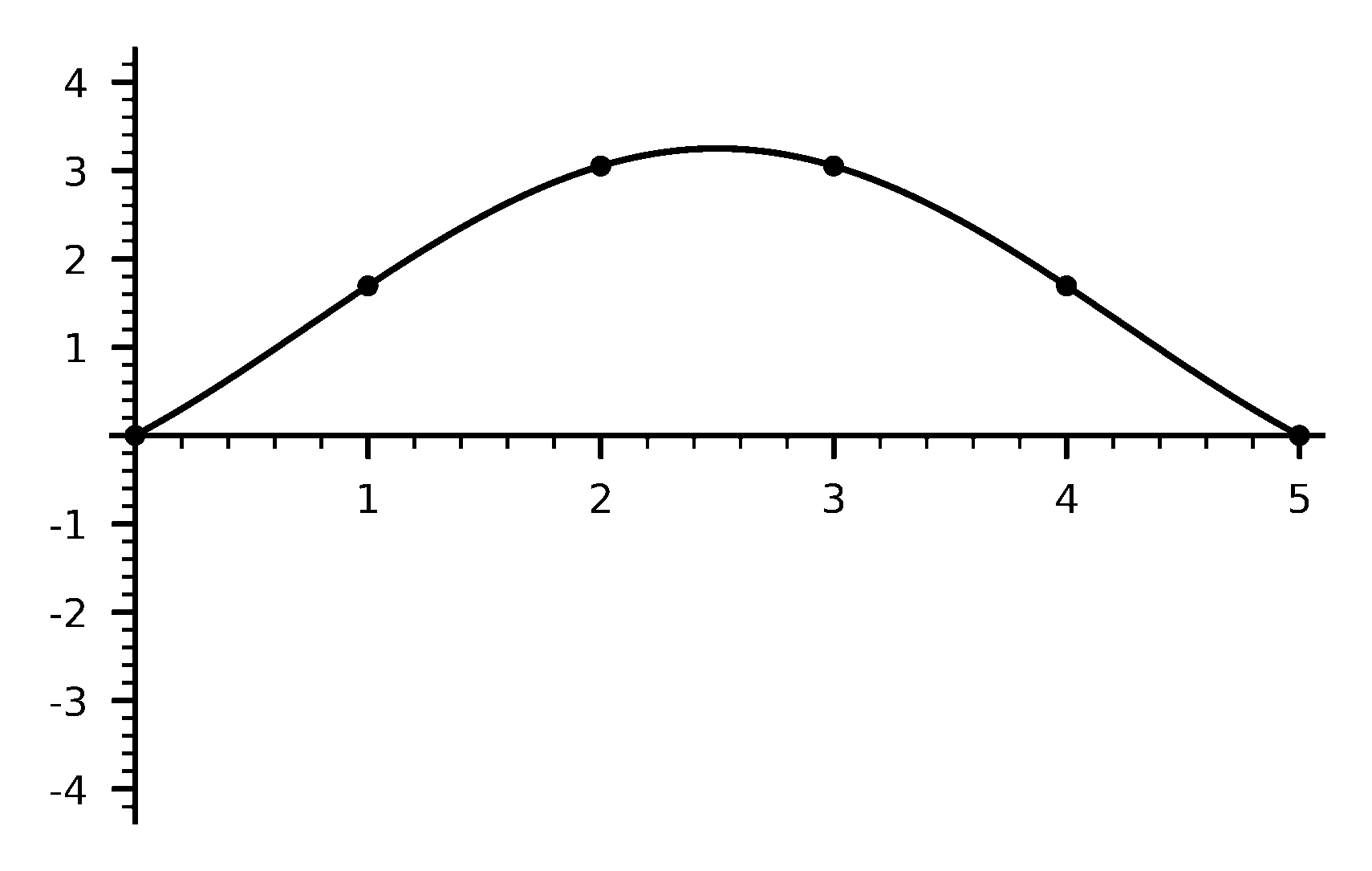}
 \includegraphics[width=7.5cm]{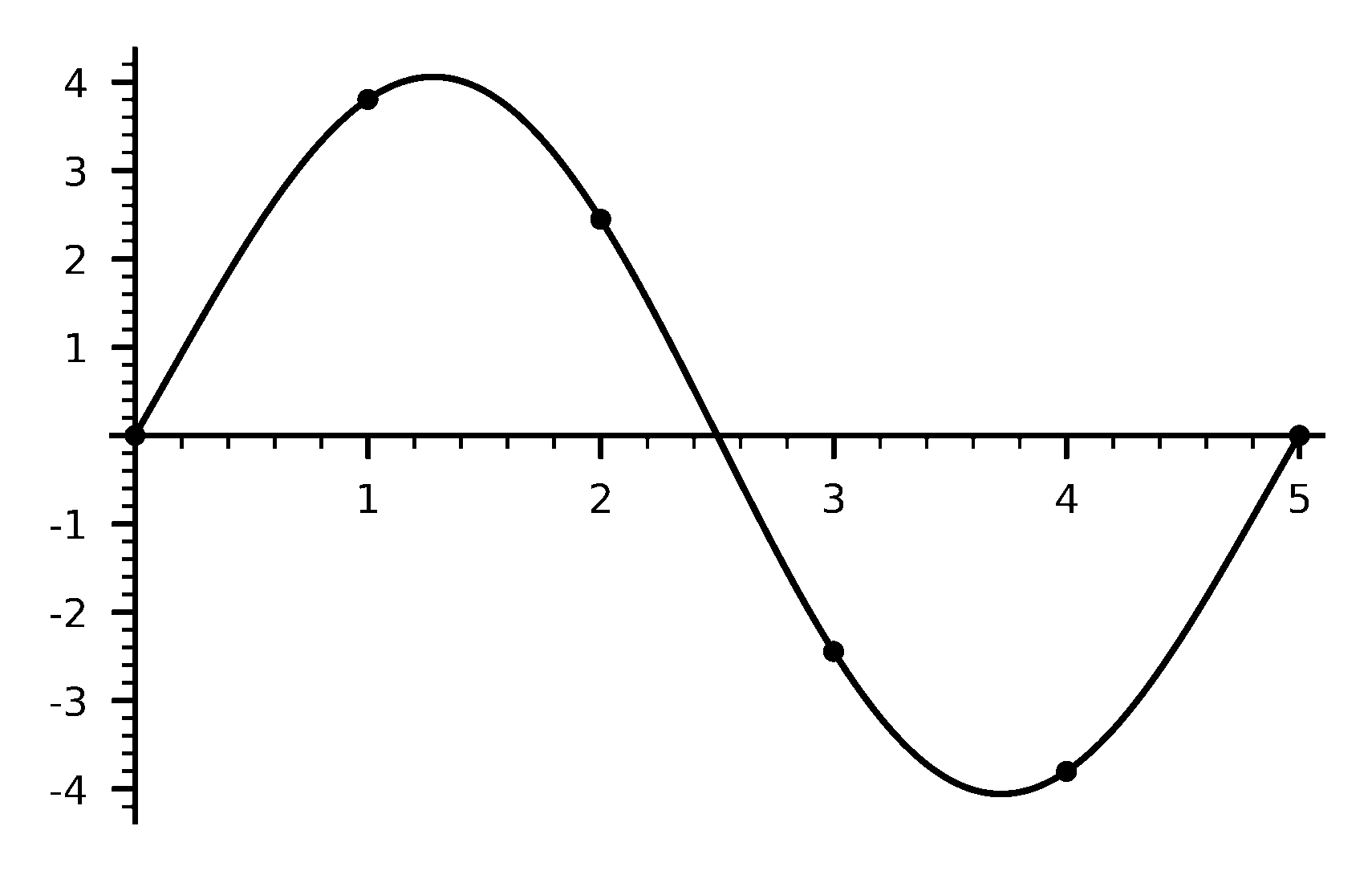}
 \caption{The values of $\big(\phi_m^{(2)}\big)_{m=0 \dots, \level}$ and $\big(\phi_m^{(3)}\big)_{m=0 \dots, \level}$ for $\level=5$.} \label{fig:phi a 2 and 3}
\end{figure}

Using Lemma \ref{lemma: P rec}, we can construct non-zero solutions of
the difference equation \eqref{eq:phi rec} satisfying the boundary conditions \eqref{eq:phi boundary}.
\begin{Theorem}\label{theorem:solution phi}
 For any $a = 2, 3,\dots , \level$ and $m=0, 1,\dots, \level$,
 let $\phi_m^{(a)} = P_{m}^{(a)} (\zeta)$.
 Then the following properties hold:
 \begin{enumerate}\itemsep=0pt
 \item[$1.$] The numbers $\big\{ \phi_m^{(a)} \,|\, 0 \leq m \leq \level \big\}$ satisfy the difference equation \eqref{eq:phi rec} for $\lambda = \zeta^a$.
 \item[$2.$] The boundary conditions \eqref{eq:phi boundary} hold, i.e., $\phi_0^{(a)} = \phi_{\level}^{(a)}=0$.
 \item[$3.$] There exists $m$ such that $\phi_m^{(a)} \neq 0$.
 \end{enumerate}
\end{Theorem}
\begin{proof}The property (1) follows from Lemma~\ref{lemma: P rec} and the definition of~$z_m$~\eqref{eq:def of zm}.
Now we prove the property~(2).
The definition of $\PP$ immediately implies that $\PP[0]=0 $, hence $\phi_0^{(a)} = 0$.
To prove $ \phi_{\level}^{(a)}=0$,
we define a Laurent polynomial $\widetilde{P}^{(a)} (x,y) $ of two variables by
\begin{gather*}
 \widetilde{P}^{(a)} (x,y) = \det
 \left(
 \begin{matrix}
 \alpha^{(a)} (x) & \alpha^{(a)} (y) \\
 \beta^{(a)} (x) & \beta^{(a)} (y)
 \end{matrix}
 \right).
\end{gather*}
It is easy to see that the polynomial $\widetilde{P}^{(a)} (x,y)$ satisfies
\begin{gather*}
\widetilde{P}^{(a)} (x,x^{m+1}) = P_m^{(a)} (x), \\
\widetilde{P}^{(a)} (x,y^{-1}) = \widetilde{P}^{(a)} (x,y), \\
\widetilde{P}^{(a)} (x,-y) = (-1)^a \widetilde{P}^{(a)} (x,y).
\end{gather*}
Thus we have
\begin{align*}
\phi_\level^{(a)} &= P_\level^{(a)} (\zeta)
= \widetilde{P}^{(a)} \big(\zeta,\zeta^{\level+1}\big)
=\widetilde{P}^{(a)} \big(\zeta,- \zeta^{-1} \big) \\
&=(-1)^a \widetilde{P}^{(a)} \big(\zeta,\zeta^{-1} \big)
=(-1)^a \widetilde{P}^{(a)} (\zeta,\zeta)=0.
\end{align*}
Now we prove the property (3).
We will show that $\phi_1^{(a)} > 0$.
The number $\phi_1^{(a)}$ can be written as
\begin{align*}
\phi_1^{(a)}
&= \zeta^a + \zeta^{-a} \cdot \frac{\zeta^{2a-2} - \zeta^{-2a+2}}{\zeta^2-\zeta^{-2}}
-\zeta^{2a} + \zeta^{-2a} \cdot \frac{\zeta^{a-1} - \zeta^{-a+1}}{\zeta-\zeta^{-1}} \\
&=\frac{\zeta^{a-1} - \zeta^{-a+1}}{\zeta^2-\zeta^{-2}}
\bigl( \big(\zeta^a + \zeta^{-a}\big) \big( \zeta^{a-1} + \zeta^{-a+1}\big)
- \big(\zeta^{2a} + \zeta^{-2a} \big) \big(\zeta + \zeta^{-1}\big) \bigr) \\
&= \frac{\zeta^{a-1} - \zeta^{-a+1}}{\zeta^2-\zeta^{-2}}
\bigl( \big( \zeta + \zeta^{-1} \big) - \big(\zeta^{2a+1}+\zeta^{-2a-1} \big) \bigr) \\
&={\sin \frac{(a-1) \pi }{\level +2}} \left( \sin \frac{2 \pi }{\level +2} \right)^{-1}
\left( 2\cos \frac{\pi}{\level +2} - 2\cos \frac{(2a +1 )\pi}{\level +2} \right).
\end{align*}
This shows that $\phi_1^{(a)} > 0$ for $a=2, 3, \dots , \level$.
\end{proof}

\begin{figure}[t] \centering
 \includegraphics[width=7.5cm]{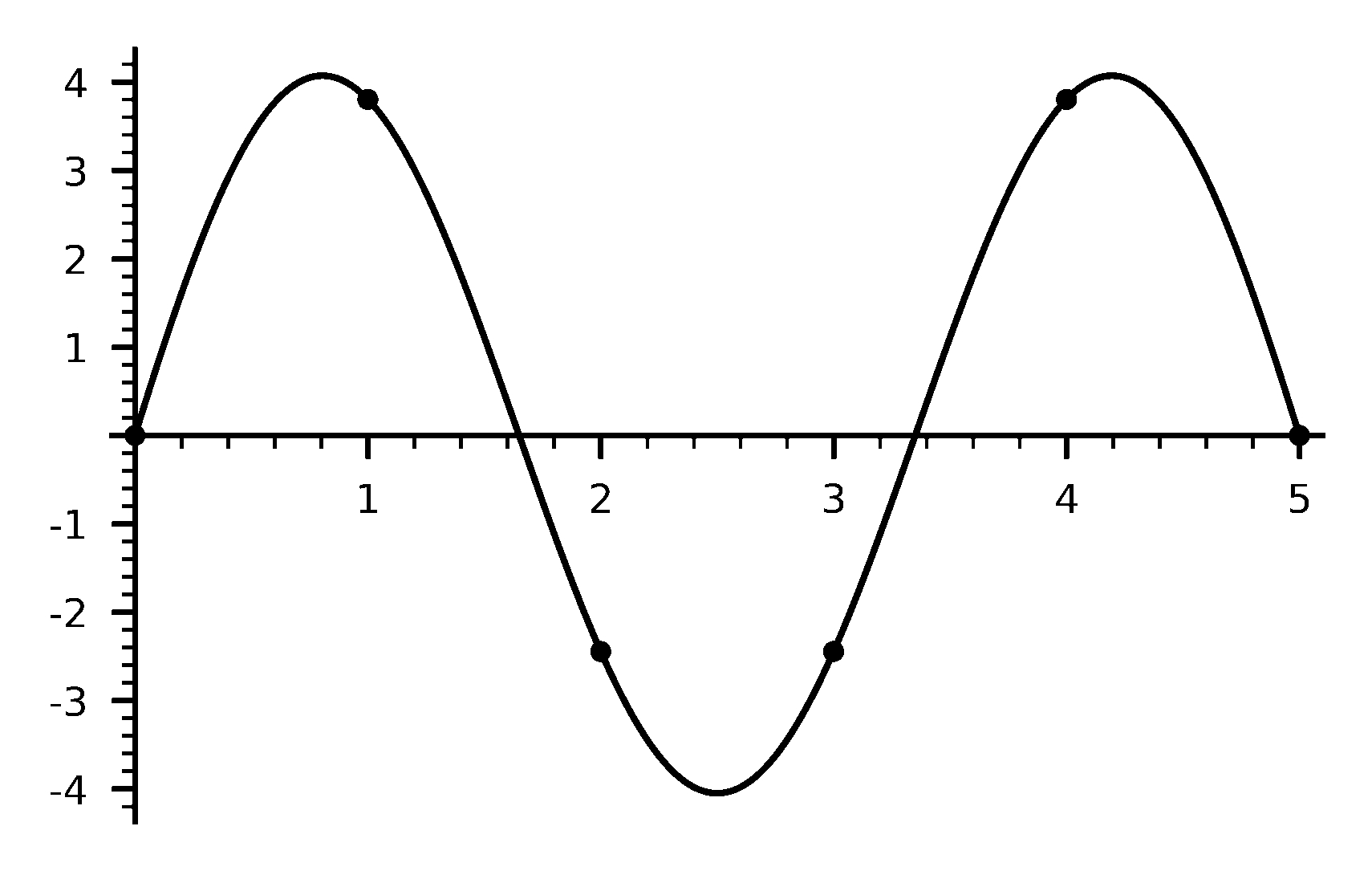}
 \includegraphics[width=7.5cm]{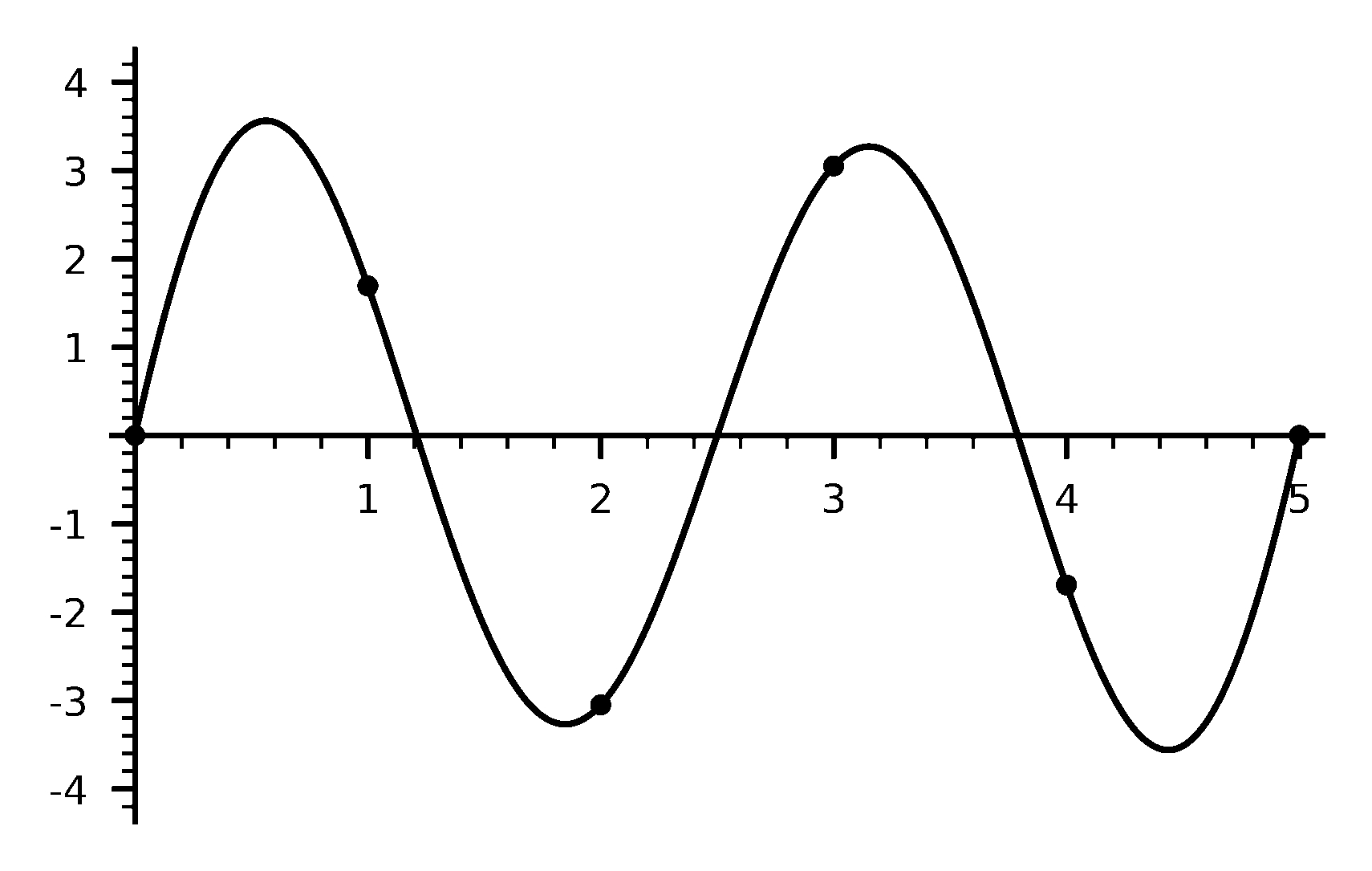}
 \caption{The values of $\big(\phi_m^{(4)}\big)_{m=0 \dots, \level}$ and $\big(\phi_m^{(5)}\big)_{m=0 \dots, \level}$ for $\level=5$.} \label{fig:phi a 4 and 5}
\end{figure}

We plot the values of $\big(\phi_m^{(a)}\big)_{m=0 \dots, \level}$ for $\level=5$ and $a = 2,3,4,5$ in Figs.~\ref{fig:phi a 2 and 3} and~\ref{fig:phi a 4 and 5}.
The underlying graphs are plots of the function
\begin{gather*}
\det
\begin{pmatrix}
2\cos \dfrac{\pi a}{\level+2} & 2 \cos \dfrac{\pi a(u+1)}{\level+2}\vspace{1mm}\\
{\sin \dfrac{\pi (a-1)}{\level+2} }\big/{\sin \dfrac{\pi}{\level+2}} & {\sin \dfrac{\pi (a-1)(u+1) }{\level+2} }\big/{\sin \dfrac{\pi (u+1)}{\level+2}}
\end{pmatrix}
\end{gather*}
in an interval $0 \leq u \leq \level$, and the points on these graphs represent the values of $\big(\phi_m^{(a)}\big)_{m=0 \dots, \level}$.

\begin{Corollary}\label{corollary:exponents A1 ell}
 The exponents of $Q(A_1, \level)$ are $2,3,\dots, \level$, that is,
 \begin{gather*}
 \det (xI - J_\gamma(\eta)) =
 \prod_{a=2}^{\level} \big(x - {\rm e}^{\frac{2 \pi{\rm i} a}{\level + 2}}\big).
 \end{gather*}
 In particular, Conjecture~{\rm \ref{conj: char poly}} is true in this case.
\end{Corollary}
\begin{proof}
 From Theorem \ref{theorem:solution phi},
 we find that ${\rm e}^{\frac{2 \pi{\rm i} a}{\level + 2}}$ for $a=2,3, \dots, \level$ are eigenvalues of $J_\gamma(\eta)$.
 These are all the eigenvalues and their multiplicities are one, since the size of $J_\gamma(\eta)$ is $\level-1$.
\end{proof}

\subsection[$(A_r, 2)$ case]{$\boldsymbol{(A_r, 2)}$ case}

First, we will see that the right-hand side of the conjectural formula~\eqref{eq:conj char poly} for $(A_r, 2)$ is the same as that for $(A_1, \level)$ if we change the parameter as $r \leftrightarrow \level-1$.
Such a phenomenon is known as the \emph{level-rank duality}.
\begin{Lemma}\label{lemma:Ar 2 N/D}
 The right-hand side in \eqref{eq:conj char poly} is given by
 \begin{gather*}
 \frac{N_{A_r, 2}(x)}{D_{A_r, 2}(x)} =
 \prod_{a=2}^{r+1} \big(x - {\rm e}^{\frac{2 \pi{\rm i} a}{r + 3}}\big).
 \end{gather*}
\end{Lemma}
\begin{proof}
 Because $\dcn = r+1$, we obtain
 \begin{gather*}
 N_{A_r ,2} (x) = \left( \frac{x^{r+3} - 1}{x -1} \right)^r,
 \end{gather*}
 and
 \begin{gather*}
 D_{A_r, 2} (x) = \prod_{\alpha \in \Delta}
 \big(x - {\rm e}^{\frac{2\pi{\rm i} \innerproduct{\rho}{\alpha}}{r+3}}\big).
 \end{gather*}

 We can easily see that, say using a concrete realization of the root system of type~$A_r$, the following holds:
 \begin{gather*}
 \# \{ \alpha \in \Delta_+ \,|\, \innerproduct{\rho}{\alpha} =a \}
 =
 \begin{cases}
 r+1-a & \text{if $1\leq a \leq r+1$},\\
 0 & \text{if $a=r+2$}.
 \end{cases}
 \end{gather*}
 Thus for any $1 \leq a \leq r+2$, we obtain
 \begin{gather*}
 \#\{ \alpha \in \Delta_+ \,|\, \innerproduct{\rho}{\alpha} =a \mod r+3 \} \\
 \qquad{} = \#\{ \alpha \in \Delta_+ \,|\, \innerproduct{\rho}{\alpha} =a \} + \#\{ \alpha \in \Delta_+ \,|\, \innerproduct{\rho}{-\alpha} =a -r-3 \}\\
 \qquad{} =
 \begin{cases}
 r & \text{if $a=1$,}\\
 r-1 & \text{if $ 2\leq a \leq r+1$,}\\
 r & \text{if $a=r+2$,}
 \end{cases}
 \end{gather*}
 and this implies that
 \begin{gather*}
 D_{A_r, 2} (x) \prod_{a=2}^{r+1} \big(x - {\rm e}^{\frac{2 \pi{\rm i} a}{r + 3}}\big) =N_{A_r, 2} (x),
 \end{gather*}
 completing the proof.
 Fig.~\ref{fig:exponents Ar 2} illustrates these calculations.
\end{proof}

\begin{figure} \centering
 \begin{tikzpicture}
 [scale=1.0,auto=left]
 \foreach \x in {1, 2, ..., 8}
 \foreach \y in {1, 2, ..., 6}
 \draw [black] (\x,\y) circle (5pt);
 \foreach \x / \y in {1/1, 1/2, 1/3,1/4,1/5,1/6, 2/1,2/2,2/3,2/4,2/5,3/1,3/2,3/3,3/4,4/1,4/2,4/3,5/1,5/2,6/1}
 \draw [black,fill] (\x,\y) circle (3pt);
 \foreach \x / \y in {8/6,8/5,8/4,8/3,8/2,8/1,7/6,7/5,7/4,7/3,7/2,6/6,6/5,6/4,6/3,5/6,5/5,5/4,4/6,4/5,3/6}
 \draw (\x,\y) node[minimum size =2pt,inner sep=2pt,diamond,draw,fill]{};
 \foreach \x in {1, 2, ..., 8}
 \node(\x) at (\x, 0.3) {$\x$};
 \end{tikzpicture}
 \caption{The exponents of $N_{A_r, 2}$ and $D_{A_r, 2}$ for $r=6$. The meanings of the symbols are the same as in Fig.~\ref{fig:exponents A3 3}.}
 \label{fig:exponents Ar 2}
\end{figure}

We now complete the proof of Theorem \ref{theorem: A1 level}.

\begin{Corollary} Conjecture {\rm \ref{conj: char poly}} is true for $(A_r,2)$.
\end{Corollary}
\begin{proof}
 When we choose suitable bipartite decompositions,
 the mutation loop $\gamma = \gamma(A_r,2)$
 is obtained from
 the mutation loop $\gamma' = \gamma(A_1,r+1)$ by reversing all arrows in quivers.
 This implies that their cluster transformations
 are related by $\mu_{\gamma'} = \iota \circ \mu_{\gamma} \circ \iota$,
 where $\iota$ is the map $(y_1,\dots,y_r) \mapsto \big(y_1^{-1},\dots,y_r^{-1}\big)$.
 From this relation,
 we can easily verify that the exponents for~$(A_r,2)$ are the same as
 those for $(A_1,r+1)$.
 Thus the corollary follows from Corollary~\ref{corollary:exponents A1 ell} and
 Lemma~\ref{lemma:Ar 2 N/D}.
\end{proof}

\section[Partition $q$-series]{Partition $\boldsymbol{q}$-series}\label{sec:partition q series}

\subsection{Mutation networks}\label{sec:mutation network}
For a mutation loop $\gamma = (Q,m,\nu)$, we define a combinatorial object called a \emph{mutation network}, which we will denote by $\network_\gamma$. This was introduced in~\cite{TerashimaYamazaki}.
Roughly speaking, a mutation network is a graph obtained by extracting only the parts that change in the transitions of the quivers. We will use mutation networks to define partition $q$-series.

Consider the sets defined by
\begin{gather*}
\overline{E}(t) = \{ (i,t) \,|\, 1\leq i \leq n \}, \qquad
\overline{E} = \bigcup_{0\leq t \leq T }\overline{E}(t) .
\end{gather*}
We define an equivalence relation $\sim$ on $\overline{E} $ as follows:
\begin{enumerate}\itemsep=0pt
 \item For all $t=1 ,\dots ,T$, $(i,t-1) \sim (i,t)$ if $i \neq m_t$.
 \item For all $i = 1 \dots, n$, $(\nu(i),T) \sim (i,0)$.
\end{enumerate}
Let $E = \overline{E} / \sim$ be the corresponding quotient set.

First, we define graphs $\overline{\network}_{\gamma}(t)$ for all $t=1 ,\dots, T$.
These consist of two types of vertices: black vertices and a single square vertex.
The black vertices are labeled with elements in $\overline{E}(t-1) \cup \overline{E}(t)$, and represented as solid circles $\edgevertex$, while the square vertex is labeled with $t$, and represented as a hollow square $\tetvertex$.
There are three types of edges joining these vertices:
broken edges, arrows from the square vertex to a black vertex, and arrows from a black vertex to the square vertex.
These edges are added as follows.
First, we add broken edges between the square vertex and black vertices labeled with $(m_t,t-1)$ or $(m_t,t)$.
Then, for all $i=1,\dots, n$, we add $Q_{m_t,i}(t-1)$ arrows from the square vertex to the black vertex labeled with $(i,t-1)$ if $Q_{m_t,i}(t-1)>0$.
Finally, for all $i=1,\dots, n$, we add $Q_{i,m_t}(t-1)$ arrows from the black vertex labeled with $(i,t-1)$ to the square vertex if $Q_{i,m_t}(t-1)>0$.
These rules can be summarized as follows:
\begin{enumerate}\itemsep=0pt
 \item[1)] $\defnp$ if $(i,s) = (m_t,t-1)$ or $(m_t, t)$,
 \item[2)] $\defn$ if $a=Q_{m_t,i}(t-1)>0$ and $s = t-1$, for all $i$,
 \item[3)] $\defnpp$ if $a=Q_{i,m_t}(t-1)>0$ and $s = t-1$, for all $i$.
\end{enumerate}

Let $ \overline{\network}_\gamma = \cup_{1 \leq t \leq T} \overline{\network}_\gamma (t)$.
Then the mutation network $\network_{\gamma}$ of the mutation sequence $\gamma$ is defined to be the quotient of $\overline{\network}_\gamma$ by~$\sim$:
\begin{gather*}
\network_{\gamma} = \overline{\network}_\gamma / \sim.
\end{gather*}
Here, the quotient carried out such that the vertices are labeled with the elements of~$E$, and all edges are included (i.e., we do not cancel the arrows pointing in the opposite direction). We denote the set of square vertices of $\network_\gamma$ by $\Delta$, that is, $\Delta = \{ 1,\dots , T \}$.

For a mutation network $\network_\gamma$, we define a map $\varphi\colon \Delta \to E$ from the set of the black vertices to the set of the square vertices by
\begin{gather*}
\varphi (t) = [ (m_t , t-1 ) ],
\end{gather*}
where $[(m_t,t-1)]$ is an equivalence class that contains $(m_t,t-1)$.
Let $\mathbf{I} / \nu$ is the set of $\nu$-orbits in $\mathbf{I}$:
\begin{gather*}
\mathbf{I} / \nu = \big\{ \big\{ \nu^k(i) \,|\, k \in \Z \big\} \,|\, i \in \mathbf{I} \big\}.
\end{gather*}
We also define a map $\psi\colon \Delta \to \mathbf{I} / \nu$ by
\begin{gather*}
\psi(t) = [m_t],
\end{gather*}
where $[m_t]$ is the $\nu$-orbit of $m_t$.
\begin{Lemma}\label{lemma: regular mutation loop}
 If $\gamma$ is regular, both the map $\varphi$ and $\psi$ defined above are bijections.
\end{Lemma}
\begin{proof} The injectivity of $\varphi$ follows from the definition of $E$ and $\varphi$, and the surjectivity of $\varphi$ follows from the condition~(1) in Definition~\ref{def:regular}.
 The injectivity of $\psi$ follows from the condition~(1) in Definition~\ref{def:regular}, and the surjectivity of $\psi$ follows from the condition~(2) in Definition~\ref{def:regular}.
\end{proof}

In particular, the number of black vertices and the number of square vertices for a regular mutation loop are the same since $\varphi$ is a bijection.

\begin{Example}\label{example: A2}
 Define a quiver $Q$ by
 \[
 Q=\lower1.0ex\hbox{
 \begin{tikzpicture}
 [scale=0.75]
 \node (1) at (0,0) [] {1};
 \node (2) at (2,0) [] {2.};
 \draw[arrows={-Stealth[scale=1.1]}] (2)--(1);
 \end{tikzpicture}}
 \]
 Let $\gamma = (Q, m, \nu)$, where $m=(1,2)$ and $\nu = \id $.
 It is a regular mutation loop since
 the quiver transitions are given by
 \[
 \begin{tikzpicture}
 [scale=0.75]
 \node (1) at (0,0) [] {\underline{1}};
 \node (2) at (2,0) [] {2};
 \draw[arrows={-Stealth[scale=1.1]}] (2)--(1);
 \node (mu1) at (3,-0.1) [] {$\longrightarrow$};
 \node (mu1') at (3,0.2) [] {$\mu_1$};
 \node (1') at (4,0) [] {1};
 \node (2') at (6,0) [] {\underline{2}};
 \draw[arrows={-Stealth[scale=1.1]}] (1')--(2');
 \node (mu1) at (7,-0.1) [] {$\longrightarrow$};
 \node (mu1') at (7,0.2) [] {$\mu_2$};
 \node (1'') at (8,0) [] {1};
 \node (2'') at (10,0) [] {{2,}};
 \draw[arrows={-Stealth[scale=1.1]}] (2'')--(1'');
 \end{tikzpicture}
 \]
 where the underlines indicate the mutated vertices.
 The graphs $\overline{\network}_\gamma (t)$ can then be given as
 \[
 \overline{\network}_{\gamma}(1)=
 \begin{tikzpicture}
 [baseline={([yshift=-.5ex]current bounding box.center)},scale=0.9,auto=left,black_vertex/.style={circle,draw,fill,scale=0.75},square_vertex/.style={rectangle,scale=0.8,draw}]
 \node (a) at (0,0) [black_vertex][label=above:${(1,0)}$]{};
 \node (t1) at (0,-1) [square_vertex][label=right:$1$]{};
 \node (ap) at (0,-2) [black_vertex][label=below:${(1,1)}$]{};
 \node (bp) at (1,-2) [black_vertex][label=below:${(2,1)}$]{};
 \node (b) at (1,0) [black_vertex][label=above:${(2,0)}$]{};
 \draw [dashed] (a)--(t1);
 \draw [dashed] (t1)--(ap);
 \draw[arrows={-Stealth[scale=1.2]}] (b)--(t1);
 \end{tikzpicture}, \qquad
 \overline{\network}_{\gamma}(2) =
 \begin{tikzpicture}
 [baseline={([yshift=-.5ex]current bounding box.center)},scale=0.9,auto=left,black_vertex/.style={circle,draw,fill,scale=0.75},square_vertex/.style={rectangle,scale=0.8,draw}]
 \node (ap) at (-1,0) [black_vertex][label=above:${(1,1)}$]{};
 \node (t3) at (0,-1) [square_vertex][label=right:$2$]{};
 \node (b) at (0,0) [black_vertex][label=above:${(2,1)}$]{};
 \node (a) at (-1,-2) [black_vertex][label=below:${(1,2)}$]{};
 \node (bp) at (0,-2) [black_vertex][label=below:${(2,2)}$]{};
 \draw [dashed] (b)--(t3);
 \draw [dashed] (t3)--(bp);
 \draw[arrows={-Stealth[scale=1.2]}] (ap)--(t3);
 \end{tikzpicture}.
 \]
 Let $E = \{ e_1 , e_2 \} $
 be the set of black vertices in $\network_\gamma$,
 where
 \begin{gather*}
 e_1 = \{ (1,0), (1,1) ,(1,2) \}, \qquad
 e_2 = \{ (2,0), (2,1) ,(2,2) \}.
 \end{gather*}
 The mutation network of $\gamma$ is then given by
 \begin{align}\label{eq: A2 network}
 \network_{\gamma} =
 \;
 \begin{tikzpicture}
 [baseline={([yshift=-.5ex]current bounding box.center)},scale=1.5,auto=left,black_vertex/.style={circle,draw,fill,scale=0.75},square_vertex/.style={rectangle,scale=0.8,draw}]
 \node (a) at (0,0) [black_vertex][label=right:$e_1$]{};
 \node (b) at (1,0) [black_vertex][label=right:$e_2$]{};
 \node (t1) at (0,-1) [square_vertex][label=right:$1$]{};
 \node (t3) at (1,-1) [square_vertex][label=right:$2$]{};
 \draw [dashed] (a)edge [bend left=15](t1);
 \draw [dashed] (t1)edge [bend left=15](a);
 \draw [dashed] (b)edge [bend left=15](t3);
 \draw [dashed] (t3)edge [bend left=15](b);
 \draw[arrows={-Stealth[scale=1.2]}] (b)--(t1);
 \draw[arrows={-Stealth[scale=1.2]}] (a)--(t3);
 \end{tikzpicture}.
 \end{align}
\end{Example}

\begin{Definition}
 Let $\gamma$ be a mutation loop and $\network_\gamma$ be its mutation network.
 For any given pair consisting of a black vertex labeled with $e \in E$ and a square vertex labeled with $t \in \Delta$,
 let~$N_{0}^{et}$ be the number of broken lines between~$e$ and~$t$,
 $N_{+}^{et}$ be the number of arrows from $t$ to $e$,
 and~$N_{-}^{et}$ be the number of arrows from $e$ to $t$:
 \begin{gather*}
 N_{0}^{et} = \# \Big\{ \defNp \Big\}, \qquad
 N_{+}^{et} = \# \Big\{ \defN \Big\}, \qquad
 N_{-}^{et} = \# \Big\{ \defNpp \Big\}.
 \end{gather*}
 Let $N_0$ be the $E \times \Delta$ matrix whose $(e,t)$-entry is $N_{0}^{et} $.
 Define the $E \times \Delta$ matrices $N_+$ and $N_-$ likewise.
 We say that the matrices $N_0$, $N_+$, and $N_-$ are the \emph{adjacency matrices} of the mutation network $\network_\gamma$.
\end{Definition}

\begin{Definition} Let $N_0$, $N_+$, and $N_-$ be the adjacency matrices of the mutation network $\network_\gamma$.
 Define two additional $E \times \Delta$ matrices $A_+ =\big(A_{+}^{et}\big)$ and $A_- =\big(A_{-}^{et}\big)$ by
 \begin{gather*}
 A_+ = N_0 - N_+ , \qquad
 A_- = N_0 - N_- .
 \end{gather*}
 We call the matrices $A_+$ and $A_-$ the \emph{Neumann--Zagier matrices} of the mutation loop~$\gamma$.
\end{Definition}

\begin{Example}
 Let $\gamma$ be the mutation loop defined in Example~\ref{example: A2}.
 It follows from~\eqref{eq: A2 network} that the adjacency matrices of the mutation network $\network_\gamma$ are given by
 \begin{gather*}
 N_0 =
 \begin{pmatrix}
 2 & 0 \\
 0 & 2
 \end{pmatrix},\qquad
 N_+ =
 \begin{pmatrix}
 0 & 0 \\
 0 & 0
 \end{pmatrix},\qquad
 N_- =
 \begin{pmatrix}
 0 & 1 \\
 1 & 0
 \end{pmatrix},
 \end{gather*}
 and the Neumann--Zagier matrices of the mutation loop $\gamma$ are given by
 \begin{gather*}
 A_+ =
 \begin{pmatrix}
 2 & 0 \\
 0 & 2
 \end{pmatrix},\qquad
 A_- =
 \begin{pmatrix}
 2 & -1 \\
 -1 & 2
 \end{pmatrix}.
 \end{gather*}
\end{Example}

The following lemma is proved in~\cite{Mizuno20}.
\begin{Lemma}\label{lemma: NZ symplectic property}
 The matrix $A_+ A_-^{\mathsf{T}}$ is symmetric, where ${}^{\mathsf{T}}$ is a transpose of a matrix.
\end{Lemma}
\begin{proof}
 This follows from Proposition 3.11 in~\cite{Mizuno20}.
\end{proof}

In other words, Neumann--Zagier matrices satisfy the following relation:
\begin{gather*}
A_+ A_-^{\mathsf{T}} =A_- A_+^{\mathsf{T}} .
\end{gather*}

\subsection[Definition of partition $q$-series]{Definition of partition $\boldsymbol{q}$-series}
Let $\gamma$ be a mutation loop. Let $E$ and $\Delta$ be the set of the black vertices and of the square vertices of the mutation network $\network_\gamma$, respectively.

We define two abelian groups $Z_\gamma$ and $B_\gamma$ by
\begin{gather*}
 Z_\gamma =
\big\{ (u,v)
 \,|\,
 u \in \Z^\Delta, \, v \in \Q^\Delta , \, A_- u = A_+ v
\big\},\qquad
 B_\gamma =
 \big\{ \big( A_+^{\mathsf{T}} w , A_-^{\mathsf{T}} w\big) \,|\, w \in \Z^E
 \big\}.
\end{gather*}
It follows from Lemma \ref{lemma: NZ symplectic property} that $B_\gamma$ is a subgroup of~$Z_\gamma$.
Let $H_\gamma$ be the quotient group of~$Z_\gamma$ by~$B_\gamma$:
\begin{gather*}
 H_\gamma = Z_\gamma / B_\gamma.
\end{gather*}
For any elements $\sigma \in H_\gamma$, we define $\sigma_{\geq 0} \subset \sigma$ as follows:
\begin{gather*}
 \sigma_{\geq 0} = \{ (u,v) \in \sigma \,|\, \text{$u_i \geq 0$ for all $i=1,\dots, N$} \}.
\end{gather*}

For vectors $u \in (\Z_{\geq 0})^\Delta$ and $v \in \Q^\Delta$, we define a $q$-series $W(x,y)$ by
\begin{gather*}
W( u,v ) = \frac{q^{ \frac{1}{2} u \cdot v }}{ (q)_{u_1} \cdots (q)_{u_T} } ,
\end{gather*}
where $u \cdot v$ is the usual dot product, that is, $u \cdot v := \sum\limits_{i=1}^T u_i v_i$, and $(q)_n$ for a non-negative integer~$n$ is defined by
\begin{gather*}
 (q)_n = \prod_{k=1}^n \big(1-q^k\big).
\end{gather*}

A mutation loop is called \emph{nondegenerate} if $A_+$ is a non-singular matrix.
In addition, a~nondegenerate mutation loop is called \emph{positive definite} if $A_+^{-1} A_-$ is a positive definite symmetric matrix.
Note that $A_+^{-1} A_-$ is always symmetric (if the mutation is nondegenerate) since $A_+^{-1} A_- = A_+^{-1} A_- A_+^{\mathsf{T}} \big(A_+^{-1}\big)^{\mathsf{T}}$, and $A_- A_+^{\mathsf{T}}$ is symmetric.

\begin{Definition}Let $\gamma$ be a nondegenerate positive definite mutation loop, and $\sigma$ be an element of $H_\gamma$.
 The \emph{partition $q$-series} of $\gamma$ associated with $\sigma \in H_\gamma$ is defined by
 \begin{gather*}
 \mathcal{Z}_\gamma^\sigma (q) =
 \sum_{ (u,v) \in \sigma_{\geq 0}} W(u, v).
 \end{gather*}
 We also define the \emph{total partition $q$-series} of $\gamma$ as the sum of the partition $q$-series over~$H_\gamma$:
 \begin{gather*}
 \mathcal{Z}_\gamma (q) = \sum_{\sigma \in H_\gamma} \mathcal{Z}_\gamma^\sigma (q) .
 \end{gather*}
\end{Definition}

\begin{Lemma}\label{lemma:total Z formula}
 Suppose that $\gamma$ is a nondegenerate positive definite mutation loop. Then the total partition $q$-series is given by
 \begin{gather*}
 \mathcal{Z}_\gamma (q) = \sum_{u \in (\Z_{\geq 0})^\Delta }
 \frac{q^{ \frac{1}{2} u^{\mathsf{T}} A_+^{-1} A_- u }}{ (q)_{u_1} \cdots (q)_{u_T} }.
 \end{gather*}
\end{Lemma}
\begin{proof} Because $\gamma$ is nondegenerate, the abelian group $Z_\gamma$ can be written as
 \begin{gather}\label{eq: non-singular Z}
 Z_\gamma = \big\{ \big(u, A_+^{-1} A_- u\big) \,|\, u \in \Z^\Delta \big\}.
 \end{gather}
 This implies that
 \begin{gather*}
 \mathcal{Z}_\gamma (q) = \sum_{u \in (\Z_{\geq 0})^\Delta } W\big(u , A_+^{-1} A_- u \big)
 = \sum_{u \in (\Z_{\geq 0})^\Delta }
 \frac{q^{ \frac{1}{2} u^{\mathsf{T}} A_+^{-1} A_- u }}{ (q)_{u_1} \cdots (q)_{u_T} },
 \end{gather*}
 and the sum is well-defined since $A_+^{-1} A_-$ is positive definite.
\end{proof}

\begin{Corollary}Suppose that $\gamma$ is a nondegenerate positive definite mutation loop.
 Then partition $q$-series of $\gamma$ associated with any $\sigma \in H_\gamma$ and the total partition $q$-series are well-defined as a formal power series with integer coefficients in some rational power of~$q$.
\end{Corollary}

\subsection[Asymptotics of partition $q$-series]{Asymptotics of partition $\boldsymbol{q}$-series}\label{sec:asymptotics of Z}

In this section, we investigate the asymptotics of total partition $q$-series of nondegenerate positive definite mutation loops. The asymptotics will be described by using the Rogers dilogarithm function and the Jacobian matrix associated with the mutation loop.

Let $\gamma$ be a nondegenerate positive definite mutation loop.
For any $t=1,\dots , T$, we define rational functions $z_+(y;t)$ and $z_-(y;t)$ by
\begin{gather*}
z_+(y; t) = \frac{Y_{m_t}(t-1) }{1 + Y_{m_t}(t-1)},\qquad
z_-(y; t) = \frac{1}{1 + Y_{m_t}(t-1)},
\end{gather*}
where $Y_{m_t}(t-1)$ is the rational function in $y$ defined by~\eqref{eq:Y transition}.

\begin{Lemma}\label{lemma:unique fixed point}
 There is $\eta \in (\R_{>0})^n$ such that $\mu_\gamma(\eta) = \eta$, and such $\eta$ is unique.
\end{Lemma}
\begin{proof} We define the following two sets:
 \begin{gather*}
 \mathfrak{F}_\gamma^{>0} = \big\{ \eta \in (\R_{>0})^n \,|\, f_\gamma(\eta) = \eta \big\}, \\
 \mathfrak{D}_\gamma^{(0,1)} = \left\{ \zeta \in (0,1)^T \,|\, \prod_{t=1}^T \zeta_t^{-A_+^{et}} (1-\zeta_t)^{A_-^{et}} =1 \text{ for all $e \in E$} \right\}.
 \end{gather*}
 Using \cite[Corollary 3.9]{Mizuno20}, we find the map $\Phi\colon \mathfrak{F}_\gamma^{>0} \to \mathfrak{D}_\gamma^{(0,1)}$ defined by
 \begin{gather*}
 \Phi (\eta) = z_+ (\eta; t)
 \end{gather*}
 is a bijection.

 Since the matrix $A_+$ is invertible, the set $\mathfrak{D}_\gamma^{(0,1)}$ can be expressed as
 \begin{gather}\label{eq:another expression of D}
 \mathfrak{D}_\gamma^{(0,1)} = \left\{ \zeta \in (0,1)^T \,|\, \zeta_s = \prod_{t=1}^{T} (1-\zeta_t)^{(A_+^{-1} A_-)_{st}} \text{ for all $s=1,\dots , T$} \right\}.
 \end{gather}
 Using the fact that $A_+^{-1} A_-$ is positive definite and \cite[Lemma 2.1]{VlasenkoZwegers}, we find that the right-hand side of~\eqref{eq:another expression of D} consists of one element.
 Thus $\mathfrak{F}_\gamma^{>0}$ also consists of one element.
\end{proof}

The \emph{Rogers dilogarithm function} $L(x)$ is defined by
\begin{gather*}
L(x)= -\frac{1}{2} \int_{0}^{x} \left\{ \frac{\log(1-y)}{y} + \frac{\log y}{1-y} \right\} ,\qquad 0 \leq x \leq 1.
\end{gather*}
Using the Rogers dilogarithm function, we define a positive real number $a$ by
\begin{gather*}
a = \sum_{t=1}^T L( z_+(\eta; t) ) ,
\end{gather*}
where $\eta$ is the unique solution of the fixed point equation $\mu_\gamma (y)=y$ as in Lemma~\ref{lemma:unique fixed point}.

Let $J(y)$ be the Jacobian matrix of the rational function~$\mu_\gamma$:
\begin{gather*}
J_\gamma(y) := \left( \frac{\partial \mu_{i}}{\partial y_j} (y) \right)_{i,j \in \mathbf{I} },
\end{gather*}
where $\mu_i$ is the $i$-th components of $\mu_\gamma$.

Let $\varepsilon$ be a positive real number, and $\varepsilon \to 0$ denote the limit in the positive real axis.

\begin{Theorem}\label{thm:asymptotics of Z} Let $\gamma$ be a nondegenerate positive definite mutation loop. Then
 \begin{align}\label{eq:asymptotics of Z}
 \lim_{\varepsilon \to 0}\mathcal{Z}_\gamma \big({\rm e}^{-\varepsilon}\big) {\rm e}^{-\frac{a}{\varepsilon}} =
 \sqrt{\frac{\det A_+ }{\det(I - J_\gamma (\eta))}}.
 \end{align}
\end{Theorem}
\begin{proof}
 Using \cite[Theorem 2.3]{VlasenkoZwegers}, we find that
 \begin{gather*}
 \lim_{\varepsilon \to 0}\mathcal{Z}_\gamma \big({\rm e}^{-\varepsilon}\big) {\rm e}^{-\frac{a'}{\varepsilon}} = b',
 \end{gather*}
 where $a'$ and $b'$ are the real numbers given by
 \begin{gather*}
 a' = \sum_{t=1}^T \big( L(1) - L(z_-(\eta;t)) \big),\\
 b'= \det \big(A_+^{-1} A_- + \diag( \xi_1, \dots, \xi_T ) \big)^{-\frac{1}{2}} \prod_{t=1}^{T} z_+(\eta; t)^{-\frac{1}{2}},
 \end{gather*}
 and $\xi_t = z_-(\eta;t) / z_+(\eta;t)$.
 Since Rogers dilogarithm satisfies $L(x) + L(1-x) = L(1)$, we find that $a'=a$.

 Let $Z_\pm = \diag (z_\pm(\eta; 1) ,\dots, z_\pm(\eta; T) )$.
 Then we compute the number $b$ as follows:
 \begin{align*}
 b' &= \det \big(A_+^{-1} A_- + \diag( \xi_1, \dots, \xi_T ) \big)^{-\frac{1}{2}} \prod_{t=1}^{T} z_+(\eta; t)^{-\frac{1}{2}}\\
 &= \det \big(A_+^{-1} A_- + \diag( \xi_1, \dots, \xi_T ) \big)^{-\frac{1}{2}} (\det Z_+)^{-\frac{1}{2}}\\
 &=\det \big(A_+^{-1} A_- Z_+ + Z_- \big)^{-\frac{1}{2}}
 = (\det A_+)^{\frac{1}{2}} \det (A_+ Z_- + A_- Z_+)^{-\frac{1}{2}}.
 \end{align*}
 Using \cite[Theorem 4.2]{Mizuno20}, we find that
 \begin{gather*}
 \det (A_+ Z_- + A_- Z_+) = \det (I - J_\gamma(\eta)),
 \end{gather*}
 completing the proof.
\end{proof}

\begin{Remark}In~\cite{VlasenkoZwegers} (as well as~\cite{Ter,Zagier}), they not only gave the limit of $q$-series, but also gave a formula on the asymptotic expansion in the powers of $\varepsilon$. Although it might be interesting to investigate higher order terms of the asymptotic expansion, it is not dealt with in this paper.
\end{Remark}

\subsection[Partition $q$-series of the mutation loop on $Q(X_r, \ell)$]{Partition $\boldsymbol{q}$-series of the mutation loop on $\boldsymbol{Q(X_r, \ell)}$}\label{sec:partition q of gamma Xr level}

Let $X_r$ be a finite type Dynkin diagram, and $\level$ be a positive integer such that $\level \geq 2$.
Let $\gamma = \gamma(X_r, \level)$ be the mutation loop defined in Section~\ref{section: mutation loop Xr l}. In this section, we give an explicit expression of the partition $q$-series of $\gamma(X_r, \level)$.

For a Dynkin diagram of type $X_r$,
we define another Dynkin diagram of type $Y_{r'}$ as follows:
\begin{gather*}
 Y_{r'}=
 \begin{cases}
 X_r &\text{if $X=A,D$ or $E$},\\
 A_{2r-1} &\text{if $X_r=B_r$},\\
 D_{r+1} &\text{if $X_r=C_r$},\\
 E_6 &\text{if $X_r=F_4$},\\
 D_4 &\text{if $X_r=G_2$}.
 \end{cases}
\end{gather*}

\begin{figure}[t] \centering
 \begin{tikzpicture}
 [scale=1.2,auto=left,black_vertex/.style={circle,draw,fill,scale=0.75}]
 \node (a1) at (0,3) [black_vertex][label=above:$1$]{};
 \node (a2) at (1,3) [black_vertex][label=above:$2$]{};
 \node (ar-1) at (3,3) [black_vertex][label=above:$r-1$]{};
 \node (ar) at (4,2.5) [black_vertex][label=above:$r$]{};
 \node (ar+1) at (3,2) [black_vertex][label=above:$r+1$]{};
 \node (a2r-2) at (1,2) [black_vertex][label=above:$2r-2$]{};
 \node (a2r-1) at (0,2) [black_vertex][label=above:$2r-1$]{};
 \node (adots1) at (2,3) []{$\cdots$};
 \node (adots2) at (2,2) []{$\cdots$};
 \node (aname) at (-1,2.5) {$A_{2r-1}:$};
 \node (bname) at (-1,0) {$B_r:$};
 \node (1) at (0,0) [black_vertex][label=below:$1$]{};
 \node (2) at (1,0) [black_vertex][label=below:$2$]{};
 \node (dots) at (2,0) []{$\cdots$};
 \node (r-1) at (3,0) [black_vertex][label=below:$r-1$]{};
 \node (r) at (4,0) [black_vertex][label=below:$r$]{};
 \draw [] (1)--(2);
 \draw [] (2)--(dots);
 \draw [] (dots)--(r-1);
 \draw [] (3,0.08)--(4,0.08);
 \draw [] (3,-0.08)--(4,-0.08);
 \draw
 (3.5,0) --++ (120:.3)
 (3.5,0) --++ (-120:.3);
 \draw [] (a1)--(a2);
 \draw [] (a2)--(adots1);
 \draw [] (ar-1)--(ar);
 \draw [] (ar)--(ar+1);
 \draw [] (adots2)--(a2r-2);
 \draw [] (a2r-2)--(a2r-1);
 \draw [dashed,arrows={-Stealth[scale=1.2]}] (a1) edge[bend right=25](1);
 \draw [dashed,arrows={-Stealth[scale=1.2]}] (a2r-1) edge[bend left=15](1);
 \draw [dashed,arrows={-Stealth[scale=1.2]}] (a2) edge[bend right=25](2);
 \draw [dashed,arrows={-Stealth[scale=1.2]}] (a2r-2) edge[bend left=15](2);
 \draw [dashed,arrows={-Stealth[scale=1.2]}] (ar-1) edge[bend right=25](r-1);
 \draw [dashed,arrows={-Stealth[scale=1.2]}] (ar+1) edge[bend left=15](r-1);
 \draw [dashed,arrows={-Stealth[scale=1.2]}] (ar) edge[bend right=0](r);
 \end{tikzpicture}

 \begin{tikzpicture}
 [scale=1.2,auto=left,black_vertex/.style={circle,draw,fill,scale=0.75}]
 \node (d1) at (0,2.5) [black_vertex][label=above:$1$]{};
 \node (d2) at (1,2.5) [black_vertex][label=above:$2$]{};
 \node (dr-1) at (3,2.5) [black_vertex][label=above:$r-1$]{};
 \node (dr) at (4,3) [black_vertex][label=above:$r$]{};
 \node (dr+1) at (4,2) [black_vertex][label=above:$r+1$]{};
 \node (ddots1) at (2,2.5) []{$\cdots$};
 \node (dname) at (-1,2.5) {$D_{r+1}:$};
 \node (cname) at (-1,0) {$C_r:$};
 \node (1) at (0,0) [black_vertex][label=below:$1$]{};
 \node (2) at (1,0) [black_vertex][label=below:$2$]{};
 \node (dots) at (2,0) []{$\cdots$};
 \node (r-1) at (3,0) [black_vertex][label=below:$r-1$]{};
 \node (r) at (4,0) [black_vertex][label=below:$r$]{};
 \draw [] (1)--(2);
 \draw [] (2)--(dots);
 \draw [] (dots)--(r-1);
 \draw [] (3,0.08)--(4,0.08);
 \draw [] (3,-0.08)--(4,-0.08);
 \draw
 (3.5,0) --++ (60:.3)
 (3.5,0) --++ (-60:.3);
 \draw [] (d1)--(d2);
 \draw [] (d2)--(ddots1);
 \draw [] (dr-1)--(dr);
 \draw [] (dr-1)--(dr+1);
 \draw [dashed,arrows={-Stealth[scale=1.2]}] (d1) edge[bend right=0](1);
 \draw [dashed,arrows={-Stealth[scale=1.2]}] (d2) edge[bend right=0](2);
 \draw [dashed,arrows={-Stealth[scale=1.2]}] (dr-1) edge[bend right=0](r-1);
 \draw [dashed,arrows={-Stealth[scale=1.2]}] (dr+1) edge[bend left=15](r);
 \draw [dashed,arrows={-Stealth[scale=1.2]}] (dr) edge[bend right=25](r);
 \end{tikzpicture}
 \caption{The map $\tau$ for $X_r =B_r, C_r$.} \label{fig:tau B and C}
\end{figure}
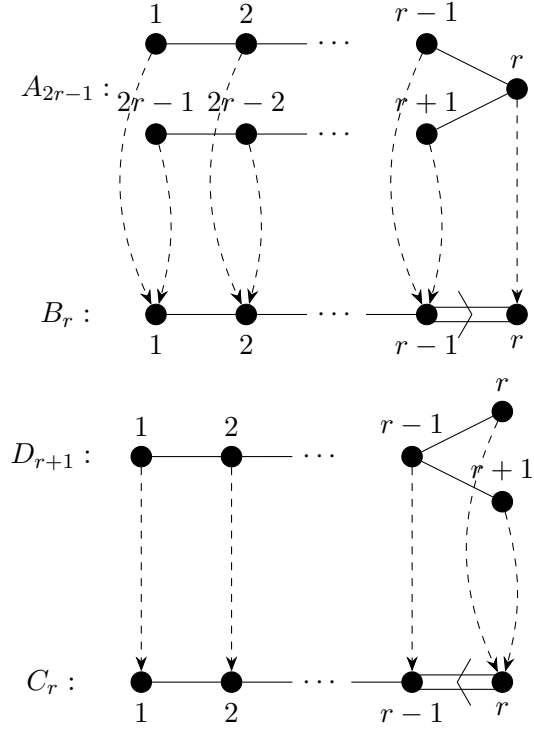

We define a map
\begin{gather*}
 \tau \colon \ \{ 1, \dots , r' \} \to \{ 1, \dots , r \}
\end{gather*}
between the nodes of these two Dynkin diagrams by
\begin{gather*}
 \tau(a) = a
\end{gather*}
if $X= A ,D$ or $E$,
\begin{gather*}
 \tau(a) =
 \begin{cases}
 a &\text{if $1 \leq a \leq r$},\\
 2r-a &\text{if $r+1 \leq a \leq 2r-1$}
 \end{cases}
\end{gather*}
if $X_r = B_r$,
\begin{gather*}
\tau(a) =
\begin{cases}
a &\text{if $1 \leq a \leq r$,}\\
r &\text{if $a=r+1$}
\end{cases}
\end{gather*}
if $X_r = C_r$,
\begin{gather*}
\tau(a) =
\begin{cases}
a &\text{if $1 \leq a \leq 4$,}\\
2 &\text{if $a=5$,}\\
1 &\text{if $a=6$}
\end{cases}
\end{gather*}
if $X_r = F_4$, and
\begin{gather*}
\tau(a) =
\begin{cases}
1 &\text{if $a=1,3$ or $4$,}\\
2 &\text{if $a=2$}
\end{cases}
\end{gather*}
if $X_r = G_2$. These definitions are indicated in Figs.~\ref{fig:tau B and C} and~\ref{fig:tau F and G}.

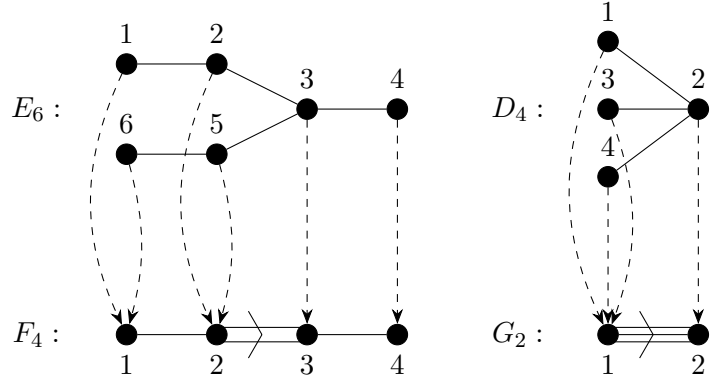
\begin{figure}[t] \centering
 \begin{tikzpicture}
 [scale=1.2,auto=left,black_vertex/.style={circle,draw,fill,scale=0.75}]
 \node (e1) at (0,3) [black_vertex][label=above:$1$]{};
 \node (e2) at (1,3) [black_vertex][label=above:$2$]{};
 \node (e3) at (2,2.5) [black_vertex][label=above:$3$]{};
 \node (e4) at (3,2.5) [black_vertex][label=above:$4$]{};
 \node (e5) at (1,2) [black_vertex][label=above:$5$]{};
 \node (e6) at (0,2) [black_vertex][label=above:$6$]{};
 \node (ename) at (-1,2.5) {$E_6:$};
 \node (fname) at (-1,0) {$F_4:$};
 \node (1) at (0,0) [black_vertex][label=below:$1$]{};
 \node (2) at (1,0) [black_vertex][label=below:$2$]{};
 \node (3) at (2,0) [black_vertex][label=below:$3$]{};
 \node (4) at (3,0) [black_vertex][label=below:$4$]{};
 \draw [] (1)--(2);
 \draw [] (1,0.08)--(2,0.08);
 \draw [] (1,-0.08)--(2,-0.08);
 \draw [] (3)--(4);
 \draw
 (1.5,0) --++ (120:.3)
 (1.5,0) --++ (-120:.3);
 \draw [] (e1)--(e2);
 \draw [] (e2)--(e3);
 \draw [] (e3)--(e4);
 \draw [] (e3)--(e5);
 \draw [] (e5)--(e6);
 \draw [dashed,arrows={-Stealth[scale=1.2]}] (e1) edge[bend right=25](1);
 \draw [dashed,arrows={-Stealth[scale=1.2]}] (e2) edge[bend right=25](2);
 \draw [dashed,arrows={-Stealth[scale=1.2]}] (e3) edge[bend right=0](3);
 \draw [dashed,arrows={-Stealth[scale=1.2]}] (e4) edge[bend left=0](4);
 \draw [dashed,arrows={-Stealth[scale=1.2]}] (e5) edge[bend left=15](2);
 \draw [dashed,arrows={-Stealth[scale=1.2]}] (e6) edge[bend left=15](1);
 \end{tikzpicture}
 \qquad
 \begin{tikzpicture}
 [scale=1.2,auto=left,black_vertex/.style={circle,draw,fill,scale=0.75}]
 \node (d1) at (0,3.25) [black_vertex][label=above:$1$]{};
 \node (d2) at (1,2.5) [black_vertex][label=above:$2$]{};
 \node (d3) at (0,2.5) [black_vertex][label=above:$3$]{};
 \node (d4) at (0,1.75) [black_vertex][label=above:$4$]{};
 \node (dname) at (-1,2.5) {$D_4:$};
 \node (gname) at (-1,0) {$G_2:$};
 \node (1) at (0,0) [black_vertex][label=below:$1$]{};
 \node (2) at (1,0) [black_vertex][label=below:$2$]{};
 \draw [] (1)--(2);
 \draw [] (0,0.08)--(1,0.08);
 \draw [] (0,-0.08)--(1,-0.08);
 \draw
 (0.5,0) --++ (120:.3)
 (0.5,0) --++ (-120:.3);
 \draw [] (d1)--(d2);
 \draw [] (d2)--(d3);
 \draw [] (d2)--(d4);
 \draw [dashed,arrows={-Stealth[scale=1.2]}] (d1) edge[bend right=25](1);
 \draw [dashed,arrows={-Stealth[scale=1.2]}] (d2) edge[bend right=0](2);
 \draw [dashed,arrows={-Stealth[scale=1.2]}] (d3) edge[bend left=20](1);
 \draw [dashed,arrows={-Stealth[scale=1.2]}] (d4) edge[bend left=0](1);
 \end{tikzpicture}
 \caption{The map $\tau$ for $X_r =F_4, G_2$.} \label{fig:tau F and G}
\end{figure}
We define integers $\kappa_a$ $(a=1 ,\dots , r')$ by
\begin{gather*}
 \kappa_a = \# \{ b \,|\, 1 \leq b \leq r' ,\, \tau( a ) = \tau(b) \}.
\end{gather*}
Then the vertices in the quiver $Q(X_r , \level)$ can be parameterized by the elements of the set
\begin{gather*}
 \mathbf{I} = \{ (a,m) \,|\, 1 \leq a \leq r' ,\, 1 \leq m \leq \kappa_a \level -1 \}.
\end{gather*}
Here we use the letter $\mathbf{I}$ so that we identify the set of the vertices in the quiver with this set.\footnote{This parametrization is slightly different from that in Section~\ref{section: type G mutation loop} that we used to define a mutation loop on~$Q(G_2,\level)$.}

Let $\mathbf{J}$ be the set defined by
\begin{gather*}
 \mathbf{J} = \{ (a,m) \,|\, 1\leq a\leq r , \, 1\leq m\leq t_a \level -1 \}.
\end{gather*}
\begin{Lemma}\label{lemma: tau prime} The map
 \begin{gather*}
 \tau' \colon \ \mathbf{I} \longrightarrow \mathbf{J}
 \end{gather*}
 defined by $\tau' (a,m) = (\tau(a),m)$ is well-defined and surjective.
\end{Lemma}
\begin{proof} The lemma follows from the fact that $\kappa_a = t_b$ if $\tau(a)=b$.
\end{proof}
\begin{Lemma}\label{lemma: tau prime tilde}
 The induced map
 \begin{gather*}
 \tilde{\tau}' \colon \ \mathbf{I} / \nu \longrightarrow \mathbf{J}
 \end{gather*}
 defined by $\tilde{\tau}' ([(a,m)]) = \tau'(a,m)$ is well-defined and bijective.
\end{Lemma}
\begin{proof}From the definitions of $\nu$ and $\tau'$, we see that $(a,m)$ and $(b,k)$ belong to the same $\nu$-orbit if and only if $\tau'(a,m) = \tau'(b,k)$. Thus $\tilde{\tau}'$ is well-defined and injective.
 It is also surjective from Lemma~\ref{lemma: tau prime}.
\end{proof}

Using Lemmas~\ref{lemma:gamma is regular}, \ref{lemma: regular mutation loop} and~\ref{lemma: tau prime tilde}, we identify $E$ and $\Delta$ with $\mathbf{J}$.
In particular, we think that the Neumann--Zagier matrices $A_+$ and $A_-$ are $\mathbf{J} \times \mathbf{J}$ matrices.

Let $A_{\pm}^{ab,mk}$ be the $((a,m),(b,k))$-entries of the Neumann--Zagier matrices $A_{\pm}$.
Let $K$ be the $J_\level \times J_\level$ matrix defined by
\begin{gather*}
K_{ab,mk} = \left( \min (t_b m, t_a k) - \frac{mk}{\level} \right) (\alpha_a \,|\, \alpha_b),
\end{gather*}
where $K_{ab,mk}$ is the $((a,m),(b,k))$-entry of $K$.
Let $\bar{C}^a =\big(\bar{C}_{mk}^a\big)_{m,k=1 ,\dots ,t_a \level -1}$ be the Cartan matrix of type $A_{t_a \level -1}$:
\begin{gather*}
 \bar{C}_{mk}^a = 2 \delta_{mk} - \delta_{m,k-1} - \delta_{m,k+1}.
\end{gather*}

\begin{Proposition}\label{prop: A plus minus}
 The Neumann--Zagier matrices of $\gamma(X_r , \level)$ are given by
 \begin{gather*}
 A_{+}^{ab,mk} = \delta_{ab} \bar{C}_{mk}^a, \\
 A_{-}^{ab,mk} = \sum_{n=1}^{t_a \level -1} \bar{C}_{mn}^a K_{ab,nk} .
 \end{gather*}
\end{Proposition}
\begin{proof} First, we note that the components of $A_+$ and $A_-$ come from the arrows of the vertical and horizontal directions, respectively, in the figures of quivers in Section~\ref{section: quiver Xr l}.
 Then, the formula for $A_+$ that we want follows from the definition of $\gamma(X_r, \level)$ because the components of $A_+$ come from arrows in the vertical $A_{t_a \level -1}$ quivers.
 We can also verify from the definition of $\gamma(X_r, \level)$ that the following formula for $A_-$ holds:
 \begin{gather*}
 A_{-}^{ab,mk} =
 \begin{cases}
 -C_{ab} (\delta_{2m,k-1} + 2 \delta_{2m,k} + \delta_{2m,k+1} ) & \text{if $t_b / t_a =2$},\\
 -C_{ab} (\delta_{3m,k-2} + 2 \delta_{3m,k-1} + 3 \delta_{3m,k}
 + 2 \delta_{3m,k+1} + \delta_{3m,k+2} ) & \text{if $t_b / t_a =3$},\\
 -C_{ba} \delta_{t_b m, t_a k} & \text{otherwise},
 \end{cases}
 \end{gather*}
 where $C_{ab}$ is the $(a,b)$-entry of the Cartan matrix of type $X_r$:
 \begin{gather*}
 C_{ab} = \frac{2\innerproduct{\alpha_a}{\alpha_b}}{\innerproduct{\alpha_a}{\alpha_a}}.
 \end{gather*}
 By concrete calculation, we can see that the right-hand side of this equation coincides with
 \begin{gather*}
 \sum_{n=1}^{t_a \level -1} \bar{C}_{mn}^a K_{ab,nk} ,
 \end{gather*}
 and this complete the proof.
\end{proof}

\begin{Corollary}\label{corollary: K = Ap inv Am}
 The following relation holds:
 \begin{gather*}
 K = A_{+}^{-1} A_-.
 \end{gather*}
 Moreover, $\gamma = \gamma(X_r,\level)$ is nondegenerate and positive definite,
 and $\mu_{\gamma}(y)=y$ has a unique positive real solution.
\end{Corollary}
\begin{proof}
 The equation follows from Proposition~\ref{prop: A plus minus}.
 The positive definiteness of~$K$ is well-known (e.g., see the proof of Proposition~1.8 in~\cite{IIKKNa}).
 The last argument follows from Lemma~\ref{lemma:unique fixed point}.
\end{proof}

Let $Q$ be the free abelian group defined by
\begin{gather*}
Q= \bigoplus_{a=1}^r \Z \alpha_a.
\end{gather*}
The free abelian group $Q$ is called a root lattice.
Let $M$ be the free abelian subgroup of $Q$ defined by
\begin{gather*}
M=\bigoplus_{a=1}^r \Z t_a \alpha_a .
\end{gather*}

For any element $u \in \Z^{\mathbf{J}} $, let $u_m^{(a)}$ denote the $(a,m)$-entry of $u$.

\begin{Lemma}\label{lemma: isom H between Q/lM}
 The group homomorphism
 \begin{gather*}
 \bar{F}\colon \ Z_\gamma \longrightarrow Q
 \end{gather*}
 defined by
 \begin{align*}
 \bar{F}( u,v ) =
 \sum_{(a,m) \in \mathbf{J}}
 m u_m^{(a)} \alpha_a
 \end{align*}
 is surjective, and the kernel of $\iota \circ \bar{F}$ is equal to $B_\gamma$, where $\iota\colon Q \to Q/\level M$ is the projection.
 Thus we obtain the group isomorphism
 \begin{gather}\label{eq: isom H between Q/lM}
 F\colon \ H_\gamma \longrightarrow Q / \level M .
 \end{gather}
\end{Lemma}
\begin{proof} Since $A_+$ is non-singular, the abelian group $Z_\gamma$ is given by~\eqref{eq: non-singular Z}. Thus the map $\bar{F}$ is surjective since
 \begin{gather*}
 \bar{F}(u,v) = \sum_{a=1}^r u_1^{(a)} \alpha_a
 \end{gather*}
 if $u_m^{(a)} =0$ for all $m>1$.

 For any element $w \in \Z^{\mathbf{J}}$, we compute
 \begin{align*}
 \bar{F}\big( A_+^{\mathsf{T}} w,A_-^{\mathsf{T}} w\big)
 &= \sum_{(a,m) \in \mathbf{J}}
 m \left( \sum_{(b,k) \in \mathbf{J}} \delta_{ab} \bar{C}_{mk}^a w_k^{(b)} \right) \alpha_a \\
 &= \sum_{a=1}^r
 \left(\sum_{m=1}^{t_a \level -1} 2 m w_m^{(a)}
 -\sum_{m=1}^{t_a \level -2} (m+1) w_{m}^{(a)}
 -\sum_{m=2}^{t_a \level -1} (m-1) w_{m}^{(a)} \right)
 \alpha_a \\
 &= \sum_{a=1}^r t_a \level w_{t_a \level -1}^{(a)} \alpha_a,
 \end{align*}
 and this shows that $B_\gamma = \ker \big(\iota \circ \bar{F} \big)$ as desired.
\end{proof}

\begin{Theorem}\label{theorem: partition q-series Xr l}
 The partition $q$-series of $\gamma = \gamma(X_r, \level)$ associated with $\sigma \in H_\gamma$ is given by
 \begin{gather*}
 \mathcal{Z}_{\gamma}^{\sigma} (q) = \sum_{u \in (\Z_{\geq 0})^{\mathbf{J}} ,(\Diamond) }
 \frac{q^{\frac{1}{2} u^{\mathsf{T}} K u} }{ \prod\limits_{(a,m) \in \mathbf{J}} (q)_{u_m^{(a)}} },
 \end{gather*}
 where the sum runs over $u \in (\Z_{\geq 0})^{\mathbf{J}}$ under the condition
 \begin{gather}
 \tag{$\Diamond$}
 \sum_{(a,m) \in \mathbf{J}} m u_m^{(a)} \alpha_a \equiv \lambda \mod \level M,
 \end{gather}
 where $\lambda$ is a representative of $F(\sigma) $.
\end{Theorem}
\begin{proof} The theorem follows from Corollary \ref{corollary: K = Ap inv Am} and Lemma~\ref{lemma: isom H between Q/lM}.
\end{proof}

\begin{Corollary} The total partition $q$-series of $\gamma=\gamma(X_r , \level)$ is given by
 \begin{gather*}
 \mathcal{Z}_\gamma (q) =
 \sum_{u \in (\Z_{\geq 0})^\mathbf{J} }
 \frac{q^{\frac{1}{2} u^{\mathsf{T}} K u } }{ \prod\limits_{(a,m) \in \mathbf{J}} (q)_{u_m^{(a)}} }.
 \end{gather*}
\end{Corollary}

\begin{Example}\label{example: A3 level3}
 Let $(X_r, \level) = (A_3,3)$.
 The quiver $Q(A_3, 3)$ is given by
 \[
 \begin{tikzpicture}
 [scale=1.0,auto=left,black_vertex/.style={circle,draw,fill,scale=0.75},white_vertex/.style={circle,draw,scale=0.75}]
 \node (11) at (2,2) [black_vertex][label=below right:{$(1,1)$}]{};
 \node at (2.3,2.3) []{$-$};
 \node (12) at (2,4) [black_vertex][label=below right:{$(1,2)$}]{};
 \node at (2.3,4.3) []{$+$};
 \node (21) at (4,2) [black_vertex][label=below right:{$(2,1)$}]{};
 \node at (4.3,2.3) []{$+$};
 \node (22) at (4,4) [black_vertex][label=below right:{$(2,2)$}]{};
 \node at (4.3,4.3) []{$-$};
 \node (31) at (6,2) [black_vertex][label=below right:{$(3,1)$}]{};
 \node at (6.3,2.3) []{$-$};
 \node (32) at (6,4) [black_vertex][label=below right:{$(3,2)$}]{};
 \node at (6.3,4.3) []{$+$};
 \foreach \from/\to in {12/11,
 21/22,
 32/31,
 11/21,
 22/12,
 22/32,
 31/21}
 \draw[arrows={-Stealth[scale=1.5]}] (\from)--(\to);
 \end{tikzpicture}
 \]
 The set of indices of $Q(A_3, 3)$ is
 \begin{gather*}
 \mathbf{I} = \{ (1,1),(1,2),(2,1),(2,2),(3,1),(3,2) \} .
 \end{gather*}

 The mutation loop $\gamma = \gamma(A_3, 3)$ is given by
 \begin{align*}
 &\begin{tikzpicture}
 [scale=0.75,auto=left,black_vertex/.style={circle,draw,fill,scale=0.75},white_vertex/.style={circle,draw,scale=0.75}]
 \node (11) at (2,2) [black_vertex]{};
 \node at (2.3,2.3) []{$-$};
 \node (12) at (2,4) [black_vertex]{};
 \node at (2.3,4.3) []{$+$};
 \node (21) at (4,2) [black_vertex]{};
 \node at (4.3,2.3) []{$+$};
 \node (22) at (4,4) [black_vertex]{};
 \node at (4.3,4.3) []{$-$};
 \node (31) at (6,2) [black_vertex]{};
 \node at (6.3,2.3) []{$-$};
 \node (32) at (6,4) [black_vertex]{};
 \node at (6.3,4.3) []{$+$};
 \foreach \from/\to in {12/11,
 21/22,
 32/31,
 11/21,
 22/12,
 22/32,
 31/21}
 \draw[arrows={-Stealth[scale=1.5]}] (\from)--(\to);
 \end{tikzpicture}\\
 \raisebox{8mm}{$\, \xrightarrow{\mu_{+}}\,$} \,\,\,
 &\begin{tikzpicture}
 [scale=0.75,auto=left,black_vertex/.style={circle,draw,fill,scale=0.75},white_vertex/.style={circle,draw,scale=0.75}]
 \node (11) at (2,2) [black_vertex]{};
 \node at (2.3,2.3) []{$-$};
 \node (12) at (2,4) [black_vertex]{};
 \node at (2.3,4.3) []{$+$};
 \node (21) at (4,2) [black_vertex]{};
 \node at (4.3,2.3) []{$+$};
 \node (22) at (4,4) [black_vertex]{};
 \node at (4.3,4.3) []{$-$};
 \node (31) at (6,2) [black_vertex]{};
 \node at (6.3,2.3) []{$-$};
 \node (32) at (6,4) [black_vertex]{};
 \node at (6.3,4.3) []{$+$};
 \foreach \from/\to in {11/12,
 22/21,
 31/32,
 21/11,
 12/22,
 32/22,
 21/31}
 \draw[arrows={-Stealth[scale=1.5]}] (\from)--(\to);
 \end{tikzpicture}\\
 \raisebox{8mm}{$\, \xrightarrow{\mu_{-}}\,$} \,\,\,
 &\begin{tikzpicture}
 [scale=0.75,auto=left,black_vertex/.style={circle,draw,fill,scale=0.75},white_vertex/.style={circle,draw,scale=0.75}]
 \node (11) at (2,2) [black_vertex]{};
 \node at (2.3,2.3) []{$-$};
 \node (12) at (2,4) [black_vertex]{};
 \node at (2.3,4.3) []{$+$};
 \node (21) at (4,2) [black_vertex]{};
 \node at (4.3,2.3) []{$+$};
 \node (22) at (4,4) [black_vertex]{};
 \node at (4.3,4.3) []{$-$};
 \node (31) at (6,2) [black_vertex]{};
 \node at (6.3,2.3) []{$-$};
 \node (32) at (6,4) [black_vertex]{};
 \node at (6.3,4.3) []{$+$};
 \foreach \from/\to in {12/11,
 21/22,
 32/31,
 11/21,
 22/12,
 22/32,
 31/21}
 \draw[arrows={-Stealth[scale=1.5]}] (\from)--(\to);
 \end{tikzpicture}
 \end{align*}

 The mutation network $\network_\gamma$ is given by
 \[
 \begin{tikzpicture}
 [scale=1.7,auto=left,black_vertex/.style={circle,draw,fill,scale=0.75},square_vertex/.style={rectangle,scale=0.8,draw}]
 \node (a1) at (0,0) [black_vertex][label=above:{$(1,1)$}]{};
 \node (a2) at (0-0.45,0+0.9) [black_vertex][label=above:{$(1,2)$}]{};
 \node (b1) at (1,0) [black_vertex][label=above:{$(2,1)$}]{};
 \node (b2) at (1-0.45,0+0.9) [black_vertex][label=above:{$(2,2)$}]{};
 \node (c1) at (2,0) [black_vertex][label=above:{$(3,1)$}]{};
 \node (c2) at (2-0.45,0+0.9) [black_vertex][label=above:{$(3,2)$}]{};
 \node (ta1) at (0,0-1.0) [square_vertex][label=below:{$(1,1)$}]{};
 \node (ta2) at (0-0.45,0+0.9-1.0) [square_vertex][label=below:{$(1,2)$}]{};
 \node (tb1) at (1,0-1.0) [square_vertex][label=below:{$(2,1)$}]{};
 \node (tb2) at (1-0.45,0+0.9-1.0) [square_vertex][label=below:{$(2,2)$}]{};
 \node (tc1) at (2,0-1.0) [square_vertex][label=below:{$(3,1)$}]{};
 \node (tc2) at (2-0.45,0+0.9-1.0) [square_vertex][label=below:{$(3,2)$}]{};
 \draw [dashed] (a1)edge [bend left=15](ta1);
 \draw [dashed] (ta1)edge [bend left=15](a1);
 \draw [dashed] (a2)edge [bend left=15](ta2);
 \draw [dashed] (ta2)edge [bend left=15](a2);
 \draw [dashed] (b1)edge [bend left=15](tb1);
 \draw [dashed] (tb1)edge [bend left=15](b1);
 \draw [dashed] (b2)edge [bend left=15](tb2);
 \draw [dashed] (tb2)edge [bend left=15](b2);
 \draw [dashed] (c1)edge [bend left=15](tc1);
 \draw [dashed] (tc1)edge [bend left=15](c1);
 \draw [dashed] (c2)edge [bend left=15](tc2);
 \draw [dashed] (tc2)edge [bend left=15](c2);
 \draw[arrows={-Stealth[scale=1.2]}] (a1)--(tb1);
 \draw[arrows={-Stealth[scale=1.2]}] (b1)--(ta1);
 \draw[arrows={-Stealth[scale=1.2]}] (b1)--(tc1);
 \draw[arrows={-Stealth[scale=1.2]}] (c1)--(tb1);
 \draw[arrows={-Stealth[scale=1.2]}] (a2)--(tb2);
 \draw[arrows={-Stealth[scale=1.2]}] (b2)--(ta2);
 \draw[arrows={-Stealth[scale=1.2]}] (b2)--(tc2);
 \draw[arrows={-Stealth[scale=1.2]}] (c2)--(tb2);
 \draw[arrows={-Stealth[scale=1.2]}] (ta1)--(a2);
 \draw[arrows={-Stealth[scale=1.2]}] (tb1)--(b2);
 \draw[arrows={-Stealth[scale=1.2]}] (tc1)--(c2);
 \draw[arrows={-Stealth[scale=1.2]}] (ta2)--(a1);
 \draw[arrows={-Stealth[scale=1.2]}] (tb2)--(b1);
 \draw[arrows={-Stealth[scale=1.2]}] (tc2)--(c1);
 \end{tikzpicture}
 \]
 The set of indices of $\network_\gamma$ is
 \begin{gather}\label{eq:J A3 3}
 \mathbf{J} = \{ (1,1),(1,2),(2,1),(2,2),(3,1),(3,2) \} .
 \end{gather}
 The adjacency matrices of $\network_\gamma$ are given by $N_0 = 2 I$ and
 \begin{gather*}
 N_+=
 \begin{pmatrix}
 0 & 1 & 0 & 0 & 0 & 0\\
 1 & 0 & 0 & 0 & 0 & 0\\
 0 & 0 & 0 & 1 & 0 & 0\\
 0 & 0 & 1 & 0 & 0 & 0\\
 0 & 0 & 0 & 0 & 0 & 1\\
 0 & 0 & 0 & 0 & 1 & 0
 \end{pmatrix} ,\qquad
 N_-=
 \begin{pmatrix}
 0 & 0 & 1 & 0 & 0 & 0\\
 0 & 0 & 0 & 1 & 0 & 0\\
 1 & 0 & 0 & 0 & 1 & 0\\
 0 & 1 & 0 & 0 & 0 & 1\\
 0 & 0 & 1 & 0 & 0 & 0\\
 0 & 0 & 0 & 1 & 0 & 0
 \end{pmatrix} ,
 \end{gather*}
 so the Neumann--Zagier matrices of $\gamma$ are given by
 \begin{gather*}
 A_+ =
 \begin{pmatrix}
 2 & -1 & 0 & 0 & 0 & 0\\
 -1 & 2 & 0 & 0 & 0 & 0\\
 0 & 0 & 2 & -1 & 0 & 0\\
 0 & 0 & -1 & 2 & 0 & 0\\
 0 & 0 & 0 & 0 & 2 & -1\\
 0 & 0 & 0 & 0 & -1 & 2
 \end{pmatrix},\qquad
 A_-=
 \begin{pmatrix}
 2 & 0 & -1 & 0 & 0 & 0\\
 0 & 2 & 0 & -1 & 0 & 0\\
 -1 & 0 & 2 & 0 & -1 & 0\\
 0 & -1 & 0 & 2 & 0 & -1\\
 0 & 0 & -1 & 0 & 2 & 0\\
 0 & 0 & 0 & -1 & 0 & 2
 \end{pmatrix} ,
 \end{gather*}
 where we choose an order of $\mathbf{J}$ as in \eqref{eq:J A3 3}.

 Lemma \ref{lemma: NZ symplectic property} implies that
 \begin{gather*}
 A_+ A_-^{\mathsf{T}} = A_- A_+^{\mathsf{T}} =
 \begin{pmatrix}
 4 & -2 & -2 & 1 & 0 & 0 \\
 -2 & 4 & 1 & -2 & 0 & 0 \\
 -2 & 1 & 4 & -2 & -2 & 1 \\
 1 & -2 & -2 & 4 & 1 & -2 \\
 0 & 0 & -2 & 1 & 4 & -2 \\
 0 & 0 & 1 & -2 & -2 & 4
 \end{pmatrix}
 \end{gather*}
 is a symmetric matrix.
 When $X=A,D$ or $E$, the Neumann--Zagier matrices $A_+$ and $A_-$ themselves are symmetric, so we find that they are commuting, that is, $A_+ A_- = A_- A_+$.
 Furthermore, they are (possibly decomposable) symmetric Cartan matrices.
 We remark that such a pair of commuting Cartan matrices appeared in the study of $W$-graphs by Stembridge~\cite{Stembridge}, and also in the classification of periodic mutations by Galashin and Pylyavskyy~\cite{GalashinPylyavskyy}.

 The positive definite matrix $A_+^{-1} A_-$ is given by
 \begin{gather*}
 A_+^{-1} A_- = \frac{1}{3}
 \begin{pmatrix}
 4 & 2 & -2 & -1 & 0 & 0 \\
 2 & 4 & -1 & -2 & 0 & 0 \\
 -2 & -1 & 4 & 2 & -2 & -1 \\
 -1 & -2 & 2 & 4 & -1 & -2 \\
 0 & 0 & -2 & -1 & 4 & 2 \\
 0 & 0 & -1 & -2 & 2 & 4
 \end{pmatrix}.
 \end{gather*}

 The partition $q$-series are parametrized by elements of the following set:
 \begin{gather*}
 H_\gamma = \{ (a,0,b,0,c,0) \,|\, 0 \leq a,b,c, \leq 2 \},
 \end{gather*}
where $(u,v) + Z_\gamma \in H_\gamma$ is denoted by $\big(u_1^{(1)},u_1^{(2)},u_1^{(3)},u_2^{(3)},u_3^{(3)}\big)$. Let us write $(a,0,b,0,c,0)$ more simply as $abc$. Then we have
 \begin{gather*}\begin{split}
 &H_\gamma = \{ 000,100,200,010,110,210,020,120,220,
 001,101,201,011,111,211,021,121,221, \\
 &\hphantom{H_\gamma = \{}{} 002,102,202,012,112,212,022,122,222 \}.
 \end{split}
 \end{gather*}

 We exhibit several low order terms of the partition $q$-series:
 \begin{gather*}
 \mathcal{Z}_\gamma^\sigma (q) =
 1 + 6q^{2} + 20q^{3} + 54q^{4} + 144q^{5} + 360q^{6} + 804q^{7} + O\big(q^{8}\big)
 \end{gather*}
 if $\sigma = 000$,
 \begin{gather*}
 q^{\frac{2}{3}} + 3q^{\frac{5}{3}} + 13q^{\frac{8}{3}} + 38q^{\frac{11}{3}} + 108q^{\frac{14}{3}} + 264q^{\frac{17}{3}} + 622q^{\frac{20}{3}} + 1364q^{\frac{23}{3}} +O\big(q^{\frac{26}{3}}\big)
 \end{gather*}
 if $\sigma = 100,200,010,110,020,220,001,011,111,002,022,222$,
 \begin{gather*}
 q + 6q^{2} + 18q^{3} + 56q^{4} + 144q^{5} + 357q^{6} + 808q^{7} + 1767q^{8} + O\big(q^{9}\big)
 \end{gather*}
 if $\sigma = 210,120,211,021,221,012,112,122$, and
 \begin{gather*}
 2q^{\frac{4}{3}} + 8q^{\frac{7}{3}} + 28q^{\frac{10}{3}} + 76q^{\frac{13}{3}} + 199q^{\frac{16}{3}} + 468q^{\frac{19}{3}} + 1060q^{\frac{22}{3}} + 2256q^{\frac{25}{3}} + O\big(q^{\frac{28}{3}}\big)
 \end{gather*}
 if $\sigma = 101,201,121,102,202,212$.
\end{Example}

\begin{Example}\label{example: B3 level 2}
 Let $(X_r, \level) = (B_3,2)$.
 The quiver $Q(B_3, 2)$ is given by
 \[
 \begin{tikzpicture}
 [scale=1.0,auto=left,black_vertex/.style={circle,draw,fill,scale=0.75},white_vertex/.style={circle,draw,scale=0.75}]
 \node (a) at (4,2) [white_vertex][label=below:{$(1,1)$}]{};
 \node at (4.3,2.3) [] {$+$};
 \node (b) at (12,2) [white_vertex][label=below:{$(5,1)$}]{};
 \node at (12.3,2.3) [] {$-$};
 \node (31) at (6,2) [white_vertex][label=below:{$(2,1)$}]{};
 \node at (6.3,2.3) []{$-$};
 \node (41) at (8,1) [black_vertex][label=below right:{$(3,1)$}]{};
 \node at (8.3,1.3) []{$+$};
 \node (42) at (8,2) [black_vertex][label=below right:{$(3,2)$}]{};
 \node at (8.3,2.3) []{$-$};
 \node (43) at (8,3) [black_vertex][label=below right:{$(3,3)$}]{};
 \node at (8.3,3.3) []{$+$};
 \node (51) at (10,2) [white_vertex][label=below:{$(4,1)$}]{};
 \node at (10.3,2.3) []{$+$};
 \foreach \from/\to in {41/42,43/42,
 31/41,
 31/43,
 42/51}
 \draw[arrows={-Stealth[scale=1.5]}] (\from)--(\to);
 \foreach \from/\to in {42/31,31/a,b/51}
 \draw[arrows={-Stealth[scale=1.5]}] (\from)--(\to);
 \end{tikzpicture}
 \]
 The set of indices of $Q(B_3, 2)$ is
 \begin{align*}
 \mathbf{I} = \{ (1,1),(2,1),(3,1),(3,2),(3,3),(4,1),(5,1) \} .
 \end{align*}

 The mutation loop $\gamma = \gamma(B_3, 2)$ is given by
 \begin{align*}
 &\begin{tikzpicture}
 [scale=0.75,auto=left,black_vertex/.style={circle,draw,fill,scale=0.75},white_vertex/.style={circle,draw,scale=0.75}]
 \node (a) at (4,2) [white_vertex]{};
 \node at (4.3,2.3) [] {$+$};
 \node (b) at (12,2) [white_vertex]{};
 \node at (12.3,2.3) [] {$-$};
 \node (31) at (6,2) [white_vertex]{};
 \node at (6.3,2.3) []{$-$};
 \node (41) at (8,1) [black_vertex]{};
 \node at (8.3,1.3) []{$+$};
 \node (42) at (8,2) [black_vertex]{};
 \node at (8.3,2.3) []{$-$};
 \node (43) at (8,3) [black_vertex]{};
 \node at (8.3,3.3) []{$+$};
 \node (51) at (10,2) [white_vertex]{};
 \node at (10.3,2.3) []{$+$};
 \foreach \from/\to in {41/42,43/42,
 31/41,
 31/43,
 42/51}
 \draw[arrows={-Stealth[scale=1.5]}] (\from)--(\to);
 \foreach \from/\to in {42/31,31/a,b/51}
 \draw[arrows={-Stealth[scale=1.5]}] (\from)--(\to);
 \end{tikzpicture}\\
 \raisebox{8mm}{$\, \xrightarrow{\mu_{+}^{\fillvertex}\mu_{+}^{\empvertex}}\,$} \,\,\,
 &\begin{tikzpicture}
 [scale=0.75,auto=left,black_vertex/.style={circle,draw,fill,scale=0.75},white_vertex/.style={circle,draw,scale=0.75}]
 \node (a) at (4,2) [white_vertex]{};
 \node at (4.3,2.3) [] {$+$};
 \node (b) at (12,2) [white_vertex]{};
 \node at (12.3,2.3) [] {$-$};
 \node (31) at (6,2) [white_vertex]{};
 \node at (6.3,2.3) []{$-$};
 \node (41) at (8,1) [black_vertex]{};
 \node at (8.3,1.3) []{$+$};
 \node (42) at (8,2) [black_vertex]{};
 \node at (8.3,2.3) []{$-$};
 \node (43) at (8,3) [black_vertex]{};
 \node at (8.3,3.3) []{$+$};
 \node (51) at (10,2) [white_vertex]{};
 \node at (10.3,2.3) []{$+$};
 \foreach \from/\to in {42/41,42/43,
 41/31,
 43/31,
 31/42}
 \draw[arrows={-Stealth[scale=1.5]}] (\from)--(\to);
 \foreach \from/\to in {51/42,a/31,51/b}
 \draw[arrows={-Stealth[scale=1.5]}] (\from)--(\to);
 \end{tikzpicture} \\
 \raisebox{8mm}{$\xrightarrow{\mu_{-}^{\fillvertex} }$}\,\,\,
 &\begin{tikzpicture}
 [scale=0.75,auto=left,black_vertex/.style={circle,draw,fill,scale=0.75},white_vertex/.style={circle,draw,scale=0.75}]
 \node (a) at (4,2) [white_vertex]{};
 \node at (4.3,2.3) [] {$+$};
 \node (b) at (12,2) [white_vertex]{};
 \node at (12.3,2.3) [] {$-$};
 \node (31) at (6,2) [white_vertex]{};
 \node at (6.3,2.3) []{$-$};
 \node (41) at (8,1) [black_vertex]{};
 \node at (8.3,1.3) []{$+$};
 \node (42) at (8,2) [black_vertex]{};
 \node at (8.3,2.3) []{$-$};
 \node (43) at (8,3) [black_vertex]{};
 \node at (8.3,3.3) []{$+$};
 \node (51) at (10,2) [white_vertex]{};
 \node at (10.3,2.3) []{$+$};
 \foreach \from/\to in {41/42,43/42,
 51/41,
 51/43,
 42/51}
 \draw[arrows={-Stealth[scale=1.5]}] (\from)--(\to);
 \foreach \from/\to in {42/31,a/31,51/b}
 \draw[arrows={-Stealth[scale=1.5]}] (\from)--(\to);
 \end{tikzpicture}
 \end{align*}
 where the first quiver and the last quiver are identified by the left-right reflection $\nu$.

 The mutation network $\network_\gamma$ is given by
 \begin{align*}
 \begin{tikzpicture}
 [scale=1.7,auto=left,black_vertex/.style={circle,draw,fill,scale=0.75},square_vertex/.style={rectangle,scale=0.8,draw}]
 \node (a) at (0,0) [black_vertex][label=above:{$(1,1)$}]{};
 \node (b) at (1,0) [black_vertex][label=above:{$(2,1)$}]{};
 \node (c1) at (2-0.45,0+0.9) [black_vertex][label=above:{$(3,3)$}]{};
 \node (c2) at (2,0) [black_vertex][label=above:{$(3,2)$}]{};
 \node (c3) at (2+0.45,0-0.9) [black_vertex][label=above:{$(3,1)$}]{};
 \node (ta) at (0,-1) [square_vertex][label=below:{$(1,1)$}]{};
 \node (tb) at (1,-1) [square_vertex][label=below:{$(2,1)$}]{};
 \node (tc1) at (2-0.45,-1+0.9) [square_vertex][label=below:{$(3,3)$}]{};
 \node (tc2) at (2,-1) [square_vertex][label=below:{$(3,2)$}]{};
 \node (tc3) at (2+0.45,-1-0.9) [square_vertex][label=below:{$(3,1)$}]{};
 \draw [dashed] (a)edge [bend left=15](ta);
 \draw [dashed] (ta)edge [bend left=15](a);
 \draw [dashed] (b)edge [bend left=15](tb);
 \draw [dashed] (tb)edge [bend left=15](b);
 \draw [dashed] (c1)edge [bend left=15](tc1);
 \draw [dashed] (tc1)edge [bend left=15](c1);
 \draw [dashed] (c2)edge [bend left=15](tc2);
 \draw [dashed] (tc2)edge [bend left=15](c2);
 \draw [dashed] (c3)edge [bend left=15](tc3);
 \draw [dashed] (tc3)edge [bend left=15](c3);
 \draw[arrows={-Stealth[scale=1.2]}] (b)--(ta);
 \draw[arrows={-Stealth[scale=1.2]}] (a)--(tb);
 \draw[arrows={-Stealth[scale=1.2]}] (c2)--(tb);
 \draw[arrows={-Stealth[scale=1.2]}] (tc2)--(c1);
 \draw[arrows={-Stealth[scale=1.2]}] (tc2)--(c3);
 \draw[arrows={-Stealth[scale=1.2]}] (tc1)--(c2);
 \draw[arrows={-Stealth[scale=1.2]}] (tc3)--(c2);
 \draw[arrows={-Stealth[scale=1.2]}] (b)--(tc1);
 \draw[arrows={-Stealth[scale=1.2] Stealth[scale=1.2] }] (b)--(tc2);
 \draw[arrows={-Stealth[scale=1.2]}] (b)--(tc3);
 \end{tikzpicture}
 \end{align*}
 The set of indices of $\network_\gamma$ is
 \begin{gather}\label{eq:J B3 2}
 \mathbf{J} = \{ (1,1) ,(2,1),(3,1),(3,2) ,(3,3) \}.
 \end{gather}
 The adjacency matrices of $\network_\gamma$ are given by $N_0 = 2 I$ and
 \begin{gather*}
 N_+=
 \begin{pmatrix}
 0 & 0 & 0 & 0 & 0 \\
 0 & 0 & 0 & 0 & 0 \\
 0 & 0 & 0 & 1 & 0 \\
 0 & 0 & 1 & 0 & 1 \\
 0 & 0 & 0 & 1 & 0
 \end{pmatrix} ,\qquad
 N_-=
 \begin{pmatrix}
 0 & 1 & 0 & 0 & 0 \\
 1 & 0 & 1 & 2 & 1 \\
 0 & 0 & 0 & 0 & 0 \\
 0 & 1 & 0 & 0 & 0 \\
 0 & 0 & 0 & 0 & 0
 \end{pmatrix} ,
 \end{gather*}
 so the Neumann--Zagier matrices of $\gamma$ are given by
 \begin{gather*}
 A_+ =
 \begin{pmatrix}
 2 & 0 & 0 & 0 & 0 \\
 0 & 2 & 0 & 0 & 0 \\
 0 & 0 & 2 & -1 & 0 \\
 0 & 0 & -1 & 2 & -1 \\
 0 & 0 & 0 & -1 & 2
 \end{pmatrix},\qquad
 A_-=
 \begin{pmatrix}
 2 & -1 & 0 & 0 & 0 \\
 -1 & 2 & -1 & -2 & -1 \\
 0 & 0 & 2 & 0 & 0 \\
 0 & -1 & 0 & 2 & 0 \\
 0 & 0 & 0 & 0 & 2
 \end{pmatrix} ,
 \end{gather*}
 where we choose an order of $\mathbf{J}$ as in~\eqref{eq:J B3 2}.

 The matrix $A_-$ is not symmetric, but
 \begin{gather*}
 A_+ A_-^{\mathsf{T}} = A_- A_+^{\mathsf{T}} =
 \begin{pmatrix}
 4 & -2 & 0 & 0 & 0 \\
 -2 & 4 & 0 & -2 & 0 \\
 0 & 0 & 4 & -2 & 0 \\
 0 & -2 & -2 & 4 & -2 \\
 0 & 0 & 0 & -2 & 4
 \end{pmatrix}
 \end{gather*}
 is symmetric as promised by Lemma \ref{lemma: NZ symplectic property}.

 The positive definite matrix $A_+^{-1} A_-$ is given by
 \begin{align*}
 A_+^{-1} A_- =
 \frac{1}{2}\begin{pmatrix}
 2 & -1 & 0 & 0 & 0 \\
 -1 & 2 & -1 & -2 & -1 \\
 0 & -1 & 3 & 2 & 1 \\
 0 & -2 & 2 & 4 & 2 \\
 0 & -1 & 1 & 2 & 3
 \end{pmatrix}.
 \end{align*}

 The partition $q$-series are parametrized by elements of the following set:
 \begin{align*}
 H_\gamma = \{ (a,b,c,0,0) \,|\, 0 \leq a,b \leq 1, \, 0 \leq c \leq 3 \},
 \end{align*}
 where $(u,v) + Z_\gamma \in H_\gamma$ is denoted by $\big(u_1^{(1)},u_1^{(2)},u_1^{(3)},u_2^{(3)},u_3^{(3)}\big)$. Let us write $(a,b,c,0,0)$ more simply as $abc$. Then we have
 \begin{gather*}
 H_\gamma = \{ 000,100,010,110, 001,101,011,111, 002,102,012,112, 003,103,013,113 \}.
 \end{gather*}

 We exhibit several low order terms of the partition $q$-series:
 \begin{gather*}
 \mathcal{Z}_\gamma^\sigma (q) = 1 + 9q^{2} + 21q^{3} + 66q^{4} + 144q^{5} + 349q^{6} + 723q^{7} + O\big(q^8\big)
 \end{gather*}
 if $\sigma = 000$,
 \begin{gather*}
 q^{\frac{1}{2}} + 4q^{\frac{3}{2}} + 13q^{\frac{5}{2}} + 38q^{\frac{7}{2}} + 97q^{\frac{9}{2}} + 228q^{\frac{11}{2}} + 504q^{\frac{13}{2}} + 1057q^{\frac{15}{2}} + O\big(q^{\frac{17}{2}}\big)
 \end{gather*}
 if $\sigma = 100,010,110,102,012,112$,
 \begin{gather*}
 q^{\frac{3}{4}} + 5q^{\frac{7}{4}} + 17q^{\frac{11}{4}} + 48q^{\frac{15}{4}} + 120q^{\frac{19}{4}} + 279q^{\frac{23}{4}} + 608q^{\frac{27}{4}} + 1261q^{\frac{31}{4}} + O\big(q^{\frac{35}{4}}\big)
 \end{gather*}
 if $\sigma = 001,011,111,003,013,113$,
 \begin{gather*}
 3 q^{\frac{5}{4}} + 9q^{\frac{9}{4}} + 30q^{\frac{13}{4}} + 75q^{\frac{17}{4}} + 187q^{\frac{21}{4}} + 411q^{\frac{25}{4}} + 885q^{\frac{29}{4}} + 1783q^{\frac{33}{4}} + O\big(q^{\frac{37}{4}}\big)
 \end{gather*}
 if $\sigma = 101,103$, and
 \begin{gather*}
 3 q + 6q^2 + 25q^{3} + 57q^{4} + 156q^{5} + 334q^{6} + 744q^{7} + 1491q^{8} + O\big(q^{9}\big)
 \end{gather*}
 if $\sigma = 002$.
\end{Example}

\section{Relationships with affine Lie algebras}\label{sec:relationships with affine Lie alg}
\subsection{Affine Lie algebras}
In this section, we review basic concepts of affine Lie algebras and their integrable highest weight modules. See~\cite{Kac} for more detail. Let $X_r$ ($=A_r$, $B_r$, $C_r$, $D_r$, $E_{6,7,8}$, $F_4$ or $G_2$) be a finite type Dynkin diagram, and $\mathfrak{g}$ be the finite dimensional simple Lie algebra of type $X_r$ over $\C$. Let $\mathfrak{h}$ be a Cartan subalgebra of~$\mathfrak{g}$, and $\Delta$ be the set of roots. We also use the notations on root systems that we used in Section~\ref{sec:root system}.
We extend the inner product $\innerproduct{\cdot}{\cdot}$ to the nondegenerate symmetric bilinear form on~$\mathfrak{h}^*$. We define the following two free abelian groups:
\begin{gather*}
 Q= \bigoplus_{a=1}^r \Z \alpha_a, \qquad
 M= \bigoplus_{a=1}^r \Z t_a \alpha_a.
\end{gather*}
We also define the following sets:
\begin{gather*}
 P = \left\{ \Lambda \in \mathfrak{h}^* \,\bigg|\, \frac{2 \innerproduct{\Lambda}{\alpha_a}}{ \innerproduct{\alpha_a}{\alpha_a}} \in \Z \text{ for all $a=1 , \dots, r$}\right\}, \\
 P_+ = \left\{ \Lambda \in P \,\bigg |\, \frac{2 \innerproduct{\Lambda}{\alpha_a}}{\innerproduct{\alpha_a}{\alpha_a}} \geq 0 \text{ for all $a=1 , \dots, r$} \right\}.
\end{gather*}

Let
\begin{gather*}
 \hat{\mathfrak{g}} = \mathfrak{g}\otimes \C \big[t, t^{-1}\big] \oplus \C K \oplus \C d
\end{gather*}
be the affine Lie algebra associated with $\mathfrak{g}$. The Lie bracket on $\hat{\mathfrak{g}}$ is defined as
\begin{gather*}
\big[ X \otimes t^m ,Y \otimes t^n \big] = [X,Y] \otimes t^{m+n} + m \delta_{m+n,0} \cdot \innerproduct{X}{Y} K, \\
[K,\hat{\mathfrak{g}}] = \{ 0 \}, \\
\big[d,X \otimes t^n\big] = n X \otimes t^n.
\end{gather*}

Fix a non-negative integer $\level$. Let us define the following set:
\begin{align*}
 P_+^{\level} = \{ \Lambda \in P_+ \,|\, \innerproduct{\Lambda}{\theta} \leq \level \},
\end{align*}
where $\theta$ is the highest root in $\Delta$. For any $\Lambda \in P_+^{\level}$, there is the unique level $\level$ integrable highest weight $\hat{\mathfrak{g}}$-module such that the classical part of its highest weight is $\Lambda$, which is denoted by $L(\Lambda)$.

We define the following rational numbers associated with $L(\Lambda)$:
\begin{gather*}
 c(\level) = \frac{\level \dim \mathfrak{g}}{\level + \dcn},\qquad
 h_\Lambda = \frac{\innerproduct{\Lambda}{\Lambda + 2 \rho}}{2 (\level + \dcn)}.
\end{gather*}

Let $q={\rm e}^{2 \pi{\rm i} \tau}$. For any diagonalizable linear map $\alpha \colon V \to V$ with eigenvalues $\lambda_1,\lambda_2,\dots$ with finite multiplicities $m_1,m_2,\dots$, we define the trace $\mathrm{tr}_V q^{\alpha}$ by $\mathrm{tr}_V q^{\alpha} = \sum_i m_i q^{\lambda_i}$.
We define the following two functions as in~\cite{KacWac}:
\begin{gather}
\chi_{\Lambda} (\tau) = q^{-\frac{c(\level)}{24}} \tr_{L(\Lambda)} q^{h_\Lambda - d},\qquad
 \label{eq:branching}
b_{\lambda}^{\Lambda}(\tau) = q^{-\frac{c(\level) - r}{24}}
 \tr_{U(\Lambda, \lambda)} q^{h_\Lambda - \frac{\innerproduct{\lambda}{\lambda}}{2 \level} -d },
\end{gather}
where $\lambda \in \mathfrak{h}^{*}$ and
\begin{gather*}
 U(\Lambda, \lambda) = \big\{ v \in L(\Lambda) \,|\, \big(h \otimes t^n\big) v = \delta_{n,0} \lambda (h) v \text{ for all $h \in \mathfrak{h}, n \geq 0$} \big\}.
\end{gather*}
These series converge to holomorphic functions on the upper half-plane $\mathbb{H} = \{ \tau \in \C \,|\, \image \tau >0 \}$.
The function $\chi_\Lambda (\tau)$ is called the (specialized) \emph{normalized character}\footnote{More precisely, $\chi_\Lambda(\tau)$ is the normalized character in~\cite{KacWac} at $z=t=0$.} of $L(\Lambda)$, and $b_{\lambda}^{\Lambda}(\tau)$ is called the \emph{branching function} for the pair $(\mathfrak{g},\mathfrak{h})$. We will simply call $b_{\lambda}^{\Lambda}(\tau)$ as the branching function.
Note that dividing this branching function by the $r$-th power of the Dedekind eta function yields the \emph{string function} in~\cite{KacPet}.

The asymptotics of the normalized character and the branching function were studied in detailed in~\cite{KacPet,KacWac}. Let $\tau \downarrow 0$ denote the limit in the positive imaginary axis.
\begin{Theorem}[\cite{KacPet,KacWac}] Suppose that $\lambda \in \Lambda + Q$. Then
 \begin{gather} \label{eq:asym of character}
 \lim_{\tau\downarrow 0} \chi_{\Lambda} (\tau) {\rm e}^{- \frac{\pi{\rm i} c(\level)}{12 \tau}} = a(\Lambda),\\
 \label{eq:asym of branching}
 \lim_{\tau\downarrow 0} b_{\lambda}^{\Lambda} (\tau) {\rm e}^{- \frac{\pi{\rm i} (c(\level) - r)}{12 \tau}} = \lvert P / Q \rvert^{\frac{1}{2}} \lvert Q / \level M \rvert^{-\frac{1}{2}} a(\Lambda),
 \end{gather}
 where $a(\Lambda)$ is the real number defined by
 \begin{gather*}
 a(\Lambda) = \big\lvert P / \big(\level + \dcn\big) M \big\rvert^{-\frac{1}{2}}
 \prod_{\alpha \in \Delta_+} 2 \sin \frac{\pi \innerproduct{\Lambda + \rho}{\alpha}}{\level + \dcn}.
 \end{gather*}
\end{Theorem}

The real number $a(\Lambda)$ is called the \emph{asymptotic dimension} of $L(\Lambda)$ for the following reason.
The module $L(\Lambda)$ is infinite dimensional unless $\Lambda$ is trivial, and $\chi_\Lambda (0) = \dim L(\Lambda) = \infty$.
However, \eqref{eq:asym of character} says that by multiplying by the appropriate exponential term and taking the limit, we can get the finite number $a(\Lambda)$.
Therefore, we can think that the real number $a(\Lambda)$ is the ``dimension'' of $L(\Lambda)$.

\subsection{Exponents and asymptotic dimension}
The branching functions with $\Lambda = 0$ are expected to coincide with the partition $q$-series that we studied in Section \ref{sec:partition q of gamma Xr level} via the conjectural formula of the branching functions in~\cite{KNS}.
This was first observed in~\cite{KatoTerashima} for the total partition $q$-series with $X=ADE$.
Let $b_{\lambda}^{\Lambda}(q)$ denote the formal $q$-series defined by the right-hand side of \eqref{eq:branching}.
\begin{Conjecture}[\cite{KNS}] \label{conj:Z=b}
 Suppose that $\level \geq 2$.
 Let $\gamma = \gamma(X_r, \level)$ be the mutation loop defined in Section~{\rm \ref{section: mutation loop Xr l}}. Suppose that $\sigma \in H_\gamma$, and $\lambda \in Q$ is a representative of $F(\sigma)$, where $F$ is defined as~\eqref{eq: isom H between Q/lM}. Then
 \begin{gather}\label{eq:Z=b}
 q^{-\frac{c(\level) - r}{24}} \mathcal{Z}_\gamma^\sigma (q) = b_{\lambda}^{0}(q).
 \end{gather}
\end{Conjecture}

By summing \eqref{eq:Z=b} for all elements in $Q/ \level M$, we obtain the following conjectural identity on the total partition $q$-series:
\begin{gather}\label{eq:b sum= Z total}
 q^{-\frac{c(\level) - r}{24}} \mathcal{Z}_\gamma (q) = \sum_{\lambda \in Q/\level M} b_{\lambda}^{0}(q).
\end{gather}

By comparing the asymptotics \eqref{eq:asymptotics of Z} and \eqref{eq:asym of branching},
we find that~\eqref{eq:b sum= Z total} yields the following identities:
\begin{gather}\label{eq:dilog identity}
\frac{6}{\pi^2} \sum_{t=1}^T L \big(z_+(\eta; t) \big) = c(\level) - r,
\end{gather}
and
\begin{gather}\label{eq:jacobian identity}
\frac{1}{ \sqrt{I - J_\gamma(\eta)} } = \lvert P / Q \rvert^{\frac{1}{2}} a(0) ,
\end{gather}
where we use $\det A_+ = \lvert Q / \level M \rvert$, which follows from Proposition~\ref{prop: A plus minus}.

The identity \eqref{eq:dilog identity} is called the dilogarithm identity in conformal field theories, and proved in~\cite{IIKKNa,IIKKNb, Nak} by using cluster algebras. On the other hand, the identity~\eqref{eq:jacobian identity} is exactly Conjecture~\ref{conj: char poly} at $x=1$ because
\begin{align*}
 \frac{N_{X_r, \level} (1)}{D_{X_r, \level}(1)} &= \prod_{a=1}^r t_a \big(\level + \dcn\big) { \left( { \prod_{\alpha \in \Delta_+} 2 \sin \frac{ \pi \innerproduct{\rho}{\alpha}} {\level + \dcn} } \right)^{-2 }}\\
 & ={\big\lvert Q/ \big(\level + \dcn\big) M \big\rvert } { \left( { \prod_{\alpha \in \Delta_+} 2 \sin \frac{ \pi \innerproduct{\rho}{\alpha}} {\level + \dcn} } \right)^{-2 }}
 = \lvert P / Q \rvert^{-1} a(0)^{-2}.
\end{align*}
Consequently, the conjectural identity~\eqref{eq:jacobian identity} gives us a consistency between our conjecture on exponents (Conjecture~\ref{conj: char poly}) and the known conjecture on $q$-series (Conjecture~\ref{conj:Z=b}), and also gives an interesting connection between the theory of cluster algebras and the representation theory of affine Lie algebras.

\subsection*{Acknowledgements}
The author gratefully acknowledges the help provided by Yuji Terashima.
This work is supported by JSPS KAKENHI Grant Number JP18J22576.

\LastPageEnding

\end{document}